\documentclass[11pt]{amsart}

\usepackage[utf8]{inputenc}
\usepackage{enumitem} 

\usepackage[T1]{fontenc}

\usepackage[english]{babel}

\usepackage{a4}
\usepackage{amsfonts, amssymb, amsmath}
\numberwithin{equation}{section}
\usepackage{ae, aecompl}

\usepackage{color}

\usepackage{graphicx}

\theoremstyle{plain} 
\usepackage{stmaryrd}

\usepackage{appendix}

\usepackage{scrextend}

\usepackage{mathrsfs}%
\usepackage{csquotes}
\usepackage{tikz}
\usetikzlibrary{matrix}
\usetikzlibrary{cd}

\usepackage{enumitem} 

\usepackage[linkcolor=blue, citecolor=green, colorlinks=true]{hyperref}
\usepackage{pxfonts}
\usepackage{geometry}

\usepackage{minitoc}

\newcommand{\nocontentsline}[3]{}

\newcommand{\tocless}[2]{\bgroup\let\addcontentsline=\nocontentsline#1{#2}\egroup}

\newcommand{\R}{\mathbf{R}}

\newcommand{\C}{\mathbf{C}}

\newcommand{\Z}{\mathbf{Z}}

\newcommand{\Ind}{\mathrm{Ind}}

\renewcommand{\det}{\mathrm{det}}

\newcommand{\Ima}{\text{Im}}

\newcommand{\g}{\mathfrak{g}}

\renewcommand{\dim}{\mathrm{dim}}

\DeclareMathAlphabet{\mathpzc}{OT1}{pzc}{m}{it}

\theoremstyle{plain}

\newtheorem{thm}{Theorem}[section]
 
\newtheorem{prop}{Proposition}[section]

\newtheorem{cor}{Corollary}[thm]
\newtheorem{lemma}{Lemma}

\theoremstyle{definition}

\newtheorem{dfn}{Definition}[section]
\newtheorem{example}{Example}[section]
\theoremstyle{remark}
\newtheorem*{notations*}{Notations}
\newtheorem*{convention}{Convention}
\newtheorem*{rmk}{Remark}
\newcommand{\bw}{\mkern-3mu\mathop{{}_P\mkern-7mu\le}}
\newcommand{\bwg}{\mkern-1mu\mathop{\ge{}\mkern-3mu_P}}
\newcommand{\bwL}{\mkern-3mu\mathop{{}_{P'_L}\mkern-7mu\le}}

\newcommand{\prodscal}[2]{\left\langle#1,#2\right\rangle}

\newdimen\step   \step=13.5pt
\newdimen\auxstep  \newdimen\another  \newdimen\dimaux 

\def\monte#1{\raise\step\hbox{#1}}
\def\descend#1{\raise-\step\hbox{#1}}

\def\montepeuc#1{\raise 3.8pt\hbox{#1}}
\def\montepeuC#1{\raise 5.8pt\hbox{#1}}

\def\descendpeuc#1{\raise -3.8pt\hbox{#1}}
\def\descendpeuC#1{\raise -5.8pt\hbox{#1}}

\def\recule#1{\auxstep=\step \multiply\auxstep by #1 \kern-\auxstep}
\def\avance#1{\auxstep=\step \multiply\auxstep by #1 \kern\auxstep}

\def\Place#1#2#3#4%
{\ifnum#1=0 \auxstep=#3pt \else \auxstep=#4pt\fi%
\another=\auxstep \kern\auxstep \auxstep=\step \advance\auxstep by -\another%
\ifnum#2=0 \another=#3pt \else \another=#4pt\fi%
\advance\auxstep by -\another\relax}

\def\Replace{\kern\auxstep\kern\another\relax}

\def\hbx{\hbox to 0pt}

\def\trait#1#2#3{\hbx{\vrule height #1pt depth #2pt width #3\hss}}

\newbox\Boitecercle                      
\setbox\Boitecercle=%
\hbx{\kern-1.75pt%
\raise-0.1pt\hbox{\kern-0.75pt\hbx{$\circ$\hss}\kern4.25pt}\kern-1.75pt}

\def\cercle{\copy\Boitecercle}

\newbox\Boiterond                        
\setbox\Boiterond=%
\hbx{\kern-1.75pt%
\hbox{\kern-0.75pt\hbx{$\bullet$\hss}\kern4.25pt}\kern-1.75pt}

\newbox\BoiteCercle
\setbox\BoiteCercle=%
\hbx{\kern-3.1389pt%
\hbox{\kern-0.75pt\hbox{$\odot$\hss}\kern-0.75pt}\kern-3.1389pt}

\def\Cercle{\hbox{\copy\BoiteCercle\copy\Boiterond}}

\def\Dessinetv#1#2%
{\auxstep=2.2222pt \dimaux=\step \advance\dimaux by\auxstep%
\ifnum#1=0 \another=1.75pt \else\another=3.1389pt\fi%
\advance\auxstep by\another \advance\dimaux by -\auxstep%
\ifnum#2=0 \another=1.75pt \else \another=3.1389pt\fi%
\advance\dimaux by-\another \kern-0.125pt%
\raise\auxstep\hbox{\vrule height\dimaux depth 0pt width 0.25pt}%
\kern-0.125pt}

\def\Dessineth#1#2%
{\Place{#1}{#2}{1.75}{3.1389}%
\raise 2.2222pt\trait{0.125}{0.125}{\auxstep}\Replace}

\def\thcc{{\Dessineth00}}    \def\thCc{{\Dessineth10}}
\def\thcC{{\Dessineth01}}    \def\thCC{{\Dessineth11}}

\def\DessineTh#1#2%
{\Place{#1}{#2}{1.75}{3.1389}%
\raise 2.2222pt\trait{0.35}{0.35}{\auxstep}\Replace}

\def\Dessinethh#1#2%
{\Place{#1}{#2}{1.55}{3.00}%
\raise3.2222pt\trait{0.125}{0.125}{\auxstep}%
\raise1.2222pt\trait{0.125}{0.125}{\auxstep}\Replace}

\def\Dessinethhh#1#2%
{\Place{#1}{#2}{1.25}{2.8}%
\raise 3.4444pt\trait{0.125}{0.125}{\auxstep}%
\raise 1.0000pt\trait{0.125}{0.125}{\auxstep}\Replace}

\def\thhhcc{\hbox{\Dessinethhh00\kern-\step\thcc}}
\def\thhhcC{{\Dessinethhh01\kern-\step\thcC}}
\def\thhhCc{{\Dessinethhh10\kern-\step\thCc}}
\def\thhhCC{{\Dessinethhh11\kern-\step\thCC}}

\def\Dessinethhhh#1#2%
{\Place{#1}{#2}{1.25}{2.9}%
\raise 3.4222pt\trait{0.125}{0.125}{\auxstep}%
\raise 1.0222pt\trait{0.125}{0.125}{\auxstep}%
\Replace\kern-\step%
\Place{#1}{#2}{1.5}{3.0}%
\raise 2.6222pt\trait{0.125}{0.125}{\auxstep}%
\raise 1.8222pt\trait{0.125}{0.125}{\auxstep}%
\Replace}

\def\Dessinepoints#1#2%
{\Place{#1}{#2}{1.75}{3.1389}%
\divide\auxstep by 4%
\kern\auxstep\kern-1.1pt\hbx{$\cdot$\hss}\kern 1.1pt%
\kern\auxstep\kern-1.1pt\hbx{$\cdot$\hss}\kern 1.1pt%
\kern\auxstep\kern-1.1pt\hbx{$\cdot$\hss}\kern 1.1pt%
\kern\auxstep\kern\another}

\def\Dessinef#1#2#3%
{\Place{#2}{#1}{1.75}{3.1389}\Replace\kern-\step%
\dimaux= 7pt \advance\auxstep by-\dimaux \divide\auxstep by 2%
\advance\auxstep by \another \advance\auxstep by \dimaux%
\raise 0.68pt\hbx{\kern-\auxstep$\scriptstyle #3$\hss}}

\def\Dessineff#1#2#3%
{\hbx{\kern-1pt\Dessinef{#1}{#2}{#3}\kern2pt\Dessinef{#1}{#2}{#3}\hss}}

\def\Dessinefff#1#2#3%
{\hbx{\kern-1.5pt\Dessinef{#1}{#2}{#3}%
\kern1.5pt\Dessinef{#1}{#2}{#3}\kern1.5pt\Dessinef{#1}{#2}{#3}\hss}}

\def\Dessinefourche#1#2#3
{\hbox{\kern #1pt%
\hbx{$\displaystyle <$\hss}\kern8pt%
\ifnum#2=1{\descendpeuC\Cercle}\else{\descendpeuc\cercle}\fi%
\ifnum#3=1{\montepeuC\Cercle}\else{\montepeuc\cercle}\fi}}

                   \let\th=\thcc
\let\thh=\thhcc                 \let\thhh=\thhhcc
             
\let\points=\pointscc
\let\fourche=\cfourchecc

\let\descendpeu=\descendpeuc    \let\montepeu=\montepeuc
\let\fgauche=\fgauchecc         

\let\fdroite=\fdroitecc

\newbox\aux

\def\gnoteN#1#2%
{\setbox\aux=\hbox{$\scriptscriptstyle #1$}%
\dimaux=\wd\aux \divide\dimaux by 2%
\kern-\dimaux\raise #2pt\hbx{$\scriptscriptstyle #1$\hss}\kern\dimaux}

\def\noteN#1{\gnoteN{#1}{5}}
\def\noteNC#1{\gnoteN{#1}{6.5}}
\def\noteS#1{\gnoteN{#1}{-3.6}}

\def\gnoteE#1#2#3%
{\kern #2pt\raise #3pt\hbx{$\scriptscriptstyle #1$\hss}\kern-#2pt}

\def\noteE#1{\gnoteE{#1}{2.5}{1}}

\def\gnoteO#1#2#3%
{\setbox\aux=\hbox{$\scriptscriptstyle #1$}%
\kern-#2pt\kern-\wd\aux\raise #3pt\hbx{$\scriptscriptstyle #1$\hss}%
\kern #2pt\kern\wd\aux}

\newbox\Boitedisque                      
\setbox\Boitedisque=%
\hbx{\kern-1.75pt%
\raise-0.1pt\hbox{\kern-0.75pt\hbx{$\bullet$\hss}\kern4.25pt}\kern-1.75pt}
\def\disque{\copy\Boitedisque}

\def\DessinefourcheD#1#2#3
{\hbox{\kern #1pt%
\hbx{$\displaystyle <$\hss}\kern8pt%
\ifnum#2=1{\descendpeuC\disque}\else{\descendpeuc\disque}\fi%
\ifnum#3=1{\montepeuC\disque}\else{\montepeuc\disque}\fi}}

\def\Dessinefourched#1#2#3
{\hbox{\kern #1pt%
\hbx{$\displaystyle <$\hss}\kern8pt%
\ifnum#2=1{\descendpeuC\disque}\else{\descendpeuc\disque}\fi%
\ifnum#3=1{\montepeuC\cercle}\else{\montepeuc\cercle}\fi}}

\title[The generalized Injectivity conjecture]{The generalized Injectivity conjecture}
\author[S.~Dijols]{Sarah Dijols}
\address{312 Jingzhai, Tsinghua University, Qinghuqyuan street, Haidian District, Beijing, China}
\email{sarah.dijols@hotmail.fr}

\begin{document}

\begin{abstract}
We prove a conjecture of Casselman and Shahidi stating that the unique irreducible generic subquotient of a standard module is necessarily a subrepresentation for a large class of connected quasi-split reductive groups, in particular for those which have a root system of classical type (or product of such groups). To do so, we prove and use the existence of strategic embeddings for irreducible generic discrete series representations, extending some results of Moeglin.
\end{abstract}


\maketitle
\setcounter{tocdepth}{1}
\setlength{\parskip}{0em}


\section{Introduction} \label{intro1}

\subsection{}
Let $G$ be a quasi-split connected reductive group over a non-Archimedean local field $F$ of characteristic zero. We assume we are given a 
standard parabolic subgroup $P$ with Levi decomposition $P=MU$ as well as an irreducible, tempered, generic representation $\tau$ of $M$. 
Let now $\nu$ be an element in the 
dual of the real Lie algebra of the split component of $M$; we take it in the positive Weyl chamber. The induced 
representation 
$I_P^G(\tau,\nu):= I_P^G(\tau_{\nu})$, called the standard module, has a unique irreducible quotient, $J(\tau_{\nu})$, often named the 
Langlands quotient. 
Since the representation $\tau$ is generic (for a non-degenerate character of $U$, see the Section \ref{preliminaries}), i.e. has a Whittaker 
model, the 
standard module $I_P^G(\tau_{\nu})$ is also generic. Further, by a result of Rodier \cite{rodier2} any generic induced module has a unique irreducible 
generic 
subquotient.

In their paper Casselman and Shahidi \cite{sha} conjectured that:
\begin{enumerate}[label=(\Alph*)]
\item $J(\tau_{\nu})$ is generic if and only if $I_P^G(\tau_{\nu})$ is irreducible.
\item The unique irreducible generic subquotient of $I_P^G(\tau_{\nu})$ is a subrepresentation.
\end{enumerate}
These questions were originally formulated for real groups by Vogan \cite{vogan}.
Conjecture (B), was resolved in \cite{sha} provided the inducing data is cuspidal.
Conjecture (A), known as the Standard Module Conjecture, was first proven for classical groups by Mui\'c in \cite{SMCmuic}, and was settled for quasi-split p-adic groups in \cite{smc} assuming the Tempered L Function Conjecture 
proven a few years later in \cite{opdamh}.

The second conjecture, known as the Generalized Injectivity Conjecture was proved for classical groups $SO(2n+1), Sp(2n)$, and $SO(2n)$ 
for $P$ a 
maximal parabolic subgroup, by Hanzer in \cite{gic1}. 

In the present work we prove the Generalized Injectivity Conjecture (Conjecture (B)) for a large class of quasi-split connected reductive groups provided the irreducible components of a certain root system (denoted $\Sigma_{\sigma}$) are of type $A,B,C$ or $D$ (see Theorem \ref{mainresult} below for a precise statement).
Following the terminology of Borel-Wallach [4.10 in \cite{borelwallach}], for a standard parabolic subgroup $P$, $\tau$ a tempered representation and $\eta \in (a_M^*)^+$, a positive Weyl chamber, $(P,\tau, \eta)$ is referred as Langlands data, and $\eta$ is the Langlands parameter, see the Definition \ref{Langlandsdata} in this 
manuscript.

We will study the unique irreducible generic subquotient of a standard module $I_P^G(\tau_{\eta})$ and make \emph{first} the following reductions:
\begin{itemize}
\item $\tau$ is discrete series representation of the standard Levi subgroup $M$
\item $P$ is a maximal parabolic subgroup. 

Then, $\eta$ is written $s\tilde{\alpha}$, see the Subsection \ref{max} for a definition of the latter.
\end{itemize}

Then, our approach has two layers: 
First, we realize the generic discrete series $\tau$ as a subrepresentation of an induced module $I_{P_1\cap M}^{M}(\sigma_{\nu})$ for a unitary generic cuspidal representation of $M_1$ (using Proposition 2.5 of \cite{opdamh}), and the parameter $\nu$ is dominant (i.e in some positive 
closed Weyl 
chamber) in a sense later made precise; Using induction in stages, we can therefore embed the standard module $I_P^G(\tau_{s
\tilde{\alpha}})$ in $I_{P_1}
^G(\sigma_{\nu + s\tilde{\alpha}})$.

Let us denote $\nu + s\tilde{\alpha}:= \lambda$. The unique generic subquotient of the standard module is also the unique generic subquotient 
in $I_{P_1}
^G(\sigma_{\lambda})$. By a result of Heiermann-Opdam [Proposition 2.5 of \cite{opdamh}], this generic subquotient appears as a 
subrepresentation of yet 
another induced representation $I_{P'}^G(\sigma'_{\lambda'})$ characterized by a parameter $\lambda'$ in the closure of some positive Weyl 
chamber. 

In an ideal scenario, $\lambda$ and $\lambda'$ are dominant with respect to $P_1$ (resp. $P'$), i.e. $\lambda$ and $\lambda'$ are in the 
closed positive 
Weyl chamber, and we may then build a bijective operator between those two induced representations using the dominance property of the 
Langlands 
parameters.

In case the parameter $\lambda$ is not in the closure of the positive Weyl chamber, two alternatives procedures are considered: first, another 
strategic 
embedding of the irreducible generic subquotient in the representation induced from $\sigma''_{\lambda''}$ (relying on extended Moeglin's Lemmas) 
when the 
parameter $\lambda''$ (which depends on the form of $\lambda$) has a very specific aspect (this is Proposition \ref{embedding}); or (resp. 
and) showing the 
intertwining operator between $I_{P'}^G(\sigma'_{\lambda'})$ (resp. $I_{P_1}^G(\sigma''_{\lambda''})$) and $I_{P_1}^G(\sigma_{\lambda})$ 
has non-generic 
kernel.

\subsection{} \label{respoints}

In order to study a larger framework than the one of classical groups studied in \cite{gic1}, we will use the notion of \emph{residual points} of the $\mu$ function (the $\mu$ function is the main ingredient of the Plancherel density for p-adic groups (see the Definition \ref{residualpoint} and Subsection \ref{mu}).

Indeed, as briefly suggested in the previous point, the triple $(P_1, \sigma, \lambda)$, introduced above, plays a pivotal role in all the arguments developed thereafter, and of particular importance, the parameter $\lambda$ is related to the $\mu$ function in the following ways:
\begin{itemize}
\item  When $\sigma_{\lambda}$ is a residual point for the $\mu$ function (abusively one says that $\lambda$ is a residual point once the context is clear), the unique irreducible generic subquotient in the module induced from $\sigma_{\lambda}$ is discrete series (a result of Heiermann in \cite{heiermann}, see Proposition \ref{heir}).

\item Once the cuspidal representation $\sigma$ is fixed, we attach to it the set $\Sigma_{\sigma}$, a root system in a  subspace of $a_{M_1}^*$ defined using the $\mu$ function.
More precisely, let $\alpha$ be a root in the set of reduced roots of $A_{M_1}$ in Lie($G$) and $(M_1)_{\alpha}$ be the centralizer of $(A_{M_1})_{\alpha}$ (the identity component of the kernel of $\alpha$ in $A_{M_1}$). We will consider the set 

$$\Sigma_{\sigma}=\{\alpha\in \Sigma_{red}(A_{M_1})\vert \mu^{(M_1)_{\alpha}}(\sigma )=0\}$$
It is a subset of $a_{M_1}^*$ which is a root system in a subspace of  $a_{M_1}^*$(cf \cite{silbergerSD} 3.5) and we suppose the irreducible components of $\Sigma_{\sigma}$ are of type $A,B,C$ or $D$.
Let us denote $W_{\sigma}$ the Weyl group of $\Sigma_{\sigma}$.

This is where stands the particularity of our method, to deal with all possible standard modules, we needed an explicit description of this parameter $\lambda$ lying in $a_{M_1}^*$. Thanks to Opdam's work in the context of affine Hecke algebras and Heiermann's one in the context of p-adic reductive groups such descriptive approach is made possible. Indeed, we have a bijective correspondence between the following sets explained in Section \ref{balacarter}:
$\left\{\text{dominant residual point}\right\} \leftrightarrow \left\{\text{Weighted Dynkin diagram(s)}\right\}$ \par
\end{itemize}
The notion of Weighted Dynkin diagram is established and recalled in the Appendix \ref{WDD1}. We use this correspondence to express the coordinates of the dominant residual point and name this expression of the residual point a \emph{residual segment} generalizing  the classical notion of segments (of Bernstein-Zelevinsky). We associate to such a residual segment \emph{set(s) of Jumps} (a notion connected to that of Jordan blocks elements in the classical groups setting of Moeglin-Tadi\'c in \cite{MT}).

Further, the $\mu$ function is intrinsically related to the \emph{intertwining operators} mentioned in the previous subsection: A key aspect of this work is an appropriate use of (standard) intertwining operators, more precisely the use of intertwining operators with non-generic kernel. Using the functoriality of induction, it is always possible to reduce the study of intertwining operators to \emph{rank one} intertwining operators (i.e consider the well-understood intertwining operator $J_{s_{\alpha_i} P_1|P_1}$ between $I_{P_1\cap (M_1)_{\alpha_i}}^{M_1}(\sigma_{\lambda})$ and $I_{\overline{P_1}\cap (M_1)_{\alpha_i}}^{M_1}(\sigma_{\lambda})$); and in particular if $\sigma$ is irreducible cuspidal (see Theorem \ref{hc}). At the level of rank one intertwining operator (where $I_{P_1\cap (M_1)_{\alpha_i}}^{M_1}(\sigma_{\lambda})$ is the direct sum of two non-isomorphic representations, see Theorem \ref{hc}), determining the non-genericity of the kernel of the map $J_{s_{\alpha_i} P_1|P_1}$ reduces to a simple condition on the relevant coordinates (i.e the coordinates determined by $\alpha_i$) of $\lambda \in a_{M_1}^*$.

\subsection{} \label{sub}
Having defined the root system $\Sigma_{\sigma}$, let us present the main result of this paper:

\begin{thm}[Generalized Injectivity conjecture for quasi-split group] \label{mainresult}
Let $G$ be a quasi-split, connected group defined over a p-adic field $F$ (of characteristic zero) such that its root system is of type $A,B,C$ or $D$ (or product of these).
Let $\pi_0$ be the unique irreducible generic subquotient of the standard module $I_P^G(\tau_{\nu})$, then $\pi_0$ embeds as \emph{a 
subrepresentation} in the 
standard module $I_P^G(\tau_{\nu})$.
\end{thm}

\begin{thm}[Generalized Injectivity conjecture for quasi-split group] \label{mainresult2}
Let $G$ be a quasi-split, connected group defined over a p-adic field $F$ (of characteristic zero). 
Let $\pi_0$ be the unique irreducible generic subquotient of the standard module $I_P^G(\tau_{\nu})$, let $\sigma$ be an irreducible, generic, cuspidal representation of $M_1$ such that a twist by an unramified real character of $\sigma$ is in the cuspidal support of $\pi_0$.

Suppose that all the irreducible components of $\Sigma_{\sigma}$ are of type $A,B,C$ or $D$, then, under certain conditions on the Weyl group of $\Sigma_{\sigma}$ (explained in Section \ref{onsomeconditions}, in particular Corollary \ref{mmm}), $\pi_0$ embeds as \emph{a 
subrepresentation} in the 
standard module $I_P^G(\tau_{\nu})$.
\end{thm}

Theorem \ref{mainresult} results from \ref{mainresult2}. The Theorem \ref{mainresult2} is true when the root system of the group $G$ contains components of type $E, F$ provided $\Sigma_{\sigma}$ is irreducible of type $A$. 
We do not know if an analogue of Corollary \ref{mmm} hold for groups whose root systems are of type $E$ or $F$. Further, in the exceptional groups of type $E$ or $F$, many cases where the cuspidal support of $\pi_0$ is $(P_0, \sigma)$ (generalized principal series) cannot be dealt with the methods proposed in this work, see Section \ref{excp} for details. 

\subsection{}
Let us briefly comment on the organisation of this manuscript, therefore giving a general overview of our results and the scheme of proof. 

In Section \ref{io1}, we formulate the problem in an as broad as possible context (any quasi-split reductive p-adic group $G$) and prove a few results on intertwining operators.

As M.Hanzer in \cite{gic1}, we distinguish two cases: the case of a generic 
discrete series subquotient, and the case of a non-discrete series generic subquotient. As stated in \ref{respoints}, the case of discrete series subquotient corresponds to $\sigma_{\lambda}$ (in the cuspidal support of the generic discrete series) being a residual point.  

As just stated in \ref{respoints}, our approach uses the bijection between Weyl group orbits of residual points and weighted Dynkin diagrams as studied in \cite{opdamspec} and explained in the Appendix \ref{BCT}. 

Through this approach, we can make explicit the Langlands  parameters of subquotients of the representations $I_{P_1}^G(\sigma_{\lambda})$ induced from the generic cuspidal support $\sigma_{\lambda}$  and classify them using the order on parameters in $a_{M_1}^*$ as given in 
Chapter XI, 
Lemma 2.13 in \cite{borelwallach}. In particular, the minimal element for this order (in a sense later made 
precise) 
characterizes the unique irreducible generic non-discrete series subquotient, see Theorem \ref{minLP=irr}.

Although requiring us to get acquainted with the notions of residual points, and then residual segments, our methods have two advantages.

The first is proving the Generalized Injectivity Conjecture for a large class of quasi-split reductive groups (provided a certain construction of the standard Levi subgroup $M_1$ and the 
irreducible components of $\Sigma_{\sigma}$ to be of type $A,B,C$ or $D$; we have verified those conditions when the root system of the quasi-split (hence reductive) group is of type $A,B,C$ or $D$), and recovering the results of Hanzer through alternative proofs. 
\noindent
In particular, a key ingredient (which was not used by Hanzer in \cite{gic1}) in our method is an embedding result of Heiermann-Opdam (Proposition \ref{opdamh}). The second is a self-contained and uniform (in the sense that cases of root systems of type $B,C$ and $D$ are all treated in the same proofs) treatment.

Although based on the ideas of Hanzer in \cite{gic1}, our approach includes a much larger class of quasi-split groups and some cases of exceptional groups.



%




\emph{We separate this work into two different problems}. The first problem is determining the conditions on $\lambda \in a_{M_1}^*$ so that the unique generic subquotient of $I_{P_1}^G(\sigma_{\lambda})$ with $\sigma$ irreducible unitary generic cuspidal representation of a standard Levi $M_1$ is a subrepresentation. The results on this problem are presented in Theorem \ref{conditionsonlambda}.

The second problem is to show that any standard module can be embedded in a module induced from cuspidal generic data, with $\lambda \in a_{M_1}^*$ satisfying one of the conditions mentioned in Theorem \ref{conditionsonlambda}. This is done in the Section \ref{gicds} and the following.

Regarding the first problem: in the Subsection \ref{Moeglinandembedding}, we present an embedding result for the unique irreducible generic discrete series subquotient of the generic standard module (see Proposition \ref{embedding}) relying on two extended Moeglin's Lemmas (see Lemmas \ref{3.2} and \ref{moeglinlemmaextended}) and the result of Heiermann-Opdam (see Proposition \ref{opdamh}). This embedding and the use of standard intertwining operators with non-generic kernel allow us to prove the Theorem \ref{conditionsonlambda}.

Once achieved the Theorem \ref{conditionsonlambda}, it is rather straightforward to prove the Generalized Injectivity Conjecture for discrete series generic subquotient, first when $P$ is a maximal parabolic subgroup and secondly for \emph{any parabolic subgroup} in Section \ref{nonmaximal}.

In Subsection \ref{gicnnds}, we continue with the case of a generic non-discrete series subquotient, and further conclude with the case of the standard module induced from a tempered representation $\tau$ in Corollary \ref{standard} and Corollary \ref{sigmasigmareducible2}.

The proof of Theorem \ref{mainresult2} is done in several steps. First, we prove it for the case of an irreducible generic discrete series subquotient assuming $\tau$ discrete series, and $\Sigma_{\sigma}$ irreducible in Proposition \ref{dsr=1}.

We use this latter result for the case of a non-square integrable irreducible generic subquotient in Proposition \ref{nndsgic1}; and also for the case of standard modules induced from non-maximal standard parabolic (Theorems \ref{dsgic2} and \ref{nndsgic2}). Then, the case of  $\tau$ tempered follows (Corollary \ref{standard}). 
The case of $\Sigma_{\sigma}$ reducible is done in Section \ref{sigmared} and relies on the Appendix \ref{lab}.

The reader familiar with the work of Bernstein-Zelevinsky on $GL_n$ (see \cite{rodierb} or \cite{Z}) may want to have a look at the author PhD thesis where we treat independently the case of $\Sigma_{\sigma}$ of type $A$ to get a quicker overview on some tools used in this work.

From here, we use the following notations:
\begin{notations*}\label{max}
\begin{itemize} \item \emph{Standard module induced from a maximal parabolic subgroup:} \\
Let $\Theta = \Delta - \left\{\alpha\right\}$ for $\alpha$ in $\Delta$, and let $P= P_{\Theta}$ be a maximal parabolic subgroup of $G$. We denote $\rho_P$ the half sum  of positive roots in $U$,and for $\alpha$ the unique simple root for $G$ which is not a root for $M$, 
$$\tilde{\alpha} = \frac{\rho_P}{(\rho_P,\alpha)} ~~ \hbox{where} ~~ (\rho_P,\alpha)= \frac{2\prodscal{\rho_P}{\alpha}}{\prodscal{\alpha}{\alpha}}$$

(Rather than $\tilde{\alpha}$, in the split case, we could also take the fundamental weight corresponding to $\alpha$). Since $\nu$ is in $a_{M}^*$ (of dimension rank($G$) - rank$(M)$= 1 since $M$ is maximal), and should satisfy $\prodscal{\nu}{\check{\beta}} > 0$ for all $\beta \in \Delta - \Theta = \left\{\alpha
\right\}$, the standard module in this case is $I_P^G(\tau_{s\tilde{\alpha}})$ where $s \in \mathbb{R}$ such that $s > 0$, and $\tau$ is an irreducible tempered representation of $M$.

	\item For the sake of readability we sometimes denote $I_{P_1}^G(\sigma(\lambda)):= I_{P_1}^G(\sigma_{\lambda})$ when the parameter $\lambda$ is expressed in terms of residual segments.
	\item Let $\sigma$ be an irreducible cuspidal representation of a Levi subgroup $M_1 \subset M$ in a standard parabolic subgroup $P_1$, and let $
\lambda$ be in $
(a_{M_1}^*)$, we denote $Z^M(P_1, \sigma, \lambda)$ the unique irreducible generic discrete series (resp. essentially square-integrable) 
in the standard 
module $I_{P_1\cap M}^M(\sigma_{\lambda})$. 

We will omit the index when the representation is a representation of $G$: $Z(P_1, \sigma, \lambda)$; often $\lambda$ will be written explicitly with residual segments to emphasize the dependency on specific sequences of exponents.
\end{itemize}
\end{notations*}

\renewcommand{\abstractname}{Acknowledgements}
\begin{abstract}
This work is part of the author's PhD thesis under the supervision of Volker Heiermann, at Aix-Marseille University. The author has benefited from a grant of Agence Nationale de la Recherche with reference ANR-13-BS01-0012 FERPLAY.
We are very grateful to Patrick Delorme for a careful reading and detailed comments on various part of this work.
We also thank Dan Ciabotaru, Jean-Pierre Labesse, Omer Offen, François Rodier, Allan Silberger, and Marko Tadi\'c for interesting suggestions and discussions. The author would like to express her gratitude to the anonymous referee for her/his very careful reading which helped improve the clarity of exposition. 
\end{abstract}

\section{Preliminaries}\label{preliminaries}

\subsection{Basic objects}

Throughout this paper we will let $F$ be a non-Archimedean local field of characteristic 0. We will denote by $G$ the group of $F$-rational points of a quasi-split connected reductive group defined over $F$.
We fix a minimal parabolic subgroup $P_0$ (which is a Borel $B$ since $G$ is quasi-split) with Levi decomposition $P_0 = M_0 U_0$ and $A_0$ a maximal split torus (over $F$) of $M_0$. $P$ is said to be standard if it contains $P_0$. More generally, if $P$ rather contains $A_0$, it is said to be semi-standard. Then $P$ contains a unique Levi subgroup $M$ containing $A_0$, and $M$ is said to be semi-standard. For a semi-standard Levi subgroup $M$, we denote $\mathcal{P}(M)$ the set of parabolic subgroups $P$ with Levi factor $M$. 

We denote by $A_M$ the maximal split torus in the center of $M$, $W= W^G$ the Weyl group of $G$ defined with respect to $A_0$ (i.e. 
$N_G(A_0)\slash 
Z_G(A_0)$). The choice of $P_0$ determines an order in $W$, and we denote by $w_0^G$ the longest element in $W$.

If $\Sigma$ denote the set of roots of $G$ with respect to $A_0$, the choice of $P_0$ also determines the set of positive roots (resp., negative 
roots, simple 
roots) which we denote by $\Sigma^+$ (resp., $\Sigma^-$, $\Delta$). 
\\

To a subset $\Theta \subset \Delta$ we associate a standard parabolic subgroup $P_{\Theta}= P$ with Levi decomposition $MU$, and denote $A_M$ the split component of $M$. 
We will write $a_M^*$ for the dual of the real Lie-algebra $a_M$ of $A_M$, $(a_M)_{\mathbb{C}}^*$ for its complexification and $a_M^{*+}$ for 
the positive Weyl chamber in $a_M^*$ defined with respect to $P$. Further $\Sigma(A_M)$ denotes the set of roots of $A_M$ in Lie($G$). It is a subset of $a_M^*$. For any root $\alpha \in \Sigma(A_M)$, we 
can associate a coroot $\check{\alpha} \in a_M$. For $P \in \mathcal{P}(M)$, we denote $\Sigma(P)$ the subset of positive roots of $A_M$ relative to $P$.

\vspace{0,3cm}

Let $\mbox{Rat}(M)$ be the group of $F$-rational characters of $M$, we have:
$$a_M^*= \mbox{Rat}(M) \otimes_{\mathbb{Z}} \mathbb{R} \ \mbox{and} \ (a_M)_{\mathbb{C}}^* = a_M^* \otimes_{\mathbb{R}} \mathbb{C}$$

For $\chi\otimes r \in a_M^*$, $r \in \mathbb{R}$, and $\lambda$ in $a_M$, the pairing $a_M \times a_M^* \rightarrow \mathbb{R}$ is given by: $
\prodscal{\lambda}{\chi\otimes r } = \lambda(\chi).r$

Following \cite{wald} we define a map $$H_M: M \rightarrow a_M= \mbox{Hom}(\mbox{Rat}(M), \mathbb{R})$$ such that 
$$|\chi(m)|_F = q^{-\prodscal{\chi}{H_M(m)}}$$
 for every $F$-rational character $\chi$ in $a_M^*$ of $M$, $q$ being the cardinality of the residue field of $F$. 
Then $H_P$ is the extension of this homomorphism to $P$, extended trivially along $U$.

We denote by $X(M)$ the group of unramified characters of $M$. 

\vspace{0,5cm}
Let us assume that $(\sigma, V)$ is an admissible complex representation of $M$.
We adopt the convention that the isomorphism class of $(\sigma, V)$ is denoted by $\sigma$. If $\chi_{\nu}$ is in $X(G)$, with $\nu \in a_{G,\C}^*$, then we write $(\sigma_{\nu}, V_{\chi_{\nu}})$ for the representation $\sigma\otimes\chi_{\nu}$ on the space $V$.

\vspace{0,2cm}
Let $(\sigma, V)$ be an admissible representation of finite length of $M$, a Levi subgroup containing $M_0$ a minimal Levi subgroup, centralizer of the maximal split torus $A_0$. Let $P$ and $P'$ be in $\mathcal{P}(M)$.
Consider the intertwining integral: 
$$(J_{P'|P}(\sigma_{\nu})f)(g) = \int_{U\cap U'\backslash U'} f(u'g)du' \quad f \in I_P^G(\sigma_{\nu}) $$
where $U$ and $U'$ denote the unipotent radical of $P$ and $P'$, respectively.

For $\nu$ in $X(M)$ with Re($\prodscal{\nu}{\check{\alpha}}) > 0$ for all $\alpha$ in $\Sigma(P)\cap\Sigma(P')$ the defining integral of $J_{P'|P}(\sigma_{\nu})$ converges absolutely. Moreover, $J_{P'|P}$ defined in this way on some open subset of $\mathcal{O} = \left\{ \sigma_{\nu} | \nu \in X(M) \right\}$ becomes a rational function on $\mathcal{O}$ (\cite{wald} Theorem IV 1.1). Outsides its poles, this defines an element of $$\mbox{Hom}_G(I_P^G(V_{\chi}),I_{P'}^G(V_{\chi}))$$
Moreover, for any $\chi$ in $X(M)$, there exists an element $v$ in 
$I_P^G(V_{\chi})$ such that $J_{P'|P}(\sigma_{\chi})v$ is not zero (\cite{wald}, IV.1 (10))

In particular, for all $\nu$ in an open subset of $a_M^*$, and $\overline{P}$ the opposite parabolic subgroup to $P$, we have an intertwining operator $$J_{\overline{P}|P}(\sigma_{\nu}) : I_P^G(\sigma_{\nu}) \rightarrow I_{\overline{P}}^G(\sigma_{\nu})$$ and for $\nu$ in $(a_M^*)^+$ far away from the walls it is defined by the convergent integral: 
 $$(J_{\overline{P}|P}(\sigma_{\nu})f)(g) = \int_{\overline{U}} f(ug)du $$
The intertwining operator is meromorphic in $\nu$ and the map $J_{\overline{P}|P}J_{P|\overline{P}}$ is a scalar. Its inverse equals the Harish-Chandra $\mu$ function up to a constant and will be denoted $\mu^G(\sigma_{\nu})$.

\begin{convention}\label{convrodier}
By \cite{shabook} Sections 3.3 and 1.4, we can fix a non-degenerate character $\psi$ of $U$ which, for every Levi subgroup $M$, is compatible with $w_0^Gw_0^M$. We will still denote $\psi$ the restriction of $\psi$ to $M\cap U$. Every generic representation $\pi$ of $M$ becomes generic with respect to $\psi$ after changing the splitting in $U$. Throughout this paper, generic means $\psi$-generic. When the groups are quasi-split and connected, by a theorem of Rodier, the standard $\psi$-generic modules have exactly one  $\psi$-generic irreducible subquotient. This unicity will be used in numerous proofs: we will use the name [U] to refer to this result.
\end{convention}

\subsection{The $\mu$ function} \label{mu}
Harish-Chandra's $\mu$-function is the main ingredient of the Plancherel density
for a p-adic reductive group $G$ \cite{wald}. It assigns to every discrete series representation
of a Levi subgroup a complex number and can be analytically extended to
a meromorphic function on the space of essentially square-integrable representations
of Levi subgroups.

Let $Q = NV$ be a parabolic subgroup of a connected reductive group $G$ over $F$ and $\sigma$ 
an irreducible unitary cuspidal representation of $N$, then the Harish-Chandra's $\mu$-function $\mu^G$ corresponding to $G$ defines a meromorphic function 
$a_{N,\mathbb{C}}^* \rightarrow \mathbb{C}$, $\lambda \rightarrow \mu^G(\sigma_{\lambda})$ (cf. \cite{heiermann}, Proposition 4.1, \cite{silbergerannals}, 1.6) which (in a certain context, see Proposition 4.1 in \cite{heiermann}) can be written:
$$\mu^G(\sigma_{\lambda}) = f(\lambda)\prod_{\alpha \in \Sigma(Q)}
\frac{(1- q^{\prodscal{\check{\alpha}}{\lambda}}) (1- q^{-\prodscal{\check{\alpha}}{ \lambda}} ) }
{(1- q^{\epsilon_{\alpha}+\prodscal{\check{\alpha}}{ \lambda}})
(1- q^{\epsilon_{\alpha}-\prodscal{\check{\alpha}}{\lambda}})}$$  
where $f$ is a meromorphic function without poles and zeroes on $a_N^*$ and the $\epsilon_{\alpha}$ are non-negative rational numbers such that 
$\epsilon_{\alpha}= \epsilon_{\alpha'}$ if $\alpha$ and $\alpha'$ are conjugate.
We refer the reader to Sections IV.3 and V.2 of \cite{wald} for some further properties of the Harish-Chandra $\mu$ function. 

Clearly the $\mu$ function denoted above $\mu^G$ can be defined with respect to any reductive group $G$, in particular we will use below the functions $\mu^M$  for a Levi subgroup $M$. 

Let $P_1= M_1U_1$ be a standard parabolic subgroup. In \cite{heiermannorbit} and \cite{heiermannope}, with the notations introduced in the Section \ref{ionngk}, the following results are mentioned:
\begin{thm}[Harish-Chandra, see \cite{heiermannope}, 1.2]\label{hc}
Fix a root $\alpha \in \Sigma(P_1)$ and an irreducible cuspidal representation $\sigma$ of $M_1$.

a) If $\mu^{(M_1)_{\alpha}} (\sigma) = 0$, then there exists a unique (see Casselman's notes, 7.1 in \cite{casselman}) non trivial element $s_{\alpha}$ in $W^{(M_1)_{\alpha}}(M_1)$ so that $s_{\alpha}(P_1 \cap (M_1)_{\alpha}) = \overline{P_1}\cap (M_1)_{\alpha}$ and $s_{\alpha} \sigma \cong \sigma$.

b) If there exists a unique non trivial element $s_{\alpha}$ in $W^{(M_1)_{\alpha}}(M_1)$ so $s_{\alpha}(P_1 \cap (M_1)_{\alpha}) = \overline{P_1}\cap (M_1)_{\alpha}$ and $s_{\alpha} \sigma \cong \sigma$, then $\mu^{(M_1)_{\alpha}}(\sigma) \neq 0  \Leftrightarrow I_{P_1\cap (M_1)_{\alpha}}
^{(M_1)_{\alpha}}(\sigma)$ is reducible.

If it is reducible, it is the direct sum of two non isomorphic representations.
\end{thm}

The $\mu$ function's factor in this setting is:
$$\mu^{(M_1)_{\beta}}(\sigma_{\lambda}) = c_{\beta}(\lambda).\frac
{ (1- q^{\prodscal{\check{\beta}}{ \lambda}})(1- q^{-\prodscal{\check{\beta}}{ \lambda}})  }
{ (1- q^{\epsilon_{\check{\beta}} + \prodscal{\check{\beta}}{ \lambda}}) (1- q^{\epsilon_{\check{\beta}} - \prodscal{\check{\beta}}{ \lambda}})  }$$

\begin{lemma}[Lemma 1.8 in \cite{heiermannope}]\label{1.8}
Let $\alpha \in \Delta_{\sigma}$, $s=s_{\alpha}$ and assume $(M_1)_{\alpha}$ is a standard Levi subgroup of $G$. The 
operators $J_{sP_1|P_1}$ are meromorphic functions in $\sigma_{\lambda}$ for $\sigma$ unitary cuspidal representation and $\lambda$ a parameter in $(a_{M_1}^{(M_1)_{\alpha}}*)$. 

The poles of $J_{sP_1|P_1}$ are precisely the zeroes of $\mu^{(M_1)_{\alpha}}$. Any pole has order one and its residue is bijective. Furthermore, 
$J_{P_1|sP_1}J_{sP_1|P_1}$ equals $(\mu^{(M_1)_{\alpha}})^{-1}$ up to a multiplicative constant.
\end{lemma}

Let us summarize the different cases:
\begin{itemize}
	
	\item If $\mu^{(M_1)_{\alpha}}$ has a pole at $\sigma_{\lambda}$; then, the operators $J_{P_1|sP_1}$ and $J_{sP_1|P_1}$ (which are 
necessarily 
both non-zero) cannot be bijective. Indeed, at $\sigma_{\lambda}$ their product is zero, if any was bijective, it would imply the other is zero.
	\item If $\mu^{(M_1)_{\alpha}}$ has a zero in $\sigma_{\lambda}$; it is Lemma \ref{1.8} above.

\end{itemize}

Further by a general result concerning the $\mu$ function, it has one and only one pole on the positive real axis if and only if, for $\sigma$ a unitary irreducible cuspidal representation, $\mu(\sigma)=0$. Therefore, for each  $\alpha \in \Sigma_{\sigma}$, by definition, there is be one $\lambda$ on the positive real axis such that $\mu^{(M_1)_{\alpha}}$ has a 
pole.

\begin{example}
Consider the group $G = GL_{2n}$ and one of its maximal Levi subgroups $M:= GL_n\times GL_n$. Set $\sigma_s:= \rho|\det|^s\otimes\rho|\det|^{-s}$ with $
\rho$ irreducible 
unitary cuspidal representation of $GL_n$. Then, $\mu(\rho\otimes\rho)= 0$ and it is well known that at $s=\pm 1/2$, $\mu(\sigma_s)$ has a 
pole and the 
operators $J_{P|\overline{P}}$ and $J_{\overline{P}|P}$ are not bijective.
\end{example}

\subsection{Some results on residual points}

Let $Q$ be any parabolic subgroup of $G$, with Levi decomposition $Q= LU$.
We recall that the parabolic rank of $G$ (with respect to $L$) is $rk_{ss}(G) - rk_{ss}(L)$, where $rk_{ss}$ stands for the semi-simple rank. 
The following 
definition will be useful:

\begin{dfn}[residual point] \label{residualpoint}
A point $\sigma_{\nu}$ for $\sigma$ an irreducible unitary cuspidal representation of $L$ is called a residual point for $\mu^G$ if 
$$|\left\{ \alpha \in \Sigma(Q) |\prodscal{\check{\alpha}}{\nu} = \pm \epsilon_{\alpha} \right\}| - 2|\left\{ \alpha \in \Sigma(Q) |
\prodscal{\check{\alpha}}{\nu} = 
0\right\}| = \dim(a_L^*/a_G^*) = rk_{ss}(G) - rk_{ss}(L) $$ where $\epsilon_{\alpha}$ appears in the Section \ref{mu}.
\end{dfn}

\begin{rmk} 
Since the $\mu$ function depends only on a complex variable identified with $\sigma\otimes\chi_{\lambda}$, for $\lambda \in (a_L^G)^*$; once the unitary cuspidal representation $\sigma$ is fixed we will freely talk about $\lambda$ (rather than $\sigma_{\lambda}$) as a residual point.
\end{rmk}

The main result of Heiermann in \cite{heiermann} is the following:
\begin{thm}[Corollary 8.7 in \cite{heiermann}] \label{heir}
Let $Q=LU$ be a parabolic subgroup of $G$, $\sigma$ a unitary cuspidal representation of $L$, and $\nu$ in $a_L^*$. For the induced 
representation 
$I_Q^G(\sigma_{\nu})$ to have a discrete series subquotient, it is necessary and sufficient for $\sigma_{\nu}$
to be a residual point for $\mu^G$ and the restriction of $\sigma_{\nu}$ to $A_G$ (the maximal split component in the center of $G$) to be a 
unitary 
character. 
\end{thm}

We will also make a crucial use of the following result from \cite{opdamh}:
\begin{prop}[Proposition 2.5 in \cite{opdamh}]\label{opdamh}
Let $\pi$ be an irreducible generic representation which is a discrete series of $G$. There exists a standard parabolic subgroup $Q=LU$ of $G$ and a unitary 
generic 
cuspidal representation $(\sigma,E)$ of $L$, with $\nu \in \overline{(a_L^*)^+}$ such that $\pi$ is a subrepresentation of $I_Q^G(\sigma_{\nu})
$.
\end{prop}

We need the following definition to recall the Langlands' classification (see for instance \cite{borelwallach} Theorem 2.11 or \cite{konno}): 
\begin{dfn} \label{Langlandsdata}
A set of Langlands data for $G$ is a triple $(P, \tau, \nu )$ with the following
properties:
\begin{enumerate}
\item $P = MU$ is a standard parabolic subgroup of $G$
\item $\nu$ is in $(a_M^*)^+$
\item $\tau$ is (the equivalence class of) an irreducible tempered representation
of $M$. 
\end{enumerate}
\end{dfn}

\begin{thm}[Langlands' classification] \label{LC}

\begin{enumerate}
\item Let $(P, \tau, \nu)$ be a set of Langlands data. Then the induced 
representation $I_P^G(\tau_{\nu})$ has a unique irreducible quotient, the Langlands quotient denoted $J(P,\nu,\tau)$ 

\item Let $\pi$ be an irreducible admissible representation of $G$. Then there exists a unique triple $(P,\nu, \tau)$ as in (1) such that $\pi =  
J(P,\nu,\tau)$. 
We call this triple the Langlands data, and $\nu$ is called the Langlands parameter of $\pi$.

\end{enumerate}
\end{thm}

\begin{thm}[\emph{Standard module conjecture} proved in \cite{smc} and \cite{opdamh}] \label{SMC}
Let $\nu \in a_M^{*+}$, and $\tau$ be an irreducible tempered generic representation of $M$. Denote $J(\tau, \nu)$ the Langlands quotient of the induced representation $I_P^G(\tau_{\nu})$. Then, the representation $J(\tau, \nu)$ is generic if and only if $I_P^G(\tau_{\nu})$ is irreducible.
\end{thm}

\section{Setting and first results on intertwining operators} \label{io1}
\subsection{The setting}
Following \cite{opdamh}, let us denote $a_{M_1}^{M*} = \mathbb{R}\Sigma^M \subset a_{M_1}^{G*}$, where $\Sigma^M$ are the roots in $\Sigma$ which are in $M$ (with basis $\Delta^M$) (see also \cite{renard} V.3.13).

With the setting and notations as given at the end of the introduction (see \ref{max}), we consider $\tau$ a generic discrete series of $M$. By the above proposition (Proposition \ref{opdamh}) there exists a standard parabolic subgroup $P_1 = M_1U_1$ of $G$, and we could further assume $M_1 \subset M$, $\sigma_{\nu}$ a cuspidal representation of $M_1$, Levi subgroup of $M \cap P_1$ such that $\tau$ is a generic discrete series that appears as subrepresentation of $I_{M\cap P_1}^M(\sigma_{\nu})$, with $\nu$ is in the closed positive Weyl chamber relative to $M$, $\overline{(a_{M_1}^{M*})^{+}}$. 
Moreover, $\sigma_{\nu}$  is a residual point for $\mu^M$.

By transitivity of induction, we have: 
$$I_P^G(\tau_{s\tilde{\alpha}}) \hookrightarrow I_P^G (I_{M\cap P_1}^M(\sigma_{\nu}))_{s\tilde{\alpha}} = I_{P_1}^G(\sigma_{\nu+s
\tilde{\alpha}})$$ 
where $s \in \mathbb{R}$ satisfies $s > 0$ and $\tilde{\alpha} = (\rho_P,\alpha)^{-1}\rho_P$ (Rather than $\tilde{\alpha}$, we could also take the fundamental weight corresponding to $\alpha$, but we will rather follow a convention of Shahidi [see \cite{sha}]).

\begin{convention}
The reader should note that our standard module $I_P^G(\tau_{s\tilde{\alpha}})$ is induced from an essentially square integrable 
representation $\tau_{s\tilde{\alpha}}$. 
The general case of a tempered representation $\tau$ will follow in the Corollary \ref{standard}. Throughout this paper, we 
will adopt the 
following convention: $\tau$ will denote a discrete series representation, $\sigma$ an (irreducible) cuspidal representation. Also following 
notations (as for 
instance in \cite{gic1} or \cite{MT}), $\pi \leq \Pi$ means $\pi$ is realised as a subquotient of $\Pi$, whereas $\pi \hookrightarrow \Pi$ is 
stronger, and means 
it embeds as a subrepresentation.
\end{convention}

In the following sections we will study the generic subquotient of $I_{P_1}^G(\sigma_{\nu + s\tilde{\alpha}})$ and consider the cases where 
either there 
exists a discrete series subquotient, or there isn't and therefore tempered or non-tempered generic (not square integrable) subquotients may 
occur. 

Given a generic discrete series subquotient $\gamma$ in $I_{P_1}^G(\sigma_{\nu + s\tilde{\alpha}})$, using Proposition \ref{opdamh} above, it appears 
as a generic subrepresentation in some induced representation $I_{P'}^G(\sigma'_{\lambda'})$ for $\lambda'$ in the closure of the positive Weyl chamber with respect to 
$P'$, and $\sigma'$ irreducible cuspidal generic.

The set-up is summarized in the following diagram:

\begin{tikzcd}
\gamma \leq & I_P^G(\tau_{s\tilde{\alpha}}) \arrow[hook]{r}  & I_{P_1}^G(\sigma_{\nu + s\tilde{\alpha}})  \\
 \gamma \arrow[hook]{rr} &  & I_{P'}^G(\sigma'_{\lambda'}) \arrow[u, dotted,]	\\
\end{tikzcd}

We will investigate the existence of a bijective up-arrow on the right of this diagram.

\subsection{Intertwining operators}

\begin{lemma}\label{lemma}
Let $P_1$ and $Q$ be two parabolic subgroups of $G$ having the same Levi subgroup $M_1$.

Then, there exists an isomorphism $r_{P_1|Q}$ between the two induced modules $I_{Q}^G(\sigma_{\lambda})$ and $I_{P_1}
^G(\sigma_{\lambda})$ for any 
irreducible unitary cuspidal representation $\sigma$ whenever $\lambda$ is dominant for both $P_1$ and $Q$. 
\end{lemma}

\begin{proof}
We first assume that $Q$ and $P_1$ are 
adjacent (Two parabolic subgroups $Q$ and $P_1$ are adjacent along 
$\alpha$ if $\Sigma(P_1)\cap -\Sigma(Q)= \left\{\alpha\right\}$). 
We denote $\beta$ the common root of $\Sigma(\overline{Q})$ and $\Sigma(P_1)$. $\overline{Q}$ is the parabolic subgroup 
opposite to $Q$ with Levi subgroup $M_1$. 
\\

We have 
$$I_Q^G(\sigma_{\lambda}) = I_{Q_{\beta}}^G(I_{Q\cap (M_1)_{\beta}}^{(M_1)_{\beta}}(\sigma_{\lambda}))$$
 where $(M_1)_{\beta}$ is the centralizer of 
$A_{\beta}$ (the identity component in the kernel of $\beta$) in $G$, a semi-standard Levi subgroup (confer section 1 in \cite{wald}), and the 
same inductive  formula holds replacing $Q$ by $P_1$.
Since $\lambda$ is dominant for both $Q$ and $P_1$, $\prodscal{\lambda}{\beta} \geq 0$ (since $\beta$ is a root in $\Sigma(P_1)$), but also 
$\prodscal{\lambda}{-\beta} \geq 0 $ since $-\beta$ is a root in $\Sigma(Q)$. Therefore, $\prodscal{\check{\beta}}{\lambda} = 0$.
\\
We have $\lambda$ in $a_{M_1}^*$ which decomposes as $(a_{M_1}^{(M_1)_{\beta}})^*\oplus (a_{(M_1)_{\beta}})^*$
 and we write $\lambda = \mu \oplus \eta$.
The dual of the Lie algebra, $(a_{M_1}^{(M_1)_{\beta}})^*$, is of dimension one (since $M_1$ is a maximal Levi subgroup in $(M_1)_{\beta}$) generated 
by $\check{\beta}$. If $\prodscal{\check{\beta}}{\lambda} = 0$, the projection of $\lambda$ on $(a_{M_1}^{(M_1)_{\beta}})^*$ is also zero. That is 
$\prodscal{\check{\beta}}{\mu} = 0$ or $\chi_{\mu}$ is unitary.

Therefore, with $\sigma$ unitary, and $\chi_{\mu}$ a unitary character, the representations 
$$I_{Q\cap (M_1)_{\beta}}^{(M_1)_{\beta}}
(\sigma_{\mu})\qquad\hbox{and}\qquad
I_{P_1\cap (M_1)_{\beta}}^{(M_1)_{\beta}}(\sigma_{\mu})$$ are unitary. Since they trivially satisfy the conditions (i) of Theorem 2.9 in \cite{BZI} 
(see also 
\cite{renard} VI.5.4) they have equivalent Jordan-H\"older composition series, and are therefore isomorphic (as unitary representations, 
having equivalent 
Jordan-H\"older composition series). 
Tensoring with $\chi_{\eta}$ preserves the isomorphism between 
$$I_{Q\cap (M_1)_{\beta}}^{(M_1)_{\beta}}(\sigma_{\mu})\qquad\hbox{and}\qquad I_{P_1\cap 
(M_1)_{\beta}}
^{(M_1)_{\beta}}(\sigma_{\mu})$$
That is, there exists an isomorphism between $I_{Q\cap (M_1)_{\beta}}^{(M_1)_{\beta}}(\sigma_{\lambda})$ and 
$I_{P\cap 
(M_1)_{\beta}}^{(M_1)_{\beta}}(\sigma_{\lambda})$. The induction of this isomorphism therefore gives an isomorphism between 
$I_Q^G(\sigma_{\lambda})$ and 
$I_{P_1}^G(\sigma_{\lambda})$ that we call $r_{P_1|Q}$.

If we further assume that $Q$ and $P_1$ are not adjacent, but can be connected by a sequence of adjacent parabolic subgroups of $G$, $\left\{ Q=Q_1, Q_2, Q_3, \ldots, Q_n =P_1 \right\}$ with $\Sigma(Q_i) \cap \Sigma(\overline{Q_{i+1}}) = \left\{\beta_i\right\}$
We have the following set-up :
$$ I_{Q}^G(\sigma_{\lambda}) \stackrel{r_{Q_2|Q}}{\rightarrow} I_{Q_2}^G(\sigma_{\lambda}) \stackrel{r_{Q_3|Q_2}}{\rightarrow} I_{Q_3}^G(\sigma_{\lambda}) \ldots \stackrel{r_{Q_n|Q_{n-1}}}{\rightarrow} I_{P_1}^G(\sigma_{\lambda})$$

Again, under the assumption that $\lambda$ is dominant for $P_1$ and $Q$, we have $\prodscal{\beta_i}{\lambda} \geq 0$ and $\prodscal{-\beta_i}{\lambda} \geq 0$ for each $\beta_i$ in $\Sigma(P_1) \cap \Sigma(\overline{Q})$, hence $\prodscal{\check{\beta_i}}{\lambda} = 0$. Therefore, there exists an isomorphism between $I_{Q_i}^G(\sigma_{\lambda})$ and $I_{Q_{i+1}}^G(\sigma_{\lambda})$ denoted $r_{Q_{i+1}|Q_i}$. The composition of the isomorphisms $r_{Q_{i+1}|Q_i}$ will eventually give us the desired isomorphism between $I_{Q}
^G(\sigma_{\lambda})$ and 
$I_{P_1}^G(\sigma_{\lambda})$.
\end{proof}

 \begin{prop} \label{prop}
Let $I_{P'}^G(\sigma'_{\lambda'})$ and $I_{P_1}^G(\sigma_{\lambda})$ be two induced modules with $\sigma$ (resp. $\sigma'$) irreducible cuspidal representation of $M_1$ (resp. $M'$), $\lambda \in a_{M_1}^*$, $\lambda' \in a_{M'}^*$, sharing a common subquotient, then:

\begin{enumerate}
\item There exists an element $g$ in $G$ such that $~^g\!P':= gP'g^{-1}$ and $P_1$ have the same Levi subgroup.
\item If $\lambda$ and $\lambda'$ are dominant for $P_1$ (resp. $P'$), there exists an isomorphism $R_g$ between $I_{P'}
^G(\sigma'_{\lambda'})$ and 
$I_{P_1}^G(\sigma_{\lambda})$
\end{enumerate}
\end{prop}

\begin{proof}
First, since the representations $I_{P'}^G(\sigma'_{\lambda'})$ and $I_{P_1}^G(\sigma_{\lambda})$ share a common subquotient by Theorem 2.9 in \cite{BZI}, there exists an element $g$ in $G$ such that $M_1= gM'g^{-1}$, $~^g\!\sigma'_{\lambda'} = \sigma_{\lambda}$ and $g\lambda'= \lambda$, where $~^g\!\sigma(x) = \sigma(g^{-1}xg)$ for $x \in M_1$.
The last point follows from the equality $~^g\!\chi_{\lambda'} =  \chi_{g\lambda'}$. 

For the second point, we first apply the map $t(g)$ between $I_{P'}^G(\sigma'_{\lambda'})$ and $I_{~^g\!P'}^G(~^g\!\sigma'_{\lambda'})$ which is an isomorphism that sends $f$ on $f(g^{-1}.)$ 

As $\lambda'$ is dominant for $P', g\lambda'=\lambda$ is dominant for $~^g\!P'$, and we can further apply the isomorphism defined in the previous lemma (Lemma ~\ref{lemma}): $r_{P_1|~^g\!P'}(\sigma_{\lambda})$ (since $P_1$ and $~^g\!P'$ 
have the same Levi subgroup: $M_1$), we will therefore have:
$$I_{P'}^G (\sigma'_{\lambda'}) \stackrel{t(g)}\rightarrow I_{~^g\!P'}^G(~^g\!\sigma', g.\lambda') 
\stackrel{ r_{P_1|~^g\!P'}}{\rightarrow} I_{P_1}^G(\sigma_{\lambda})$$ 
and $R_g$ is the isomorphism given by the composition of $t(g)$ and $r_{P_1|~^g\!P'}$.
\end{proof}

\subsubsection{Intertwining operators with non-generic kernels} \label{ionngk}

Our objective is to embed an irreducible generic subquotient as a subrepresentation in a module induced from the data $(P_1,\sigma,\lambda)$ knowing it embeds in one with Langlands' data $(P',\sigma',\lambda')$. Notice that $(P_1,\sigma,\lambda)$ 
is not necessarily a Langlands data since, as explained in the beginning of Section \ref{balacarter}, the parameter $\lambda$ is not necessarily in the positive Weyl chamber $(a_{M_1}^*)^+$.
If the intertwining operator between those two induced modules has non-generic kernel, the generic subrepresentation will necessarily appear 
in the image 
of the intertwining operator, and therefore will appear as a \textsl{subrepresentation} in the  induced module with Langlands' data $
(P_1,\sigma,\lambda)$.
We detail the conditions to obtain the non-genericity of the kernel of the intertwining operator.

\begin{prop} \label{nngenerick}
Let $P_1$ and $Q$ be two parabolic subgroups of $G$ having the same Levi subgroup $M_1$.

Consider the two induced modules $I_{Q}^G(\sigma_{\lambda})$ and $I_{P_1}^G(\sigma_{\lambda})$, and assume $\sigma$ is an irreducible 
generic 
cuspidal representation and $\lambda$ is dominant for $P_1$ and anti-dominant for $Q$.
Then there exists an intertwining map from $I_{Q}^G(\sigma_{\lambda})$ to $I_{P_1}^G(\sigma_{\lambda})$ which has non-generic kernel.
\end{prop}

\begin{proof}
We first assume that $Q$ and $P_1$ are adjacent. We denote $\beta$ the common root of $\Sigma(Q)$ and $\Sigma(\overline{P_1})$. 

We have $I_Q^G(\sigma_{\lambda}) = I_{Q_{\beta}}^G(I_{Q\cap (M_1)_{\beta}}^{(M_1)_{\beta}}(\sigma_{\lambda}))$ where $(M_1)_{\beta}$ is the 
centralizer of 
$A_{\beta}$ ( the identity component in the kernel of $\beta$) in $G$, a semi-standard Levi subgroup (confer Section 1 in \cite{wald}), and the 
same 
inductive formula holds replacing $Q$ by $P_1$.
Then, there are two cases: The case of $ \prodscal{\check{\beta}}{\lambda} = 0$ is Lemma \ref{lemma}. If $ 
\prodscal{\check{\beta}}{\lambda} > 
0$, let us consider the intertwining operator defined in Section \ref{preliminaries} between $I_{P_1\cap (M_1)_{\beta}}^{(M_1)_{\beta}}
(\sigma_{\lambda})$ and 
$I_{Q\cap (M_1)_{\beta}}^{(M_1)_{\beta}}(\sigma_{\lambda})$ and assume it is not an isomorphism. 
The representation $\sigma$ being cuspidal, these modules are length two representations by the Corollary 7.1.2 of Casselman's 
\cite{casselman}.
Let $S$ be the kernel of this intertwining map and the Langlands quotient $J(\sigma,P_1\cap (M_1)_{\beta},\lambda)$ its image. One has the 
exact 
sequences:
$$0 \rightarrow S \rightarrow I_{P_1\cap (M_1)_{\beta}}^{(M_1)_{\beta}}(\sigma_{\lambda}) \rightarrow J(\sigma,P_1\cap (M_1)_{\beta},\lambda) 
\rightarrow 0$$
$$0 \rightarrow J(\sigma,P_1\cap (M_1)_{\beta},\lambda) \rightarrow I_{Q\cap (M_1)_{\beta}}^{(M_1)_{\beta}}(\sigma_{\lambda}) \rightarrow S 
\rightarrow 0$$

Further, the projection from 
$$I_{Q\cap (M_1)_{\beta}}^{(M_1)_{\beta}}(\sigma_{\lambda})$$ to $$I_{Q\cap (M_1)_{\beta}}^{(M_1)_{\beta}}(\sigma_{\lambda})\slash 
J(\sigma,P_1\cap (M_1)_{\beta},\lambda) \cong S \subset I_{P_1\cap (M_1)_{\beta}}^{(M_1)_{\beta}}(\sigma_{\lambda})$$ defines a map whose 
kernel, 
$J(\sigma,P_1\cap (M_1)_{\beta},\lambda)$, is not generic (by the main result of \cite{smc} which proves the Standard module Conjecture). 
In other words, we have the following exact sequence:
$$0 \rightarrow J(\sigma,P_1\cap (M_1)_{\beta},\lambda) \rightarrow I_{Q\cap (M_1)_{\beta}}^{(M_1)_{\beta}}(\sigma_{\lambda}) \stackrel{A}
{\rightarrow} 
I_{P_1\cap (M_1)_{\beta}}^{(M_1)_{\beta}}(\sigma_{\lambda})$$

Inducing from $(P_1)_{\beta}$ to $G$, one observes that the kernel of the induced map ($I_{(P_1)_{\beta}}^G(A)$) is the induction of the kernel 
$J(\sigma,P_1\cap (M_1)_{\beta},\lambda)$. Therefore, the kernel of the induced map is non-generic (here, we use the fact that there exists an 
isomorphism 
between the Whittaker models of the inducing and the induced representations, using result of \cite{rodier2} and \cite{cassha}).

Assume now that $Q$ and $P_1$ are not adjacent, but can be connected by a sequence of adjacent parabolic subgroups of $G$, 
$$\left\{ Q=Q_1, Q_2, Q_3, \ldots, Q_n =P_1 \right\} ~\hbox{with}~ \Sigma(Q_i) \cap \Sigma(\overline{Q_{i+1}}) = \left\{\beta_i\right\}$$
We have the following set-up :
$$ I_{Q}^G(\sigma_{\lambda}) \stackrel{r_{Q_2|Q}}{\rightarrow} I_{Q_2}^G(\sigma_{\lambda}) \stackrel{r_{Q_3|Q_2}}{\rightarrow} I_{Q_3}
^G(\sigma_{\lambda}) \ldots \stackrel{r_{Q_n|Q_{n-1}}}{\rightarrow} I_{P_1}^G(\sigma_{\lambda})$$

Assume that certain maps $r_{Q_{i+1}|Q_i}$ have a kernel, by the same argument as above their kernels are non-generic and therefore the 
kernel of the 
composite map is non-generic. Indeed, we have the next Lemma ~\ref{ngk}.
\end{proof}

\begin{lemma} \label{ngk}
The composition of intertwining operators with non-generic kernel has non-generic kernel.
\end{lemma}

\begin{proof}
Consider first the composition of two operators, $A$ and $B$ as follows:
 $$I_{Q}^G(\sigma_{\lambda}) \stackrel{A}{\rightarrow} I_{Q_2}^G(\sigma_{\lambda}) \stackrel{B}{\rightarrow} I_{P_1}^G(\sigma_{\lambda})$$ 

Clearly, the kernel of the composite ($B\circ A$) contains the kernel of $A$ and the elements in the space of the representation $I_{Q}
^G(\sigma_{\lambda})
$, $x$, such that $A(x)$ is in the kernel of $B$. 

This means we have the following sequence of homomorphisms:
$$0 \rightarrow \ker(A) \rightarrow \ker(B\circ A) \stackrel{A}{\rightarrow} \ker(B)\cap 
\Ima(A) \rightarrow 0$$
pull-back by  
$A^{-1}$ of element in $\ker(B)$. The pull-back of a non-generic kernel yields a non-generic subspace in the pre-image. 
The fact that this sequence is exact is clear except for the surjectivity of the map
$\ker(B\circ A) \stackrel{A}{\rightarrow} \ker(B)\cap 
\Ima(A)$. But, if $y \in  \ker(B)\cap 
\Ima(A)$, then there exists $x$ such that $A(x)=y$ and we have $B\circ A(x) = B(y)= 0$ since $y \in \ker(B)$.

If both $\ker(B)$ and $\ker(A)$ are non-generic, the kernel of ($B\circ A$) is itself non-generic.
Extending the reasoning to a sequence of rank one operators with non-generic kernels yields the result.
\end{proof}

We have observed that the nature of intertwining operators rely on the dominance of the parameters $\lambda$ and $\lambda'$. We now need a more explicit description of these parameters; to do so we will call on a result first presented in \cite{opdamspec} in the Hecke algebra context (Theorem Proposition 8.1 in \cite{opdamspec}, see also Appendix \ref{BCT}) and further developed in \cite{heiermannorbit}.

\section{Description of residual points via Bala-Carter} \label{balacarter}

With the notations of Section \ref{io1}, we will study generic subquotient in induced modules $I_{P_1}^G(\sigma_{\nu+s\tilde{\alpha}})$ and $I_{P'}^G(\sigma'_{\lambda'})$. 

One needs to observe, following the construction of our setting in Section \ref{io1}, that $\nu$ is in the closed positive Weyl chamber relative to 
$M$, $
\overline{(a_{M_1}^{M*})^+}$, whereas 
$s\tilde{\alpha}$ is in the positive Weyl chamber $(a_{M}^*)^+$,
therefore it is not expected that $\nu+s\tilde{\alpha}$ should be in the closure of the positive Weyl chamber $\overline{(a_{M_1}^{*})^+}$. 

In particular, let $\alpha$ be the only root in $\Sigma(A_0)$ which is not in $\text{Lie}(M)$, we may have $\prodscal{\nu}{\check{\alpha}} < 0$ 
and therefore for 
some roots $\beta \in \Sigma(A_{M_1})$, written as linear combination containing the simple root $\alpha$, we may also have: $\prodscal{\nu+ s
\tilde{\alpha}}
{\check{\beta}} < 0$.

However, by the result presented in Appendix \ref{BCT}, if $\nu + s\tilde{\alpha}$ is a residual point, it is in the Weyl group orbit of a 
dominant residual point (i.e. one whose expression can be directly deduced from a weighted Dynkin diagram).
We therefore define:
\begin{dfn}[dominant residual point]
A residual point $\sigma_{\lambda}$ for $\sigma$ an irreducible cuspidal representation is dominant if $\lambda$ is in the closed positive Weyl 
chamber $
\overline{(a_{M}^{*})^+}$.
\end{dfn}

Bala-Carter theory allows to describe explicitly the Weyl group orbit of a residual point.
In the context of reductive p-adic groups studied in \cite{heiermannorbit} (see in particular Proposition 6.2 in \cite{heiermannorbit}), the fact that $\sigma_{\lambda}$ lies in the cuspidal support of 
a discrete 
series can be translated somehow to the assertion that $\sigma_{\lambda}$ corresponds to a distinguished nilpotent orbit in the dual of the Lie algebra $~^L\!\g$, and therefore by Proposition \ref{resbij} (see also \cite{opdamspec}, Appendices A and B,  Proposition 8.1) to a weighted Dynkin diagram. Notice that Proposition \ref{resbij} requires: $G$ to be a semi-simple adjoint group; a certain parameter $k_{\alpha}$ to equal one for any root $\alpha$ in $\Phi$; further, it concerns only the case of unramified characters.

In the present work we treat the case of weighted Dynkin diagrams of type $A,B,C,D$. The key proposition is Proposition \ref{1.13bis} below.

\subsubsection*{Our setting} \label{setting3}
Recall that in Section \ref{io1} we embedded the standard module as follows:
$$I_P^G(\tau_{s\tilde{\alpha}}) \hookrightarrow I_P^G (I_{M\cap P_1}^M(\sigma_{\nu}))_{s\tilde{\alpha}} = I_{P_1}^G(\sigma_{\nu+s
\tilde{\alpha}})$$ 
By hypothesis, $\sigma_{\nu}$  is a residual point for $\mu^M$.

$\lambda=\nu+s\tilde{\alpha}$ is in $a_{M_1}^*$.

Describing explicitly the form of the parameter $\lambda \in a_{M_1}^*$ is essential for two reasons:
first, to determine the nature (i.e discrete series, tempered, or non-tempered representations) of the irreducible generic subquotients in the 
induced module 
$I_{P_1}^G(\sigma_{\lambda})$; secondly, to describe the intertwining operators and in particular the (non)-genericity of their kernels.
\\
We will explain the following correspondences:
\begin{multline}\label{eq:summaryy}
\left\{\text{dominant residual point}\right\} \leftrightarrow 
\left\{\text{Weighted Dynkin diagram}\right\} \\
\leftrightarrow 
\left\{\text{residual segments}\right\} \leftrightarrow  \left\{\text{Jumps of the residual segment}\right\} 
\end{multline}

The connection between residual points and roots systems involved for Weighted Dynkin Diagrams require a careful description of the involved participants:
\subsubsection*{The root system} \label{rootsystem}
Let us now recall that $W(M_1)$ the set of representatives in $W$ of elements in the quotient group 
$\left\{w \in W| w^{-1} M_1w= M_1 \right\}\slash W^{M_1}$ of minimal length in their right classes modulo $W^{M_1}$.

Assume $\sigma$ is a unitary cuspidal representation of a Levi subgroup $M_1$ in $G$, and let $W(\sigma, M_1)$ be the subgroup of $W(M_1)$ stabilizer of $\sigma$. 
The Weyl group of $\Sigma_{\sigma}$ is $W_{\sigma}$, the subgroup of $W(M_1, \sigma)$ generated by the reflexions $s_{\alpha}$. 

\begin{prop}[3.5 in \cite{silbergerSD}] \label{propsil}
The set $\Sigma_{\sigma}: = \left\{\alpha \in \Sigma_{\text{red}}(A_{M_1})| \mu^{(M_1)_{\alpha}}(\sigma) = 0\right\}$ is 
a root system. 

For $\alpha \in \Sigma_{\sigma}$, let $s_{\alpha}$ the unique element in $W^{(M_1)_{\alpha}}(M_1,\sigma)$ which conjugates $P_1\cap M_{\alpha}$ and $\overline{P_1}\cap (M_1)_{\alpha}$.
The Weyl group $W_{\sigma}$ of $\Sigma_{\sigma}$ identifies to the subgroup of $W(M_1, \sigma)$ generated by reflexions $s_{\alpha}$, $\alpha \in \Sigma_{\sigma}$.

$\check{\alpha}$ the unique element in $a_{M_1}^{(M_1)_{\alpha}}$ which satisfies $\prodscal{\check{\alpha}}{\alpha} = 2$.

Then $\Sigma_{\sigma}^{\vee} := \left\{\check{\alpha}| \alpha \in \Sigma_{\sigma}\right\}$ is the set of coroots of $
\Sigma_{\sigma}$, the duality being that of $a_{M_1}$ and $a_{M_1}^*$.

The set $\Sigma(P_1)\cap\Sigma_{\sigma}$ is the set of positive roots for a certain order on $\Sigma_{\sigma}$.
\end{prop}

\begin{rmk}
An equivalent proposition is proved in \cite{heiermannope} (Proposition 1.3). There, the author considers $\mathcal{O}$ the set of equivalence classes of representations 
of the form $\sigma\otimes\chi$ where $\chi$ is an unramified character of $M_1$. He proves that the set $\Sigma_{\mathcal{O},\mu}: = \left\{\alpha \in \Sigma_{\text{red}}(A_{M_1})| 
\mu^{(M_1)_{\alpha}} 
\text{has a zero on} ~ \mathcal{O} \right\}$ is a root system.

The Weyl group of $G$ relative to a maximal split torus in $M_1$ acts on $\mathcal{O}$. 
The previous statement holds replacing $W_{\sigma}$ by $W(M_1, \mathcal{O})$, the subgroup of $W(M_1)$ stabilizer of $\mathcal{O}$.
\end{rmk}


\begin{lemma}
If $\sigma$ is the trivial representation of $M_1= M_0$, the root system  $\Sigma_{\sigma}$ is the root 
system of the 
group $G$ relative to $A_0$ (with length given by the choice of $P_0$). 
\end{lemma}

\begin{proof}
Recall  that $\Sigma_{\sigma}: = \left\{\alpha \in \Sigma_{\text{red}}(A_{M_1})| \mu^{(M_1)_{\alpha}}(\sigma)=0 \right\}$ 
 is a root  system. Apply this definition to the trivial representation. Clearly, for any $\alpha \in \Sigma(A_0)$, the trivial representation is fixed by any element in $W^{(M_0)_{\alpha}}(M_0)$, and therefore by $s_{\alpha}$ satisfying $s_{\alpha}(P_0 \cap (M_0)_{\alpha}) = \overline{P_0}\cap (M_0)_{\alpha}$. It is well-known that the induced representation $I_{P_0\cap (M_0)_{\alpha}}^{(M_0)_{\alpha}}(\textbf{1})$ is irreducible; therefore using Harish-Chandra's Theorem 
(Theorem \ref{hc}) 
above, $\mu^{(M_0)_{\alpha}}(\textbf{1}) = 0$.   Then 
$$\left\{\alpha \in \Sigma_{\text{red}}(A_0)| \mu^{(M_0)_{\alpha}}(\textbf{1}) = 0 \right\} := \left\{\alpha \in \Sigma(A_0)| \mu^{(M_0)_{\alpha}}(\textbf{1}) = 0 
\right\} 
= \left\{\alpha \in \Sigma(A_0) \right\}.$$
\end{proof}

In general, the root system $\Sigma_{\sigma}$ is the disjoint union of irreducible or empty components $\Sigma_{\sigma,i}$ for $i=1,\ldots, r$. 
This will be detailed in the Subsection \ref{generalcase}. 

\begin{prop}
Let $G$ be a quasi-split group whose root system $\Sigma$ is of type $A,B,C$ or $D$.
Then the irreducible components of $\Sigma_{\sigma}$ are of type $A,B,C$ or $D$.
\end{prop}

\begin{proof}
See the main result of the article \cite{PrSR} recalled in the Appendix \ref{lab}.
\end{proof}

\subsubsection*{How the root system $\Sigma _{\sigma}$ determines the Weighted Dynkin diagrams to be used in this work}


\begin{prop} \label{1.13bis}
Assume $G$ quasi-split over $F$. Let $M_1$ be a Levi subgroup of $G$ and $\sigma $ a generic irreducible unitary cuspidal representation of $M_1$. Put $\Sigma _{\sigma }=\{\alpha\in \Sigma_{red}(A_{M_1})\vert \mu^{(M_1)_{\alpha}}(\sigma )=0\}$. Let $$d=rk_{ss}(G)-rk_{ss}(M_1).$$

The set $\Sigma _{\sigma }$ is a root system in a subspace of $a_{M_1}^*$ (cf. Proposition \ref{propsil}). Suppose that the irreducible components of $\Sigma_{\sigma }$ are all of type $A$, $B$, $C$ or $D$. Denote, for each irreducible component $\Sigma _{\sigma ,i}$ of $\Sigma _{\sigma}$, by $a_{M_1}^{M^i*}$ the subspace of $a_{M_1}^{G*}$ generated by $\Sigma_{\sigma ,i}$, by $d_i$ its dimension and by $e_{i,1},\dots ,e_{i, d_i}$ a basis of $a_{M_1}^{M^i*}$ (resp. of a vector space of dimension $d_i+1$ containing $a_{M_1}^{M^i*}$ if $\Sigma _{\sigma ,i}$ is of type $A$) so that the elements of the root system $\Sigma _{\sigma ,i}$ are written in this basis as in Bourbaki \cite{bourbaki}.

For each $i$, there is a unique real number $t_i>0$ such that, if $\alpha =\pm e_{i,j}\pm e_{i,j'}$ lies in $\Sigma _{\sigma,i}$, then $I_{P_1\cap (M_1)_{\alpha}}^{(M_1)_{\alpha}}(\sigma _{\frac{t_i}{2}(\pm e_{i,j}\pm e_{i,j'})})$ is reducible.

If $\Sigma _{\sigma ,i}$ is of type $B$ or $C$, then there is in addition a unique element $\epsilon_i\in\{1/2, 1\}$ such that $I_{P_1\cap (M_1)_{\alpha_{i,d_i}}}^{(M_1)_{\alpha_{i,d_i}}}(\sigma _{\epsilon _it_i
e_{i,d_i}})$ is reducible.

Let $\lambda =\sum _i\sum _{j=1}^{d_i}\lambda_{i,j}e_{i,j}$ be in $\overline{a_{M_1}^{G*+}}$ with $\lambda _{i,j}$ real numbers.

Then $\sigma _\lambda $ is in the cuspidal support of a discrete series representation of $G$, if and only if the following two properties are satisfied

(i) $d=\sum _id_i$;

(ii) For all $i$, $\frac{2}{t_i}(\lambda _{i,1},\dots ,\lambda _{i, d_i})$ corresponds to the Dynkin diagram of a distinguished parabolic of a simple complex adjoint group of

- type $D_{d_i}$ (resp. $A_{d_i}$) if $\Sigma _{\sigma,i}$ is of type $D$ (resp. $A$);

otherwise:

- of type $C_{d_i}$, if $\epsilon _i=1/2$;

- of type $B_{d_i}$, if $\epsilon _i=1$.

\end{prop}

\begin{proof}
As $\lambda $ lies in $a_{M_1}^{G*}$, $\sigma _{\lambda }$ lies in the cuspidal support of a discrete series representation of $G$, if and only if it is a residual point of Harish-Chandra's $\mu $-function.

Denote $e_{i,j;i',j'}^\pm$ the rational character of $A_{M_1}$ whose dual pairing with an element $x$ of $a_{M_1}^G$ with coordinates $$(x_{1,1}, \ldots, x_{1,d_1}, x_{2,1}, \ldots, x_{2,d_2}, \ldots, x_{r,1}, \ldots, x_{r,d_r})$$ in the dual basis equals $x_{i,j}x_{i',j'}^{\pm 1}$ and by $e_{i,j}^\pm$ the one whose dual pair equals $x_{i,j}^{\pm 1}$.

The $\mu $-function decomposes as $\prod_{\alpha\in\Sigma(P)}\mu ^{M_{\alpha }}$. By assumption, the function $\lambda\mapsto\mu^{M_{\alpha }}(\sigma _{\lambda })$ won't have a pole or zero on $a_{M_1}^*$ except if $\alpha\in\Sigma_{\sigma }$. This means that

(i) $\alpha $ is of the form $e_{i,j;i,j'}^-$, $j<j'$;

(ii) $\alpha $ is of the form $e_{i,j;i,j'}^+$, $j<j'$, and $\Sigma _{\sigma ,i}$ of type $B$, $C$ or $D$;

(iii) $\alpha $ is of the form $e_{i,j}^+$ or  $2e_{i,j}^+$ and $\Sigma _{\sigma ,i}$ of respectively type $B$ or $C$.

Let $(\lambda _{i,j})_{i,j}$ be a family of real numbers as in the statement of the proposition and put $\lambda =\sum _i\sum _{j=1}^{d_i}\lambda _{i,j}e_{i,j}$. It follows from Langlands-Shahidi theory (cf. the proof of Theorem 5.1 in \cite{opdamh}) that there is, for each $i$, a real number $t_i>0$ and $\epsilon_i\in\{1/2,1\}$, so that:

- If $\alpha =e_{i,j;i,j'}^\pm\in\Sigma _{\sigma }$, $j< j'$, then $$\mu ^{M_{\alpha }}(\sigma _{\lambda })=c_{\alpha }(\sigma _{(\lambda _{i,j})_{i,j}})
\frac{(1-q^{\lambda _{i,j}\pm \lambda _{i,j'}})(1-q^{-\lambda _{i,j}\mp \lambda _{i,j'}})}{(1-q^{t_i- \lambda _{i,j}\pm \lambda _{i,j'}})(1-q^{t_i+\lambda _{i,j}\mp \lambda _{i,j'}})},$$
where $c_{\alpha }(\sigma _{(\lambda _{i,j})_{i,j}})$ denotes a rational function in $\sigma _{(\lambda _{i,j})_{i,j}}$, which is regular and non-zero for real $\lambda _{i,j}$.

- If $\alpha =e_{i,j}\in\Sigma _{\sigma }$ or $\alpha =2e_{i,j}\in\Sigma _{\sigma }$, then $$\mu ^{M_{\alpha }}(\sigma _{(\lambda _{i,j})_{i,j}})=c_{\alpha }(\sigma _{(\lambda _{i,j})_{i,j}})\frac{(1-q^{\lambda _{i,j}})(1-q^{-\lambda _{i,j}})}{(1-q^{\epsilon_it_i- \lambda _{i,j}})(1-q^{\epsilon_it_i+\lambda _{i,j}})}$$
with $\epsilon_i=1, 1/2$.

Put $\kappa _i^+=0$ if $\Sigma_{\sigma ,i}$ is of type $A$ and put $\kappa _i=0$ if  $\Sigma _{\sigma ,i}$ is of type $A$ or $D$ and otherwise $\kappa_i=\kappa_i^+=1$. As $\lambda $ is in the closure of the positive Weyl chamber, it follows that, for $\sigma _{\lambda }$ to be a residual point of Harish-Chandra's $\mu $-function, it is necessary and sufficient, that for every $i$, one has

\begin{align}
 d_i=\vert\{(j,j')\vert j<j', \lambda _{i,j}-\lambda _{i,j'}=t_i\}\vert+\kappa_i^+\vert\{(j,j')\vert j<j', \lambda _{i,j}+\lambda _{i,j'}=t_i\}\vert+ \kappa_i\vert\{j\vert \lambda _{i,j}=\epsilon_it_i\}\vert \\
 -2[\vert\{(j,j')\vert j<j', \lambda _{i,j}-\lambda _{i,j'}=0\}\vert+ \kappa_i^+\vert\{(j,j')\vert j<j', \lambda _{i,j}+\lambda _{i,j'}=0\}\vert+\kappa_i\vert\{j\vert \lambda _{i,j}=0\}\vert]. 
\end{align}

If $\kappa_i=0$ or $\epsilon _i=1$, then this is the condition for $\frac{2}{t_i}(\lambda_{i,1},\dots , \lambda_{i,d_i})$ defining a distinguished nilpotent element in the Lie algebra of an adjoint simple complex group of type $A_{d_i}$, $D_{d_i}$ or $B_{d_i}$ as in 5.7.5 in \cite{carter}. If $\epsilon _i=1/2$, one sees that $\frac{2}{t_i}(\lambda _{i,1},\dots , \lambda _{i,d_i})$ defines a distinguished nilpotent element in the Lie algebra of an adjoint simple complex group of type $C_{d_i}$.

In other words, $\frac{2}{t_i}(\lambda_{i,1},\dots , \lambda_{i,d_i})$ corresponds to the Dynkin diagram of a distinguished parabolic subgroup of an adjoint simple complex group of type $B_n$, $C_n$ or $D_n$, if $\kappa _i^+=1$ and $\kappa_i\epsilon _i$ is respectively $1$, $1/2$ or $0$, and of type $A_n$ if $\kappa _i=0$.

\end{proof}

\begin{example}[See also Proposition 1.13 in \cite{heiermannope} and the Appendix of the author's PhD thesis \cite{these}]\label{example1.13}

In the context of classical groups, let us spell out the Levi subgroups and cuspidal representations of these Levi considered in the previous proposition:

Let $M_1$ be a standard Levi subgroup of a classical group $G$ and $\sigma$ a generic irreducible unitary cuspidal representation of $M_1$.

Then, up to conjugation by an element of $G$, we can assume:
$$M_1= \underbrace{GL_{k_1}\times \ldots GL_{k_1}}_{d_1 ~ \text{times}}\times \underbrace{GL_{k_2}\times \ldots \times GL_{k_2}}_{d_2 ~ \text{times}}\times \ldots \times \underbrace{GL_{k_r}
\times \ldots \times GL_{k_r}}_{d_r ~ \text{times}}\times G(k)$$

where $G(k)$ is a semi-simple group of absolute rank $k$ of the same type as $G$ and
$$\sigma = \sigma_1\otimes\ldots\otimes \sigma_1\otimes\sigma_2\otimes\ldots \otimes\sigma_2\ldots \ldots \otimes\sigma_r\otimes\ldots 
\otimes\sigma_r
\otimes\sigma_c$$ 

Let us assume $k \neq 0$, and $\sigma_i \ncong \sigma_j$ if $j \neq i$.

We identify $A_{M_1}$ to $\mathbb{T} = \mathbb{G}_m^{d_1}\times\mathbb{G}_m^{d_2}\times \ldots \times \mathbb{G}_m^{d_r}$ and denote $\alpha_{i,j}$ the rational character of $A_{M_1}$ (identified with $\mathbb{T}$) which sends an element $$x= (x_{1,1}, \ldots, x_{1,d_1}, x_{2,1}, \ldots, 
x_{2,d_2}, \ldots, 
x_{r,1}, \ldots, x_{r,d_r})$$ to $x_{i,j}x_{i,j+1}^{- 1}$ if $j < d_i$ and to $x_{i,d_i}$ if $j=d_i$.

Let $(s_{i,j})_{i,j}$ be a family of non-negative real numbers, $1\leq i\leq r$, $1\leq j\leq d_i$ 
and $s_{i,j}\geq s_{i,j+1}$ for $i$ fixed. Then,
   $$\sigma_1\vert\det\vert^{s_{1,1}}\otimes\ldots\sigma_1\vert\det\vert^{s_{1,d_1}}\otimes\sigma_2\vert\det\vert^{s_{2,1}}\otimes\ldots 
\sigma_2\vert\det
\vert^{s_{2,d_2}}\otimes\ldots \otimes\sigma_r\vert\det\vert^{s_{r,1}}\otimes\ldots \sigma_r\vert\det\vert^{s_{r,d_r}}\otimes\sigma_c.$$
is in the cuspidal support of a discrete series representations of $G$, if and only if the following properties are satisfied:

i) one has $\sigma _i\simeq\sigma _i^{\vee}$ for every $i$;

ii)
denote by $s_i$ the unique element in $\{0, 1/2, 1\}$ such that the representation of $G(k+k_i)$ parabolically induced from $\sigma _i\vert\cdot\vert ^{s_i}\otimes\sigma_c$ is reducible (we use the result of Shahidi on reducibility points for generic cuspidal representations).

iii) if, in addition, $G=SO_{2n}(F)$, the situation can be a little subtler. For instance, in the maximal parabolic case, with $\sigma= \sigma_1\otimes \sigma_c$ and $k_1$ odd, the long Weyl conjugate of $\sigma_1\otimes \sigma_c$ is
$\sigma_1^{\vee}\otimes c.\sigma_c$ where $c$ is a length zero representative of $O_{2n}(F)\backslash SO_{2n}(F)$.
In particular, if $c.\sigma_c \ncong \sigma_c$, $\sigma_1^{\vee}\otimes c.\sigma_c$ is not ramified, and no $s_1$ gives reducibility. However, this can still be the support of a discrete series. 

Then, for all $i$,  $2(s_{i,1},\dots ,s_{i,d_i})$ corresponds to the 
Dynkin diagram of a distinguished parabolic subgroup of a simple complex adjoint group of

    - type $D_{d_i}$ if $s_i=0$; then $\Sigma_{\sigma,i} = \left\{\alpha_{i,1}, \ldots, \alpha_{i,d_i-1}, \alpha_{i,d_i-1} + 2\alpha_{i,d_i} \right\}$

    - type $C_{d_i}$ if $s_i=1/2$; then $\Sigma_{\sigma,i} = \left\{\alpha_{i,1}, \ldots, \alpha_{i,d_i-1}, 2\alpha_{i,d_i}\right\}$

    - type $B_{d_i}$ if $s_i=1$; then $\Sigma_{\sigma,i} = \left\{\alpha_{i,1}, \ldots, \alpha_{i,d_i-1}, \alpha_{i,d_i}\right\}$.

For $i\neq j$, since $\sigma_i \ncong \sigma_j$, we have $\Sigma_{\sigma,i} \neq \Sigma_{\sigma,j}$.

Then $M^i$ is isomorphic to 

$$\underbrace{GL_{k_1}\times\ldots GL_{k_1}}_{d_1 ~ \text{times}}\times \ldots \times \underbrace{GL_{k_{i-1}}\times \ldots \times GL_{k_{i-1}}}_{d_{i-1} ~ \text{times}}\times \underbrace{GL_{k_{i+1}}\times \ldots \times GL_{k_{i+1}}}_{d_{i+1} ~ \text{times}}\times\ldots \times \underbrace{GL_{k_r}\times \ldots \times GL_{k_r}}_{d_r ~ \text{times}}\times G(k+d_ik_i)$$

\end{example}

\subsection{From weighted Dynkin diagrams to residual segments} \label{sgm}

The Dynkin diagram of a distinguished parabolic subgroup mentioned in the Proposition \ref{1.13bis} are also called \emph{Weighted Dynkin diagrams}: a definition is given in Appendix \ref{BCT} and their forms in \ref{WDD1} .

Let a parameter $\nu \in a_{M_1}^*$ be written $(\nu_1,\nu_2, \ldots, \nu_n)$ in a basis $\left\{e_1,e_2,\ldots, e_n \right\}$ (resp. $\left\{e_1,e_2,\ldots, e_n, e_{n+1}\right\}$ for type $A$) (such that this basis is the canonical basis associated to the classical Lie algebra $a_0^*$, as in \cite{bourbaki} when $M_1$ = $M_0$) and assume it is a dominant residual point. As it is dominant, observe that $\nu_1 \geq \nu_2 \geq \ldots \geq \nu_n \geq 0$ (resp. $\nu_1 \geq \nu_2 \geq \ldots \geq \nu_n$ for type $A$). Further it corresponds by the previous Proposition (\ref{1.13bis}) to a weighted Dynkin diagram of a certain type $A,B,C$ or $D$ (see also Bala-Carter theory presented in Appendix \ref{BCT}).

Let us explain the following correspondence:
\begin{equation}\label{eq:summary}
\left\{\text{Weighted Dynkin diagram}\right\} \leftrightarrow 
\left\{\text{residual segment}\right\} 
\end{equation}

First, let us explain the following assignment: 
$$\text{WDD} \rightarrow \nu, ~~ \hbox{where} ~~ \nu ~~ \hbox{is the vector with coordinates} ~~ \prodscal{\nu}{\alpha_i} $$

Let us start with a weighted Dynkin diagram of type $A,B,C$ or $D$.
The weights under roots $\alpha_i$ are 2 (respectively 0) which correspond to $\prodscal{\nu}{\alpha_i} = 1$ (respectively 0). See the weighted Dynkin diagrams given in Appendix \ref{WDD1}. Notice that we abusively use $\alpha_i$ rather than $\check{\alpha_i}$ in the product expression, to be consistent with the notations in the weighted Dynkin diagrams.

Using the expressions of $\alpha_i$ in the canonical basis (for instance $\alpha_i = e_i - e_{i+1}$, $2e_i$, or $e_i$), we compute the vector of 
coordinates $(\nu_1,\nu_2, \ldots, \nu_n)$ with integers or half-integers entries. 
For instance, for $\alpha_i = e_i - e_{i+1}$, when $\prodscal{\nu}{\alpha_i} =\prodscal{\sum_{i=1}^n\nu_ie_i}{\alpha_i}= 1$, we get $\nu_i - 
\nu_{i+1} = 1$, 
whereas if $\prodscal{\nu}{\alpha_i} = 0$ then $\nu_i  - \nu_{i+1} = 0$. Conversely, let us be given a vector of coordinates $(\nu_1,\nu_2, \ldots, \nu_n)$ with integers or half-integers entries and the type of root system ($A,B,C$ or $D$). Using the relations $\nu_i$ and $\nu_{i+1}$ for any $i$, we deduce the weights under each root $\alpha_i$ and therefore obtain the weighted Dynkin diagram.

\begin{dfn}[residual segment] \label{rs}
The residual segment of type $B,C,D$ associated to the dominant residual point $\nu:= (\nu_1,\nu_2, \ldots, \nu_n) \in \overline{a_{M_1}^{*+}}$ (depending on a fixed irreducible cuspidal representation $\sigma$ of $M_1$) is the expression in coordinates of this dominant residual point in a particular basis of $a_{M_1}^*$ (the basis such that the roots in the Weighted Dynkin diagram are canonically expressed as in \cite{bourbaki}).

It is therefore a decreasing sequence of positive (half)-integers uniquely obtained from a Weighted Dynkin diagram by the aforementioned procedure.

It is uniquely characterized by:
\begin{itemize}
	\item An infinite tuple ($\ldots, 0, n_{\ell+m}, \ldots, n_{\ell}, n_{\ell-1}, \ldots, n_0)$ or ($\ldots, 0, n_{\ell+m}, \ldots, n_{\ell}, n_{\ell-1}, 
\ldots, n_{1/2})$ 
where $n_i$ is the number of times the integer or half-integer value $i$ appears in the sequence. 
	\item The greatest (half)-integer in the sequence, $\ell$, such that $n_{\ell}=1, n_{\ell-1}=2$ if it exists.
	\item the greatest integer, $m$, such that, for any $i \in \left\{1,\ldots, m\right\}$, $n_{\ell+i}=1$ and for any $i>m$, $n_{\ell+i}=0$.
\end{itemize}

This residual segment uniquely determines the weighted Dynkin diagram of type $B,C$ or $D$ from which it originates.

Therefore, the values obtained for the $n_i$'s depend on the Weighted Dynkin diagram  (see the Appendix \ref{WDD1}) one observes the 
following relations: 
\begin{itemize}
	\item Type $B$:  $n_{\ell}=1, n_{\ell-1}=2$, $n_{i-1} = n_i + 1$ or $n_{i-1} = n_i$, $n_0= \frac{n_1 -1}{2} ~\mbox{if} ~ n_1 ~\mbox{is odd}$ 
or $n_0=\frac{n_1}{2} ~\mbox{if} ~ n_1 ~\mbox{is even}$. 
(The regular orbit where $n_i =1$ for all $i\geq 1$ is a special case)
	
	\item Type $C$: 
	$n_{i-1} = n_i + 1$ or $n_{i-1} = n_i$ ; $n_{1/2} = n_{3/2} + 1$, $n_{\ell}=1, n_{\ell-1}=2$
(The regular orbit where $n_i =1$ for all $i\geq 1/2$ is a special case)
	
	\item Type $D$:
\begin{enumerate}
	\item $n_i=1$ for all $i \geq \ell$ and $n_0=1$, $n_i=2$ for all $i \in \left\{2, \ldots, \ell-1\right\}$.
	\item $n_{i-1}= n_i + 1$ or $n_{i-1}=n_i$, $n_0 \geq 2$, $n_0= \left\{
    \begin{array}{ll}
        \frac{n_1}{2} ~\mbox{if} ~ n_1 ~\mbox{is even} \\
        \frac{n_1+1}{2} ~\mbox{if} ~ n_1 ~\mbox{is odd}
    \end{array}  \right\}$
\end{enumerate}
\end{itemize}

It will be denoted $(\underline{n})$.

The residual segment of type $A$ (we say \emph{linear residual segment}, referring to the general \emph{linear} group) is characterized with 
the same three objects, and also corresponds bijectively to a weighted Dynkin diagram of type $A$. Then it is a decreasing sequence of (not necessarily positive) reals and the infinite tuple given above is ($
\ldots, 0, 1, 1, 1, 
\ldots, 1)$, i.e $n_i \leq 1$ for all $i$. It is symmetrical around zero.

We will also abusively say \emph{linear residual segment} for the translated version of a residual segment of type $A$; i.e if it is not symmetrical 
around zero.

\end{dfn}

\emph{We usually do not write the commas to separate the (half)-integers in the sequence.}

The use of the terminology \enquote{segments} is explained through the following example.

\subsubsection*{An example: Bernstein-Zelevinsky's segments}

Consider the weighted Dynkin diagram of type $A$:

$$ \cercle\noteN{\alpha_1}\noteS{2} \th\cercle\noteN{\alpha_2}\noteS{2}\th\points\points\points \th\cercle\noteN{\alpha_n}\noteS{2}$$
\smallskip 

As $\prodscal{\nu}{\alpha_i} = 1$ for all $i$
 $\Longleftrightarrow \nu_i - \nu_{i+1}=1$ for all $i$; the vector of coordinates is therefore a strictly decreasing sequence of real numbers :$(\mathpzc{a},\mathpzc{a}-1,\mathpzc{a}-2,\ldots, \mathpzc{b})$. Notice the specific font used to write linear residual segment.

The group $GL_n$ is an example of reductive group whose root system is of type $A$.
We may now recall the notions of segments for $GL_n$ as defined in \cite{BZI}, and following the treatment in \cite{rodierb}. We fix an irreducible cuspidal representation $\rho$, and denote $\rho(a)= \rho|\text{det}|^a$. The representation $\rho_1\times \rho_2$ denotes the parabolically induced representation from $\rho_1\otimes\rho_2$.

\begin{dfn}[Segment, Linked segments][Bernstein-Zelevinsky; following \cite{rodierb}] \label{segment}
Let $r \vert n$. A segment is an isomorphism class of irreducible cuspidal representations of a group $GL_n$, of the form $\mathcal{S}= \left\{\rho, \rho(1), \rho(2), 
\ldots, \rho(r-1)\right\}$. 
We denote it $\mathcal{S}=[\rho, \rho(r-1)]$.

There is also a notion of intersection and union of two such segments explained in particular in \cite{rodierb}: the intersection of $\mathcal{S}_1$ and $\mathcal{S}_2$ is written $\mathcal{S}_1 \cap \mathcal{S}_2$, the union is written $\mathcal{S}_1 \cup \mathcal{S}_2$.

Let $\mathcal{S}_1= [\rho_1, \rho_1'], \mathcal{S}_2= [\rho_2, \rho_2']$ be two segments. We say $\mathcal{S}_1$ and $\mathcal{S}_2$ are linked if $\mathcal{S}_1 
\not\subseteq \mathcal{S}_2,  
\mathcal{S}_2 \not\subseteq \mathcal{S}_1$ and $\mathcal{S}_1\cup\mathcal{S}_2$ is a segment.
\end{dfn}

Once $\rho$ is fixed, a segment is solely characterized by a string of (half)-integers, it seems therefore natural, in analogy with Bernstein-Zelevinsky's theory, to name any vector $(\nu_1,\ldots, \nu_k)$ corresponding to a dominant residual point and therefore by Proposition \ref{1.13bis} (see also \ref{resbij} and \cite{opdamspec}, Proposition 8.1) to a weighted Dynkin diagram: \textsl{a residual segment}.

Let us define $\mathcal{S}=[\rho(r-1), \rho]$ a sequence of representations twisted by decreasing exponents, and notice the difference with the definition of the segment as given in Bernstein-Zelevinsky where the exponents are increasing. 
The unique irreducible subrepresentation (resp. quotient) of $\rho(r-1)\times \ldots \times \rho$ is denoted $Z(\mathcal{S})$ (resp. $L(\mathcal{S})$). If it is a subrepresentation, it is essentially square-integrable. Often, we denote it $Z(\rho, r-1, 0)$, and more generally $Z(\rho, \mathpzc{a}, \mathpzc{b})$ for $\mathpzc{a}$ and $\mathpzc{b}$ any two real numbers such that $\mathpzc{a} - \mathpzc{b} \in \Z$. In the literature, the generalized Steinberg is also denoted $\text{St}_k(\varrho)$, it is the canonical discrete series associated to the segment $
[\varrho(\frac{k-1}{2}), \ldots, \varrho(\frac{1-k}{2})]$, for an irreducible cuspidal representation $\varrho$. 
Often, $\text{St}_k(\textbf{1})$ will simply be denoted $\text{St}_k$.

This is a general phenomenon, since by Theorem \ref{heir}, for any quasi-split reductive group, we associate to any residual segment an 
essentially square-
integrable (resp. discrete series) representation. The well-known example of the Steinberg representation of $GL_k$
is also characteristic since the Steinberg is the unique irreducible \emph{generic} subquotient in the parabolically induced representation $\varrho(\frac{k-1}{2})\times  \ldots \times \varrho(\frac{1-k}{2})$. 

By Theorems \ref{heir} and \ref{muic}, combined with Rodier's result, if the cuspidal support $\sigma_{\lambda}$, a residual point, is generic, then the 
induced representation is generic and the unique irreducible generic subquotient is essentially square integrable. Therefore, the phenomenon presented here with the Steinberg subquotient, occurs more generally. When the generic representation $
\sigma_{\lambda}$ is 
a dominant residual point, the residual segment corresponding to $\lambda$ characterizes the unique irreducible generic discrete series (resp. 
essentially 
square integrable) subquotient.
\begin{example}
Consider $B_{15}$ for instance (see \ref{WDD1} to understand the relations between the $p_i$'s), with $m=3, p_1=2, p_2=3, p_3=4, p_4=2$:
$$\underbrace{\cercle\noteN{\alpha_1}\noteS{2}\th\cercle\noteN{\alpha_2}\noteS{2}\th \cercle\noteS{2}}_{3}\th
\underbrace{\cercle\noteS{2}\th\cercle\noteS{0}}_{2}\th
\underbrace{\cercle\noteS{2}\th\cercle\noteS{0}\th\cercle\noteS{0}}_{3}\th
\underbrace{\cercle\noteS{2}\th\cercle\noteS{0}\th\cercle\noteS{0}\th\cercle\noteS{0}}_{4}\th \underbrace{\cercle\noteS{2}\th\cercle\noteS{0}}_{2}
\thh\fdroite\cercle\noteN{\alpha_{15}}\noteS{0}$$

\smallskip

We have $\prodscal{\nu}{\alpha_{15}} =\prodscal{\nu}{2e_{15}}=  0$ and therefore $\nu_{15}=0$. $\prodscal{\nu}{\alpha_{14}}=0$ and therefore $\nu_{14} =\nu_{15}=0$ ; $\prodscal{\nu}{\alpha_{13}}=1$, so $\nu_{13} -\nu_{14}= 1$.
Eventually the vector of coordinates corresponding to a dominant residual point, 
$\nu$ is $$(\nu_1,\nu_2,\nu_3,\ldots, \nu_{13},\nu_{14},\nu_{15})=(765 43 322 2111 10 0)$$
\end{example}

%
%
%
%
%

\subsection{Set of Jumps associated to a residual segment}\label{JB}

In a following subsection (\ref{Moeglinandembedding}), we will present certain embeddings of generic discrete series in parabolically induced 
modules. The proof of 
these embeddings necessitates to introduce the definition of the \emph{set of Jumps} associated to a residual segment and therefore, 
transitively, to an 
irreducible generic discrete series. 

These \emph{Jumps} compose a finite set, \emph{set of Jumps}, of (half)-integers $a_i$'s, such that the set of integers $2a_i+1$ is of a given 
parity. In the 
context of classical groups, the latter set (composed of elements of a given parity) coincides with the \emph{Jordan block} defined in 
\cite{MT}. We will also 
use the notion of Jordan block in this subsection.

Let us recall our steps so far.
\\ 
\noindent
If we are given $\pi_0$, an irreducible generic discrete series of $G$, by Proposition \ref{opdamh} and Theorem \ref{heir}, it embeds as a 
subrepresentation in 
$I_{P}^G(\sigma'_{\lambda'})$ for $\sigma'_{\lambda'}$ a dominant residual point.
Further, by the results of \cite{heiermannorbit} (see in particular Proposition 6.2), $\sigma'_{\lambda'}$ corresponds to a distinguished unipotent orbit and therefore a weighted 
Dynkin 
diagram. Once $\Sigma_{\sigma'}$ is fixed (see the Subsection \ref{rootsystem} or the introduction for the Definition of $\Sigma_{\sigma'}$), 
and assuming it 
is irreducible, the type of weighted Dynkin diagram is given. All details will be given in Section \ref{setting2}.
By the previous argumentation (Subsection \ref{sgm}), we associate a residual segment $(\underline{n_{\pi_0}})$ to the irreducible generic 
discrete series $
\pi_0$.

We illustrate these steps in the following example:
\begin{example}[classical groups]
Let $\sigma_{\lambda}$ be in the cuspidal support of a generic discrete series $\pi$ of a classical group (or its variants) $G(n)$, of rank $n$. 
First, assume $
\sigma_{\lambda}: = \rho|.|^{a} \otimes \ldots \rho|.|^{b}\otimes \sigma_c$ where $\rho$ is a unitary cuspidal representation of $GL_k
$, and $
\sigma_c$ a generic cuspidal representation of $G(k'), k'<n$.
Using Bala-Carter theory, since $\lambda$ is a residual point, it is in the $W_{\sigma}$-orbit of a dominant residual point, which corresponds to a 
weighted 
Dynkin diagram of type $B$ (resp. $C, D$) and further the above sequence of exponents $(a, \ldots, b)$ is encoded $(\ell+m, \ldots, 
\ell, \ell -1, \ell-1, 
\ldots, 0): =(\underline{n})$ of type $B$ (resp. $C, D$). The type of weighted diagram only depends on the reducibility point of the induced 
representation of 
$G(k+k'): I^{G(k+k')}(\rho|.|^{s}\otimes \sigma_c)$ as explained in Proposition \ref{1.13bis}.
\end{example}


\subsubsection*{The bijective correspondence between \emph{Residual segments} and \emph{set of Jumps}} \label{SOJ}

Let us start with the bijective map:
$$(\underline{n}) \rightarrow \text{set of Jumps of} ~ (\underline{n})$$
The length of a residual segment is the sum of the multiplicities: $n_{\ell+m} + n_{\ell+m-1} + \ldots n_1 + n_0$.

We first write a length $d$ residual segment $(\underline{n})$ 
\begin{small}
 $((\ell + m),\ldots, \underbrace{\ell}_{n_{\ell} ~\mbox{times}}; \underbrace{\ell-1}
_{n_{\ell-1} ~
\mbox{times}} , \ldots , \underbrace{1}_{n_{1} ~\mbox{times}} \underbrace{0}_{n_0 ~\mbox{times}})$
 \end{small} as a length $2d+1$ (resp. $2d$) 
sequence of 
exponents (betokening an unramified character of the corresponding classical group, e.g. to $B_d$ corresponds $SO_{2d+1}$)

\begin{flushleft}
\begin{multline*}
\begin{small}
((\ell + m),\ldots, \underbrace{\ell}_{n_{\ell} ~\mbox{times}}; \underbrace{\ell-1}_{n_{\ell-1} ~\mbox{times}} , \ldots , \underbrace{1}_{n_{1} ~
\mbox{times}} 
\underbrace{0}_{n_0 ~\mbox{times}}, 0 ,
 \underbrace{0}_{n_0 ~\mbox{times}} \underbrace{-1}_{n_{1} ~\mbox{times}} \ldots \underbrace{-\ell}_{n_{\ell}~\mbox{times}}, \ldots, -(\ell+m))
\end{small}
\end{multline*}
\end{flushleft}
for type $B_d$ only, we add the central zero.
It is a decreasing sequence of $2d+1$ (for type $B_d$) or $2d$ (for type $C_d,D_d$) (half)-integers; from the previous Subsection (\ref{sgm}), the reader has noticed that for $C_d, n_0=0$.

Then, we decompose this decreasing sequence as a multiset of $2n_0 +1$ (resp. $2n_1$ for type $D_d$ or $2n_{1/2}$ for type $C_d$) (it is 
the number of 
elements in the Jordan block) linear residual segments symmetrical around zero:\par\smallskip
\vbox{$\left\{(a_1, a_1 -1, \ldots, 0, \ldots, -a_1); (a_2, a_2 -1, \ldots, 0, \ldots, -a_2); \ldots\right.$
\par\smallskip\hfill
$\left. \ldots;(a_{2n_0 +1}, a_{2n_0 +1}-1, \ldots, 0, \ldots, -a_{2n_0 +1}) \right\}$}
\par\smallskip\noindent
(resp.\par\smallskip
\vbox{$\left\{(a_1, a_1-1, \ldots, 1/2, -1/2, \ldots, -a_1); (a_2, a_2 -1, \ldots, 1/2, -1/2, \ldots, -a_2); \ldots \right.$\par\smallskip\hfill
$\left.\ldots;(a_{2n_{1/2}}, a_{2n_{1/2}}-1, \ldots, 1/2, -1/2, \ldots, -a_{2n_{1/2}}) \right\}$}
\par\smallskip\noindent
where $a_1$ is the largest (half)-integer in the above decreasing sequence, $a_2$ is the largest (half)-integer with multiplicity 2, and in general $a_i$ 
is the largest 
(half)-integer with multiplicity $i$.

\begin{dfn}[set of Jumps]
The \emph{set of Jumps} is the set: $\left\{a_1, \ldots, a_{2n_0 +1} \right\}$ (resp. $\left\{a_1, \ldots, a_{2n_{1/2}} \right\}$). As one notices, the terminology comes from the observation that multiplicities at each jump increases by one: $n_{a_{i+1}} = n_{a_i} + 1$.
\end{dfn}

Let us make a parallel for the reader familiar with Moeglin-Tadi\'c terminology for classical groups \cite{MT} (see also Tadi\'c's notes 
\cite{tadicnotes1} and 
\cite{tadicnotes2} for an introductory summary of these notions). In such context the Jordan block of the irreducible discrete series $\pi$ 
associated to the 
residual segment $(\underline{n})$ (denoted $\text{Jord}_{\pi}$) is constituted of the integers: $$\left\{2a_1+1, 2a_2+1; \ldots, 2a_{2n_0 +1}+1 
\right\}$$ (resp. $
\left\{2a_1+1, 2a_2+1; \ldots, 2a_{2n_{1/2}}+1 \right\}$). This is not a complete characterization of a Jordan block: for a correct use of the definition of Jordan block, we should also fix a self-dual irreducible cuspidal representation $\rho$ of a general linear group and an irreducible cuspidal representation $\sigma_c$ of a smaller classical group. We \emph{abusively} use the terminology \emph{Jordan block} to define one partition but such partition is only one of the constituents of the Jordan block as defined in \cite{MT}. Clearly the Jordan block is a set of distinct odd (resp. even) integers. According to \cite{MT}, the following condition should also be satisfied: 
$2d+1 = \sum_i(2a_i +1)$ for type $B$   (resp. $2d = \sum_i (2a_i+1)$ for type $C$). 

\vspace{0.3cm}
Moreover, we are now going to explain there is a canonical way to obtain for a given type ($A, B, C$, or $D$) and a fixed length $d$ all 
distinguished nilpotent 
orbits, thus all Weighted Dynkin diagrams and therefore all residual segments of these given type and length.

This is given by Bala-Carter theory (see the Appendix \ref{BCT} and in particular the Theorem \ref{8.2.14}). First, one should partition 
the integer $2d
+1$ (resp. $2d$) into distinct odd (resp. even) integers (given $2d+1$, or $2d$ there is a finite number of such partitions). Each partition 
corresponds to a 
distinguished orbit and further to a dominant residual point, hence a residual segment.

In fact, each partition corresponds to a Jordan block of an irreducible discrete series $\pi$ (whose associated residual segment is $(\underline{n_{\pi}})$). Let 
us detail the three cases ($B,C$ and $D$). 

Let us finally illustrate the following correspondence:
$$\text{Jord}_{\pi} \rightarrow \text{set of Jumps} ~ (\underline{n_{\pi}}) \rightarrow (\underline{n_{\pi}})$$

\begin{itemize}
\item In case of $B_d$, the set Jumps of $({n_{\pi}})$ derives easily from the choice of \emph{one} partition of $2d+1$ in distinct odd integers: 
$\text{Jord}
_{\pi} = \left\{2a_1+1, 2a_2+1, \ldots, 2a_t+1 \right\}$. 
Then Jumps of $({n_{\pi}}) = \left\{a_1, a_2, \ldots, a_t \right\}$. 

Once this set of Jumps identified, one writes the corresponding symmetrical around zero linear segments $(a_i, \ldots, -a_i)$'s and by 
combining and 
reordering them, form a decreasing sequence of integers of length $2d+1$. 

This length $2d+1$ sequence is symmetrical around zero, with a length $d$ sequence of non-negative elements, a central zero, and the 
symmetrical sequence of 
negative elements. The length $d$ sequence of positive elements is the residual segment $(\underline{n})$.

\item Again the case of $C_d$ (by Theorem \ref{8.2.14} in Appendix \ref{BCT}) $2d$ is partitioned into distinct even integers, each 
partition 
corresponds to a distinguished orbit and further to a dominant residual point, hence a residual segment.

The correspondence is the following: to the Jordan block of a generic discrete series, $\pi$ and its associated residual segment $\underline{n}
_{\pi}$ : 

$\text{Jord}_{\pi} = \left\{2a_1+1, 2a_2+1, \ldots, 2a_t+1 \right\}$, for each $a_i$, one writes $(a_i, a_i-1, \ldots, 1/2, -1/2, \ldots -a_i)$. One 
takes all elements 
in all these sequences, reorder them to get a $2d$ decreasing sequence of half-integers. The length $d$ sequence of positive half-integers 
corresponds to 
residual segment $(\underline{n})$ of type $C_d$.

\item 
In case of $D_d$, let $\text{Jord}_{\pi} = \left\{2a_1+1, 2a_2+1, \ldots, 2a_t+1 \right\}$ be the Jordan block of a generic discrete series, $\pi$; 
then write the 
corresponding linear segments $(a_i, \ldots, -a_i)$'s, with all these residual segments, form a decreasing sequence of integers of length $2d$. 
This length 
$2d$ sequence is symmetrical around zero. The length $d$ sequence of positive elements in chosen to form the residual segment $
(\underline{n})$.

\end{itemize}

\begin{example}[$B_{14}$]
Let us consider one partition of 2.14+1 into distinct odd integers: $\left\{11,9,5,3,1 \right\}$. 

For each odd integer in this partition, write it as $2a_i+1$ and write the corresponding linear residual segments $(a_i, \ldots, -a_i)$:
$$543210-1~-2~-3~-4~-5$$
$$43210-1~-2~-3~-4$$
$$210-1~-2$$
$$10-1$$
$$ 0 $$ 

Re-assembling, we get 
$$54433222111100;0;0~0~-1~-1~-1~-1~-2~-2~-2 ~-3~-3~-4~-4~-5$$
Then, the corresponding residual segment of length 14 (29=2.14+1) is:
54433222111100.
\end{example}

\begin{example}[$C_9$]
Then $2d_i'$ is 18, and we decompose 18 into distinct even integers: 18; 14+4; 12+4+2; 16+2; 8+6+4, 12+6, 10+8.
To each of these partitions corresponds the Weyl group orbit of a residual point and therefore a residual segment. 
The regular orbit (since the exponents of the associated residual segment form a regular character of the torus) correspond to 18. It is simply 
$$(17/2,15/2,13/2, \ldots ,1/2)$$
The half-integer 17/2 is such that 2(17/2) + 1 =18.

Let us consider the third partition, 12+4+2, : 12= 2(11/2) + 1; 4= 2(3/2) + 1; 2 = 2(1/2) + 1. Each even integer gives a strictly decreasing sequence of half-integers (11/2,9/2,7/2,5/2,3/2,1/2); (3/2,1/2); (1/2).
Finally, we reorder the nine half-integers obtained as a decreasing sequence : $$(11/2,9/2,7/2,5/2,3/2,3/2,1/2,1/2,1/2)$$


\end{example}
\begin{rmk}\label{rmkrs}
Once given a residual segment, $(\underline{n})$, and its corresponding set of Jumps $a_1 > a_2 > \ldots > a_n$, one observes that for any $i$, $(a_i, \ldots ,-a_{i+1})(\underline{n_i})$ is in the $W_{\sigma}$-orbit of this residual segment, where $(a_i, \ldots ,-a_{i+1})$ is a linear residual segment and $(\underline{n_i})$ a residual segment of the same type as ($\underline{n}$).

\vspace{0.3cm}

Therefore, a set of asymmetrical linear segments $(a_i, \ldots ,-a_{i+1})$ along with the smallest residual segment of a given type (e.g $(100)$ for type $B$, resp. $(3/2, 1/2, 1/2)$ for type $C$) \emph{or} a linear segments $(a_1, a_1 -1, \ldots 0)$ (resp. $(a_1, a_1 -1, \ldots 1/2)$ for type $C$) is in the $W_{\sigma}$-orbit of the residual segment $(\underline{n})$.

Clearly, a set of linear \emph{symmetrical} segments cannot be in the $W_{\sigma}$-orbit of the residual segment $(\underline{n})$.
\end{rmk}

\subsection{Application of the theory of residual segments: reformulation of our setting} \label{setting2}

\subsubsection{Reformulation of our setting}

Let us come back to our setting (recalled at the beginning of the Section \ref{setting3}).

Let $M_1$ be a Levi subgroup of $G$ and $\sigma$ a generic irreducible unitary cuspidal representation of $M_1$. 
Put $\Sigma _{\sigma }=\{\alpha\in \Sigma_{red}(A_{M_1})\vert \mu^{M_{1,\alpha }}(\sigma )=0\}$ (resp. $\Sigma _{\sigma }^M=\{\alpha\in \Sigma_{red}^M(A_{M_1})\vert \mu^{(M_1)_{\alpha}}(\sigma )=0\}$).
The set $\Sigma_{\sigma }$ is a root system in a subspace of $(a_{M_1}^G)^*$ (resp. $(a_{M_1}^M)^*$)(cf. \cite{silbergerSD} 3.5).

Suppose that the irreducible components of $\Sigma _{\sigma}$ are all of type $A$, $B$, $C$ or $D$. First assume $\Sigma_{\sigma}$ is irreducible and let us denote $\mathcal{T}$ its type, and $\Delta_{\sigma}:= \left\{\alpha_1, \ldots, 
\alpha_d \right\}$ the basis of $\Sigma_{\sigma}$ (following our choice of basis for the root system of $G$). 

We will consider maximal standard Levi subgroups of $G$, $M 
\supset M_1$, corresponding to sets $\Delta-\{\underline{\alpha_k}\}$, for a simple root $\underline{\alpha_k} \in \Delta$ (here we use the notation $\underline{\alpha_k}$ to avoid confusion with the roots in $\Delta_{\sigma}$).
Since $M \supseteq M_1 = M_{\Theta}$, $\Theta \subset \Delta - \left\{\underline{\alpha_k} \right\}$, or in other words, if we denote $\alpha_k$ the projection of $\underline{\alpha_k}$ on the orthogonal of $\Theta$ in $a_{M_1}^*$ then $\alpha_k \in \Sigma_{\Theta}$ (see the Appendix \ref{lab} for precise definition and analysis of this set), and even more then $\alpha_k \in \Sigma_{\sigma}$. If $\underline{\alpha_k}$ is not a extremal root of the Dynkin diagram of $G$, $\Sigma^M$ decomposes in two disjoint components.

\begin{rmk}
The careful reader has already noticed that it is possible that $\Sigma^M$ breaks into \emph{three} components rather than two: in the context $\Sigma$ is of type $D_n$ and $\underline{\alpha_k}$ in the above notation is the simple root $\alpha_{n-2} \in \Delta$.
In this remark and in the Appendix \ref{lab}, we rather use the notation $\alpha_i$ to denote the simple roots in $\Sigma$; and $\overline{\alpha_i}$ their projections on the orthogonal to $\Theta$.
\noindent
By the calculations done in \cite{PrSR}, to obtain any root system in $\Sigma_{\Theta}$ for $\Sigma$ of type $D_n$, we need either $\alpha_{n-1}$ \emph{and} $\alpha_n$ in $\Delta$ to be in $\Theta$; or only one of them in $\Theta$.
\noindent
In case both of them are in $\Theta$ but $\alpha_{n-2}$ is not, we are reduced to the case of $B_{n-2}$.
Then $\overline{\alpha_{n-2}}= e_{n-2}$ would be the last root in $\Sigma_{\sigma}$. Therefore, if $M= M_{\Delta - \alpha_{n-2}}$, and therefore $\Sigma_{\sigma}^M$ is irreducible; we treat the conjecture for this case in the Subsection \ref{SsM}.
\noindent
In the case only one of them (without loss of generality $\alpha_{n-1}$) is in $\Theta$, the projection $\overline{\alpha_{n-2}} = e_{n-2} - \frac{e_n+e_{n-1}}{2}$ has squared norm equal to 3/2.
This forbids this root to belong to $\Sigma_{\Theta}$ and therefore to be the root $\underline{\alpha_k}$ such that $M$ is $M_{\Delta - \underline{\alpha_k}}$. Indeed as explained in the very beginning of Section \ref{setting2}, since $M_1= M_{\Theta} \subseteq M$ the root $\underline{\alpha_k}$ which \emph{is not a root in $M$} is not either a root in $\Theta$.
\end{rmk}

Then, $\Sigma _{\sigma }^M$ is a disjoint union of two irreducible components $\Sigma_{\sigma,1 }^M \bigcup \Sigma _{\sigma,2 }^M$ of type $A$ and $\mathcal{T}$, one of which may be empty (if we remove extremal roots from the Dynkin 
diagram). If we remove $\overline{\alpha_n}$, $\Sigma_{\sigma,2}^M$ is empty, and $\Sigma_{\sigma,1}^M$ is of type $A$, whereas if we remove $\overline{\alpha_1}$, $\Sigma_{\sigma,2}^M$ is of type $\mathcal{T}$ and $\Sigma_{\sigma,1}^M$ is empty.

\vspace{0,5cm}

Else we assume $\Sigma_{\sigma}$ is not irreducible but a disjoint union of irreducible components or empty components $\Sigma_{\sigma,i}$ 
for $i=1,\ldots, r
$ of type $A$, $B$, $C$ or $D$: $\Sigma_{\sigma }= \bigcup_i\Sigma_{\sigma,i}$.
Then, the basis of $\Sigma_{\sigma}$ is $$\Delta_{\sigma}:= \left\{ \alpha_{1,1}, \ldots, \alpha_{1,d_1};\alpha_{2,1}, \ldots, \alpha_{2,d_2}, 
\ldots,\alpha_{i,1}, 
\ldots, \alpha_{i,d_i}, \ldots,  \alpha_{r,1}, \ldots, \alpha_{r,d_r}\right\}$$

Again, we will consider maximal standard Levi subgroup of $G$, $M \supset M_1$, 
corresponding to sets $\Delta-\{\overline{\alpha_k}\}$.

Then, for an index $j \in \left\{1,\ldots, r \right\}$, $\Sigma _{\sigma,j}^M$ is a disjoint union of two irreducible components $\Sigma _{\sigma,j_1}^M\bigcup\Sigma _{\sigma,j_2}^M$ of type 
$A$ and $\mathcal{T}$, one of which may be empty (if $\overline{\alpha_k}$ is an \enquote{extremal} root of the Dynkin diagram of $G$). 
If we remove the last simple root, $\overline{\alpha_n}$, of the Dynkin diagram, $\Sigma_{\sigma,j_2}^M$ is empty, and $\Sigma_{\sigma,j_1}^M$ is of type $A$, whereas if we remove $\alpha_1$, $\Sigma_{\sigma,j_2}^M$ is of type $\mathcal{T}$ and $\Sigma_{\sigma,j_1}^M$ is empty.
Therefore, it will be enough to prove our results and statements in the case of $\Sigma_{\sigma}$ irreducible; since in case of reducibility, without loss of generality, we choose a component $\Sigma_{\sigma,j}$ and the same reasonings apply.

Now, in our setting (see the beginning of the Section \ref{setting3}), $\sigma_{\nu}$ is a residual point for $\mu^M$. 
Recall $\Sigma_{\sigma}$ is of rank $d= d_1 + d_2$. Therefore, the residual point is in the cuspidal support of the generic discrete series $\tau$ if and only if (applying 
Proposition \ref{1.13bis} above):
$rk(\Sigma _{\sigma }^M) = d_1-1 + d_2$.

We write $\Sigma _{\sigma }^M := A_{d_1-1}\bigcup\mathcal{T}_{d_2}$
and $\nu$ corresponds to residual segments $(\nu _{1,1},\dots ,\nu _{1, d_1})$ and $(\nu _{2,1},\dots ,\nu _{2, d_2})$.

Let us assume that the representation $\sigma_{\lambda}$ is in the cuspidal support of the essentially square integrable representation of $M
$, $\tau_{s
\tilde{\alpha}}$, where $\lambda = \nu + s\tilde{\alpha}$. 
We add the twist $s\tilde{\alpha}$ on the linear part (i.e corresponding to $A_{d_1-1}$), and therefore $(\nu _{2,1},\dots ,\nu _{2, d_2})$ is left 
unchanged 
and is thus $(\lambda _{2,1},\dots ,\lambda _{2, d_2})$, whereas $(\nu _{1,1},\dots ,\nu _{1, d_1})$ becomes $(\lambda _{1,1},\dots ,
\lambda_{1, d_1})$. 
\noindent

Then, we need to obtain from $(\lambda _{1,1},\dots ,\lambda_{1, d_1})(\lambda _{2,1},\dots ,\lambda _{2, d_2})$ a residual segment of length 
$d$ and type $\mathcal{T}$. Indeed, it is the only option to insure $\sigma_{\lambda}$ is a residual point (applying Proposition \ref{1.13bis}) for $\mu^G$, in particular, 
since $d= d_1+d_2$ (and therefore writing $\Sigma_{\sigma }= A_{d_1-1}\bigcup\mathcal{T}_{d_2}$ does not satisfy the requirement of Proposition \ref{1.13bis}).

\subsubsection{Cuspidal strings}
Assume we remove a non-extremal simple root of the Dynkin diagram, the parameter $\lambda$ in the cuspidal support is therefore constituted of a couple of residual 
segments, one of 
which is a linear residual segment: $(\mathpzc{a},\ldots,\mathpzc{b})$, and the other is denoted $(\underline{n})$. It will be convenient to define the 
cuspidal support to 
be given by the tuple $(\mathpzc{a},\mathpzc{b},\underline{n})$ where $\underline{n}$ is a tuple $(\ldots, 0, n_{\ell+m}, \ldots, n_{\ell},n_{\ell-1}, \ldots, n_1, 
n_0)$ characterization uniquely the residual segment. We define:

\begin{dfn}[cuspidal string]
Given two residual segments, strings of integers (or half-integers): $(\mathpzc{a},\ldots,\mathpzc{b})(\underline{n})$. 
The tuple $(\mathpzc{a},\mathpzc{b},\underline{n})$ where $\underline{n}$ is the $(\ell+m+1)$-tuple 
$$(n_{\ell+m},\ldots, n_{\ell},n_{\ell-1}, \ldots, n_1, n_0)$$
is named a 
cuspidal string.
\end{dfn}

Recall $W_{\sigma}$ is the Weyl group of the root system $\Sigma_{\sigma}$.

\begin{dfn}[$W_{\sigma}$-cuspidal string]
Given a tuple $(\mathpzc{a},\mathpzc{b},\underline{n})$ where $\underline{n}$ is the $(\ell+m+1)$-tuple 
$(n_{\ell+m},\ldots,n_{\ell}, n_{\ell-1}, \ldots, n_1, n_0)$, the set of 
all three-tuples $(\mathpzc{a}',\mathpzc{b}',\underline{n'})$ where $\underline{n'}$ is a $(\ell'+m'+1)$-tuple $(n'_{\ell'+m'}, \ldots, n'_{\ell'}, n'_{\ell'-1}, \ldots, n'_1, 
n'_0)$ in the $W_{\sigma}$ orbit of $(\mathpzc{a},\mathpzc{b},\underline{n})$ is called $W_{\sigma}$-cuspidal string. 
\end{dfn}

\begin{rmk}
These definitions can be extended to include the case of $t$ linear residual segments (i.e of type $A$) :$(\mathpzc{a}_1,\ldots,\mathpzc{b}_1)
(\mathpzc{a}_2,\ldots,
\mathpzc{b}_2)\ldots (\mathpzc{a}_t,\ldots,\mathpzc{b}_t)$ and a residual segment $(\underline{n})$ of type $B,C$ or $D$, then the parameter in the cuspidal 
support will be 
denoted $(\mathpzc{a}_1,\mathpzc{b}_1;\mathpzc{a}_2,\mathpzc{b}_2;\ldots ;\mathpzc{a}_t,\mathpzc{b}_t,\underline{n})$. 
\end{rmk}
\subsection{Application to the case of classical groups} \label{classicalgroups}

We illustrate in the following subsection how these definitions naturally appear in the context of classical groups.

\subsubsection{Unramified principal series}\label{ps}

Let $\tau$ be a generic discrete series of $M= M_L\times M_c$, the maximal Levi subgroup in a classical group $G$, $M_L \subset P_L$ is a linear group and $M_c \subset P_c$ is a smaller classical group. It is a tensor product of an 
essentially square integrable representation of a linear group and an irreducible generic discrete series $\pi$ of a smaller classical group of the same type as $G$.
$$\tau : = St_{d_1}|.|^s \otimes \pi,  ~ \text{with} ~ s=\frac{\mathpzc{a} + \mathpzc{b}}{2}$$
\footnote{It is worth noting that in the case of the Siegel parabolic for classical groups, $I_P(\tau_{s\tilde{\alpha}}) $ is $\Ind_P^G(|\det|^{s/2}\tau)$ see p7, \cite{shahiditwisted}}.
Further, let us assume $(P_1,\sigma,\lambda):= (P_0, \textbf{1}, \lambda)$.
The twisted Steinberg is the unique subrepresentation in $I_{P_{0,L}}^{M_L}(\mathpzc{a},\ldots,\mathpzc{b})$, whereas $\pi \hookrightarrow I_{P_{0,c}}^{M_c}
(\underline{n})$.

Therefore, $$I_P^G(\tau) \hookrightarrow  I_{P_c\times P_L}^G(I_{P_{0,L}}^{M_L}(\mathpzc{a},\ldots,\mathpzc{b})I_{P_{0,c}}^{M_c}(\underline{n})) \cong 
I_{P_0}^G((\mathpzc{a},
\ldots,\mathpzc{b})(\underline{n}))$$

\subsubsection{The general case}\label{generalcase}

Assume $\tau$ is an irreducible generic essentially square integrable representation of a maximal Levi subgroup $M$ of a classical group of rank $\sum_{i=1}^rd_i.\dim(\sigma_i)+ k$. Then $\tau : = 
St_{d_1}
(\sigma_1)|.|^s \otimes \pi$, with $s=\frac{\mathpzc{a} + \mathpzc{b}}{2}$.

We study the cuspidal support of the generic (essentially) square integrable representations $St_{d_1}(\sigma_1)|.|^s$ and $\pi$.

By Proposition \ref{opdamh}, $\pi \hookrightarrow I_{P_{1,c}}^{M_c}(\sigma^c_{\nu_c})$ such that: 
$$M_{1,c}= \underbrace{GL_{k_2}\times \ldots \times GL_{k_2}}_{d_2 ~ \mbox{times}} \times \ldots \times \underbrace{GL_{k_r}\times \ldots \times 
GL_{k_r}}_{d_r ~ 
\mbox{times}}\times G(k)$$
where $G(k)$ is a semi-simple group of absolute rank $k$ of the same type as $G$.

We write the cuspidal representation $\sigma^c := \sigma_2\otimes\ldots \sigma_2\otimes\ldots\otimes\sigma_r\otimes\ldots \sigma_r\otimes
\sigma_c$ of 
$M_{1,c}$ and assume the inertial classes of the representations of $GL_{k_i}$, $\sigma_i$, are mutually distinct and $\sigma_i \cong 
\sigma_i^{\vee}$ if $
\sigma_i, \sigma_i^{\vee}$ are in the same inertial orbit.

The residual point $\nu_c$ is dominant: $\nu_c \in ((a_{M_1}^M)^*+$.
Applying Proposition \ref{1.13bis} below with $\nu_c$ and the root system $\Sigma_{\sigma}^M$, we have:
$\nu_c := (\nu_2,\ldots,\nu_r)$ where each $\nu_i$ for $i \in \left\{2, \ldots, r\right\}$ is a residual point, corresponding to a residual segment of type $B_{d_i},C_{d_i}, D_{d_i}
$.

Further,
$$\text{St}_{d_1}(\sigma_1)|.|^s \hookrightarrow I_{P_{1,L}}^{M_L}(\sigma_1,\lambda_L) \cong I_{P_{1,L}}^{M_L}(\sigma_1|.|^{\mathpzc{a}}\otimes
\sigma_1|.|
^{\mathpzc{a}-1}\ldots \sigma_1|.|^{\mathpzc{b}})$$ 
where $\lambda_L$ is the residual segment of type $A$: $(\mathpzc{a},\mathpzc{a}-1,\ldots,\mathpzc{b})$, and $M_L$ is the linear part of Levi subgroup $M$. Such that eventually: 
$\sigma= \sigma_1\otimes\sigma_1\ldots \sigma_1\otimes \sigma_2\otimes\ldots\sigma_2\otimes\ldots \otimes\sigma_r\otimes\ldots \sigma_r\otimes\sigma_c$
And $\sigma_{\lambda}$ can be rewritten:
\begin{multline}
\sigma_1|.|^{\mathpzc{a}}\otimes\sigma_1|.|^{\mathpzc{a}-1}\ldots \sigma_1|.|^{\mathpzc{b}}\otimes\underbrace{\sigma_2|.|^{\ell_2}\ldots \sigma_2|.|^{\ell_2}}
_{n_{\ell_2,2} ~
\mbox{times}} \ldots \underbrace{\sigma_2|.|^{0}\ldots \otimes\sigma_2|.|^{0}}_{n_{0,2} ~\mbox{times}}\ldots\\
 \underbrace{\sigma_r|.|^{\ell_r}\ldots \otimes\sigma_r|.|^{\ell_r}}_{n_{\ell_r,r} ~\mbox{times}} \ldots \underbrace{\otimes\sigma_r|.|^{0}\ldots
\otimes\sigma_r|.|
^{0}}_{n_{0,r} ~\mbox{times}}\otimes\sigma_c
\end{multline}

The character $\nu$, representation of $M_1$, can be splitted in two parts $\nu_1$ and $\underline{\nu}=(\nu_2, \ldots, \nu_r)$, residual points, giving the discrete series denoted $\text{St}_{d_1}(\sigma_1)$ in $I_{P_{1,L}}^{M_L}(\sigma_1)$ and $\pi$ in $I_{P_{1,c}}^{M_c}(\sigma_c,\underline{\nu})$. By a simple computation, it can be shown that the twist $s\tilde{\alpha}$ will be added on the 'linear part' of the representation and leaves the semi-simple part (classical part) invariant. 

Namely $\nu$ is given by a vector $(\nu_1=0,\nu_2,\ldots,\nu_r)$ and we add the twist $s\tilde{\alpha}$ on the first element to get the vector:
$\lambda= (\lambda_1,\lambda_2,\ldots,\lambda_r)$ where each $\lambda_i$ is a residual segment $(\underline{n_i})$ associated to the 
subsystem $
\Sigma_{\sigma,i}$. 

To use the bijection between $W_{\sigma}$ orbits of residual points and weighted Dynkin diagrams, one needs to use a certain root system and its 
associated Weyl 
group. Then $\lambda$ is a tuple of $r$ residual segments of different types:
$\left\{(\underline{n}_i)\right\},i \in \left\{1, \ldots, r\right\}$. If the parameter $\lambda$ is written as a $r$-tuple: $(\lambda_1, \ldots, \lambda_r)
$, it is 
dominant if and only if each $\lambda_i$ is dominant with respect to the subsystem $\Sigma_{\sigma, i}$.

\vspace{0,5cm}

We have not yet used the \emph{genericity} property of the cuspidal support. This is where we use Proposition \ref{1.13bis}. The generic 
representation $\sigma_c$ and the reducibility point of the representation induced from $\sigma_i|.|^s\otimes \sigma_c$ determine the type of the residual 
segment $
(\underline{n}_i)$ obtained.

\section[Characterization of the generic subquotient]{Characterization of the unique irreducible generic subquotient in the standard module}

\subsection{}
Let us first outline the results presented in this section. Let us assume that the irreducible generic subquotient in the standard module is not discrete series. We characterize the Langlands parameter of this unique irreducible non-square integrable subquotient using an order on Langlands parameters given in Lemma \ref{bw} below: more precisely, in Theorem \ref{minLP=irr}, we prove this unique irreducible generic subquotient is identified by its Langlands parameter being minimal for this order. 

We then compare Langlands parameters in the Subsection \ref{lrs}, and along those results and Theorem \ref{minLP=irr}, we will prove a lemma (Lemma \ref{zelevinsky}) in the vein of Zelevinsky's Theorem at the end of this Section.

Finally, before entering the next section we need to come back on the depiction of the intertwining operators used in our context. This subsection \ref{IO} on intertwining operators also contains a lemma (Lemma \ref{nngenerickk}) which is crucial in the proof of main Theorem \ref{conditionsonlambda} in the following Section.

\subsection{An order on Langlands' parameters} \label{orderonL}

Using Langlands' classification (see Theorem \ref{LC}) and the \emph{Standard module conjecture} (see Theorem \ref{SMC}), we can characterize the unique irreducible generic non-square integrable 
subquotient, denoted $I_{P'}^G(\tau'_{\nu'})$. In particular, on a given cuspidal support, we can characterize the form of the Langlands' parameter $\nu'$. We introduce 
the necessary tools and results regarding this theory in this subsection.

To study subquotients in the standard module induced from a maximal parabolic subgroup $P$, $I_P^G(\tau_{s\tilde{\alpha}})$, we will use the following well-known lemma from \cite{borelwallach}:

Let us recall their definition of the order:
\begin{dfn}[order]\label{order}
$\lambda_{\mu} \leq \lambda_{\pi}$ if $\lambda_{\pi} - \lambda_{\mu} = \sum_i x_i\alpha_i$ for simple roots $\alpha_i$ in $a_0^*$ and $x_i 
\geq 0$.
\end{dfn}

\begin{lemma}[Borel-Wallach, 2.13 in Chapter XI of \cite{borelwallach}] \label{bw}
Let $(P,\sigma, \lambda_{\pi})$ be Langlands data. If $\mu$ is a constituent of $I_P^G(\sigma_{\lambda_{\pi}})$ the standard module, and if $\pi = J(P,\sigma, 
\lambda_{\pi})$ is the Langlands quotient, then $\lambda_{\mu} \leq \lambda_{\pi}$, and equality occurs if and only if $\mu$ is $J(P,\sigma, \lambda_{\pi})$.

We will write this order on Langlands parameters:
 $$\lambda_{\mu} \bw \lambda_{\pi}$$
\end{lemma}

\begin{lemma} \label{rmkorder}
Let $\nu = \sum_{i=1}^n a_i e_i$ in the canonical basis $\left\{e_i \right\}_i$ of $\mathbb{R}^n$.
$0 \bw \nu$ if and only if $\sum_{i=1}^k a_i \geq 0$ for any $k$ in non-$D_n$ cases.
In the case of $D_n$, one needs to specify $\sum_{i=1}^k a_i \geq 0$ for any $k\leq n-1$, $a_{n-1} \geq -a_n$ and $a_{n-1} \geq a_n$.
\end{lemma}

\begin{proof}
From the expression $\nu = \sum_{i=1}^n a_i e_i$ in the canonical basis $\left\{e_i \right\}_i$ of $\mathbb{R}^n$, we can recover an 
expression of $\nu$ in 
the canonical basis of the Lie algebra $a_0^*$: $\nu= \sum_{i=1}^n x_i\alpha_i$.

Let us make explicit $\nu= \sum_i x_i\alpha_i$: 
\begin{small}
\begin{flushleft}
$$\nu = \sum_{i=1}^{n-1} x_i(e_i - e_{i+1}) + x_n\alpha_n = x_1(e_1-e_2) + x_2(e_2 - e_3) + \ldots + x_{n-1}(e_{n-1} - e_n)=$$
$$\nu = \sum_{i=1}^n a_i e_i = x_1e_1 + (x_2 - x_1)e_2 + (x_3 - x_2)e_3 + \ldots + \begin{cases}
		    (x_{n-1}-x_{n-2})e_{n-1} -x_{n-1}e_n & \text{for $A_{n-1}$} \\
        (x_{n-1} + x_n)e_{n-1} + (x_n-x_{n-1})e_n   & \text{for $D_n$} \\
        (x_{n-1} - x_{n-2})e_{n-1}+ (x_n- x_{n-1})e_n & \text{for $B_n$}\\
				(x_{n-1} - x_{n-2})e_{n-1}+ (2x_n-x_{n-1})e_n & \mbox{for $C_n$} \\
    \end{cases} $$
\end{flushleft}
\end{small}
$\nu= \sum_{i=1}^n x_i\alpha_i \geq 0 \Leftrightarrow x_i \geq 0 ~ \forall i$

From above $x_1 = a_1, x_2 - x_1 = a_2 \Leftrightarrow x_2 = a_1 + a_2, \ldots$
We have : $x_k = \sum_{i=1}^k a_i ~ \forall k$ except for root system of type $D_n$, where for index $n-1$ and $n$,
$2x_n = \sum_{i=1}^{n-1}a_i + a_n$ and $2x_{n-1}= \sum_{i=1}^{n-1}a_i - a_n$, and for $C_n$ where $2x_n= \sum_{i=1}^na_i$.

Notice that for $A_{n-1}$, $x_{n-1}= \sum_{i=1}^{n-1}a_i$ and $a_n = -x_{n-1}$ such that $\sum_{i=1}^{n}a_i = 0$.

Therefore, $0 \bw \nu$ if and only if $\sum_{i=1}^k a_i \geq 0$ for any $k$ in non-$D_n$ cases.
In the case of $D_n$, one needs to specify $\sum_{i=1}^k a_i \geq 0$ for any $k\leq n-1$, $\sum_{i=1}^{n-1}a_i \geq -a_n$ and $\sum_{i=1}^{n-1}a_i \geq a_n$.
\end{proof}

Our next result, Theorem \ref{minLP=irr}, will be used in the course of the proof of the Generalized Injectivity Conjecture for non-discrete series subquotients presented in the Sections \ref{gicds} and \ref{gicnnds}. We use the notations of Section \ref{io1}. We will need the following theorem:

\begin{thm} \label{muic}[Theorem 2.2 of \cite{smc}]

Let $P=MU$ be a $F$-standard parabolic subgroup of $G$ and $\sigma$ an irreducible generic cuspidal representation of $M$. If the induced 
representation $I_P^G(\sigma)$ has a subquotient which lies in the discrete series of $G$ (resp. is tempered) then the unique irreducible generic sub-quotient 
of $I_P^G(\sigma)$ lies in the discrete series of $G$ (resp. is tempered).
\end{thm}

\begin{thm}\label{minLP=irr} 

Let $I_P^G(\tau_{\nu})$ be a generic standard module and $(P',\tau',\nu')$ the Langlands data of its unique irreducible generic subquotient.

If $(P'',\tau'',\nu'')$ is the Langlands data of any other irreducible subquotient, then $\nu' \bw \nu''$. The inequality
is strict if the standard module $I_{P''}^G(\tau''_{\nu''})$ is generic.

In other words, $\nu'$ is the smallest Langlands parameter for the order (defined in  Lemma \ref{bw}) among the Langlands parameters of standard modules having $(\sigma,\lambda)$ as cuspidal support. 
\end{thm}

\begin{proof}
First using the result of Heiermann-Opdam (in \cite{opdamh}), we let $I_P^G(\tau_{\nu})$ be embedded in $I_{P_1}^G(\sigma_{\nu_0+\nu})$ with 
cuspidal support $
(\sigma, \lambda= \nu_0 + \nu)$. 

Using Langlands' classification, we write $J(P',\tau',\nu')$ an irreducible generic subquotient of $I_P^G(\tau_{\nu})$. Then the standard module 
conjecture claims that $J(P',\tau',\nu') \cong I_{P'}^G(\tau'_{\nu'})$.

The first case to consider is a \emph{generic} standard module $I_{P''}^G(\tau''_{\nu'})$. From the unicity of the generic irreducible module with cuspidal support $(\sigma,\lambda)$ (Rodier's Theorem, [U]), one sees that $J(P',\tau',\nu')\cong I_{P'}^G(\tau'_{\nu'}) \leq I_{P''}^G(\tau''_{\nu''})$. Hence, $\nu' ~_{P''}\!< \nu''$.

Secondly, if the standard module $I_{P''}^G(\tau''_{\nu''})$ is any (non-generic) subquotient having $(\sigma,\lambda)$ as cuspidal support, 
since this 
cuspidal support is generic one will see that one can replace $\tau''$ by the generic tempered representation $\tau''_{\text{gen}}$ with same 
cuspidal support 
and conserve the Langlands parameter $\nu''$ and we are back to the first case. This is explained in the next paragraph. The lemma 
follows.

To replace the tempered representation $\tau''$ of $M''$ the argument goes as follows:
Since the representation $\sigma$ in the cuspidal support of this representation is generic, by Theorem \ref{muic} the unique irreducible 
generic 
representation subquotient $\tau''_{\text{gen}}$ in the representation induced from this cuspidal support is tempered. As any representation in 
the cuspidal 
support of $\tau''$ must lie in the cuspidal support of $\tau''_{\text{gen}}$, any such representation must be conjugated to $\sigma$. 
That is there exists a Weyl group element $w \in W$ such that if $\tau'' \hookrightarrow I_{P_1\cap M''}^{M''}(\sigma_{\nu_0})$
then
$$\tau''_{\text{gen}} 
\hookrightarrow I_{P_1\cap M''}^{M''}((w\sigma)_{w\nu_0})$$ Twisting by $\nu'' \in a_{M'}^*$ comes second. Therefore, conjugation by this Weyl group element leaves invariant the Langlands parameter $\nu'' \in a_{M'}^*$, and $(\tau''_{\text{gen}})_{\nu''}$ and $\tau''_{\nu''}$ share therefore the same cuspidal support.
\end{proof}
\subsection{Linear residual segments} \label{lrs}

Let $I_{P}^G(\tau_{s\tilde{\alpha}})$ be a standard module, we call the parameter $s\tilde{\alpha}$ \emph{the Langlands parameter of the standard module}.
We have seen that this Langlands parameter (the \emph{twist}) depends only on the linear (not semi-simple) part of the cuspidal 
support, i.e the linear residual segment. 

In this section and the following we use the notation $\mathcal{S}$ (see the Definition \ref{segment}) to denote a \emph{linear residual segment}, the underlying irreducible cuspidal representation $\rho$ is implicit. A simple computation gives that if a standard module $I_{P}^G(\tau_{s\tilde{\alpha}})$, where $P$ is a maximal parabolic, embeds in $I_{P_1}^G(\sigma(\mathpzc{a},\mathpzc{b},\underline{n}))$ for a cuspidal string 
$(\mathpzc{a}, \mathpzc{b}, \underline{n})$, then $s= \frac{\mathpzc{a}+\mathpzc{b}}{2}$. The parameter $s\tilde{\alpha}$ is in $(a_{M}^*)^+$, but to use Lemma \ref{bw} we will need to consider it as an element of $a_{M_1}^*$.

Then, we say this Langlands parameter is \emph{associated} to the linear residual segment $(\mathpzc{a},\ldots,\mathpzc{b})$. In this subsection, we compare Langlands parameters associated to linear residual segments. 

\begin{lemma}\label{split} 
Let $\gamma$ be a real number such that $\mathpzc{a} \geq \gamma \geq \mathpzc{b}$.

Splitting a linear residual segment $(\mathpzc{a},\ldots, \mathpzc{b})$ whose associated Langlands parameter is $\lambda=\frac{\mathpzc{a}+\mathpzc{b}}{2} \in a_{M}^*$ into 
two segments: $(\mathpzc{a},\ldots,\gamma+1)(\gamma, \mathpzc{b})$ yields necessarily a larger Langlands parameter, $\lambda'$ for the order given in Lemma \ref{bw}.
\end{lemma}

\begin{proof}
We write $\lambda \in a_{M}^*$ as an element in $a_{M_1}^*$ to be able to use Lemma \ref{bw} (i.e the Lemma \ref{bw} also applies with $a_{M_1}^*$):
$$\lambda= (\underbrace{\frac{\mathpzc{a}+\mathpzc{b}}{2}, \ldots, \frac{\mathpzc{a}+\mathpzc{b}}{2}}_{\mathpzc{a} - \mathpzc{b} +1 ~\mbox{times}}) ~~\hbox{also,} ~~ \lambda'= (\underbrace{\frac{\mathpzc{a}+(\gamma+1)}{2}, \ldots, \frac{\mathpzc{a}+(\gamma+1)}{2}}_{\mathpzc{a} - \gamma ~\mbox{times}}, \underbrace{\frac{\gamma+\mathpzc{b}}{2}, \ldots, \frac{\gamma+\mathpzc{b}}{2}}_{\gamma - \mathpzc{b} +1 ~\mbox{times}})$$

$$\lambda'-\lambda= (\underbrace{\frac{(\gamma+1)-\mathpzc{b}}{2}, \ldots, \frac{(\gamma+1)-\mathpzc{b}}{2}}_{\mathpzc{a} -\gamma ~\mbox{times}}, 
\underbrace{\frac{\gamma-\mathpzc{a}}{2}, \ldots, \frac{\gamma-\mathpzc{a}}{2}}_{\gamma - \mathpzc{b} +1 ~\mbox{times}})$$

Therefore, $x_1= \frac{(\gamma+1)-\mathpzc{b}}{2} > 0$. Since $x_k = \sum_{i=1}^k a_i$ as written in the proof of Lemma \ref{rmkorder}, one observes that $x_k > x_n$ for any $k< n=\mathpzc{a} - \mathpzc{b} +1$, and $x_n =\frac{(\gamma+1)-\mathpzc{b}}{2}(\mathpzc{a} -
\gamma) + \frac{\gamma-\mathpzc{a}}{2}(\gamma - \mathpzc{b} +1 ) = (\mathpzc{a} - \gamma)(\frac{(\gamma+1)-\mathpzc{b}}{2} - \frac{-\gamma + \mathpzc{b}-1}{2}) = 0$. Hence, $
\lambda' \bwg \lambda$ by Lemma \ref{rmkorder}. 
\end{proof}

\begin{prop} \label{lambdalambda'}
Consider two linear (i.e of type $A$) residual segments, i.e strictly decreasing sequences of real numbers such that the difference between two consecutive reals is one: $\mathcal{S}_1:=(\mathpzc{a}_1, 
\ldots, \mathpzc{b}_1); \mathcal{S}_2:=(\mathpzc{a}_2, \ldots, \mathpzc{b}_2)$.
Typically, one could think of decreasing sequences of consecutive integers or consecutive half-integers.

Assume $\mathpzc{a}_1 > \mathpzc{a}_2 > \mathpzc{b}_1 > \mathpzc{b}_2$ so that they are linked in the terminology of Bernstein-Zelevinsky. Taking intersection and 
union yield two unlinked residual segments $\mathcal{S}_1\cap\mathcal{S}_2 \subset \mathcal{S}_1\cup\mathcal{S}_2$.

Denote $\lambda \in a_M^*$ the Langlands parameter $\lambda=(s_1,s_2)$ associated to $\mathcal{S}_1$ and $\mathcal{S}_2$, and expressed in the canonical basis associated to 
the Lie algebra $a_0^*$.

Denote $ \lambda'\in a_M^*: \lambda'=(s_1',s_2')$ the one associated to the two unlinked segments $\mathcal{S}_1\cap\mathcal{S}_2 , \mathcal{S}_1\cup\mathcal{S}_2$ ordered so 
that $s_1'>s_2'$. 

Then, $\lambda' \bw \lambda$.
\end{prop}

\begin{proof}
Let $(\mathpzc{a}_1, \ldots, \mathpzc{b}_1)(\mathpzc{a}_2, \ldots, \mathpzc{b}_2)$ be two  segments 
with $\mathpzc{a}_1 > \mathpzc{a}_2 > \mathpzc{b}_1 > \mathpzc{b}_2$ so that the two segments are linked.
The associated Langlands parameter is: 
$$\lambda= (\underbrace{\frac{\mathpzc{a}_1+\mathpzc{b}_1}{2}, \ldots, \frac{\mathpzc{a}_1+\mathpzc{b}_1}{2}}_{\mathpzc{a}_1 - \mathpzc{b}_1 +1 ~\mbox{times}} , 
\underbrace{\frac{\mathpzc{a}_2+
\mathpzc{b}_2}{2}, \ldots, \frac{\mathpzc{a}_2+\mathpzc{b}_2}{2} }_{\mathpzc{a}_2 - \mathpzc{b}_2 +1 ~\mbox{times}})$$

Then taking union and intersection of those two segments gives:
$(\mathpzc{a}_1, \ldots, \mathpzc{b}_2)(\mathpzc{a}_2, \ldots, \mathpzc{b}_1)$
or $(\mathpzc{a}_2, \ldots, \mathpzc{b}_1)(\mathpzc{a}_1, \ldots, \mathpzc{b}_2)$ ordered so that $s_1'>s_2'$.
The Langlands parameter will therefore be given by:
\begin{enumerate}
	\item If $\frac{\mathpzc{a}_1+\mathpzc{b}_2}{2} \geq \frac{\mathpzc{a}_2+\mathpzc{b}_1}{2}$:
	$$\lambda'= (\underbrace{\frac{\mathpzc{a}_1+\mathpzc{b}_2}{2}, \ldots, \frac{\mathpzc{a}_1+\mathpzc{b}_2}{2}}_{\mathpzc{a}_1 - \mathpzc{b}_2 +1 ~\mbox{times}} , 
\underbrace{\frac{\mathpzc{a}_2+\mathpzc{b}_1}{2}, \ldots, \frac{\mathpzc{a}_2+\mathpzc{b}_1}{2} }_{\mathpzc{a}_2 - \mathpzc{b}_1 +1 ~\mbox{times}})$$
	
	\item If $\frac{\mathpzc{a}_2+\mathpzc{b}_1}{2} > \frac{\mathpzc{a}_1+\mathpzc{b}_2}{2}$:
$$\lambda'= ( \underbrace{\frac{\mathpzc{a}_2+\mathpzc{b}_1}{2}, \ldots, \frac{\mathpzc{a}_2+\mathpzc{b}_1}{2} }_{\mathpzc{a}_2 - \mathpzc{b}_1 +1 ~\mbox{times}}, 
\underbrace{\frac{\mathpzc{a}_1+\mathpzc{b}_2}{2}, \ldots, \frac{\mathpzc{a}_1+\mathpzc{b}_2}{2}}_{\mathpzc{a}_1 - \mathpzc{b}_2 +1 ~\mbox{times}})$$
\end{enumerate}


Then the difference $\lambda - \lambda'$ equals: 
\begin{itemize}

\item In case (1),
$~~(\underbrace{\frac{\mathpzc{b}_1-\mathpzc{b}_2}{2}, \ldots, \frac{\mathpzc{b}_1-\mathpzc{b}_2}{2}}_{\mathpzc{a}_1 - \mathpzc{b}_1 +1 ~\mbox{times}} ,\underbrace{\frac{\mathpzc{a}_2 
- \mathpzc{a}_1}{2}, 
\ldots, \frac{\mathpzc{a}_2 - \mathpzc{a}_1}{2}}_{\mathpzc{b}_1-\mathpzc{b}_2 ~\mbox{times}}, \underbrace{\frac{\mathpzc{b}_2 -\mathpzc{b}_1}{2}, \ldots, \frac{\mathpzc{b}_2-\mathpzc{b}_1}{2} }
_{\mathpzc{a}_2 - 
\mathpzc{b}_1 +1 ~\mbox{times}}, 0, \ldots, 0)$

First, $x_1 = \frac{\mathpzc{b}_1-\mathpzc{b}_2}{2}$. Secondly, since $x_k = \sum_{i=1}^k a_i$ as written in the proof of Lemma \ref{rmkorder}, one observes that all subsequent $x_k$ are greater or equal to $x_n$, for $n= \mathpzc{a}_1 - \mathpzc{b}_1 +1 + \mathpzc{a}_2 - \mathpzc{b}_2 +1$.

And $x_n = \frac{\mathpzc{b}_1-\mathpzc{b}_2}{2}(\mathpzc{a}_1 - \mathpzc{b}_1 +1) + 
\frac{\mathpzc{a}_2 - \mathpzc{a}_1}{2}(\mathpzc{b}_1-\mathpzc{b}_2) + \frac{\mathpzc{b}_2 -\mathpzc{b}_1}{2}(\mathpzc{a}_2 - \mathpzc{b}_1 +1) = \frac{\mathpzc{b}_1-\mathpzc{b}_2}{2}(\mathpzc{a}_1 - \mathpzc{b}_1 
+1 + \mathpzc{a}_2 
- \mathpzc{a}_1 -(\mathpzc{a}_2 - \mathpzc{b}_1 +1)) = 0$ 

\item In case (2),
$~~\lambda - \lambda' = 
(\underbrace{\frac{\mathpzc{a}_1-\mathpzc{a}_2}{2}, \ldots, \frac{\mathpzc{a}_1-\mathpzc{a}_2}{2}}_{\mathpzc{a}_2 - \mathpzc{b}_1 +1 ~\mbox{times}} ,\underbrace{\frac{\mathpzc{b}_1 
- \mathpzc{b}_2}{2}, 
\ldots, \frac{\mathpzc{b}_1 - \mathpzc{b}_2}{2}}_{\mathpzc{a}_1-\mathpzc{a}_2 ~\mbox{times}}, \underbrace{\frac{\mathpzc{a}_2-\mathpzc{a}_1}{2}, \ldots, \frac{\mathpzc{a}_2-\mathpzc{a}_1}
{2} }_{\mathpzc{a}_2 
- \mathpzc{b}_2 +1 ~\mbox{times}})$

Here $x_1= \frac{\mathpzc{a}_1-\mathpzc{a}_2}{2}$
$x_n = \frac{\mathpzc{a}_1-\mathpzc{a}_2}{2}(\mathpzc{a}_2 - \mathpzc{b}_1 +1) + \frac{\mathpzc{b}_1 - \mathpzc{b}_2}{2}(\mathpzc{a}_1-\mathpzc{a}_2 ) + \frac{\mathpzc{a}_2-\mathpzc{a}_1}{2}
(\mathpzc{a}_2 - \mathpzc{b}_2 
+1)
= \frac{\mathpzc{a}_2-\mathpzc{a}_1}{2}(\mathpzc{a}_2 - \mathpzc{b}_1 +1 + \mathpzc{b}_1 - \mathpzc{b}_2 -(\mathpzc{a}_2 - \mathpzc{b}_2 +1)) = 0$. 
\end{itemize}
\end{proof}

\begin{prop}\label{lambdamin}
The Langlands parameter $\lambda'$, as defined in the previous Proposition \ref{lambdalambda'}, is the minimal Langlands parameter for the order given in Lemma \ref{bw} on this cuspidal support. 
\end{prop}

\begin{proof}
Let us consider a decreasing sequence of real numbers such that the difference between two consecutive elements is one: $(\mathpzc{a}_1, \mathpzc{a}_1 -1, \ldots , \mathpzc{a}_2, \ldots, \mathpzc{b}_1, \ldots ,\mathpzc{b}_2)$ with the following conditions: $\mathpzc{a}_1 > \mathpzc{a}_2 > \mathpzc{b}_1 > \mathpzc{b}_2$ and all real numbers between $\mathpzc{a}_2$ and $\mathpzc{b}_1$ are repeated twice.
 Let us call this sequence $c$.

We consider the set $\mathscr S$  of tuple of linear segments $\mathcal{S}_i=(a_i, \ldots, b_i)$ (strictly decreasing sequence of reals) such that if $s_i=\frac{a_i+b_i}{2} \geq s_j=\frac{a_j+b_j}{2}$ then the linear segment $\mathcal{S}_i$ is placed on the left of $\mathcal{S}_j$, i.e. :
$$ (\mathcal{S}_{1}, \mathcal{S}_{2}, \ldots, \mathcal{S}_{k}) \in \mathscr S \Leftrightarrow s_{1} \geq s_{2} \ldots \geq s_{k}$$

In this set $\mathscr S$, let us first consider the special case of a decreasing sequence $\delta \in \mathscr S$  where each segment is length one and $s_i= \mathcal{S}_i$. 
Then the Langlands parameter is just $\delta= (\mathpzc{a}_1, \mathpzc{a}_1 -1, \ldots, \mathpzc{a}_2, \mathpzc{a}_2, \ldots, \mathpzc{b}_1, \mathpzc{b}_1, \ldots \mathpzc{b}_2)$

Secondly, let us consider the case where all segments are mutually unlinked, then they have to be included in one another. The reader will readily notice that the only option is the following element in $\mathscr S$ :
$$m:= (\mathpzc{a}_1, \ldots, \mathpzc{b}_2)(\mathpzc{a}_2, \ldots, \mathpzc{b}_1)$$

Its Langlands parameter is:$\lambda'= (\underbrace{\frac{\mathpzc{a}_1+\mathpzc{b}_2}{2}, \ldots, \frac{\mathpzc{a}_1+\mathpzc{b}_2}{2}}_{\mathpzc{a}_1 - \mathpzc{b}_2 +1 ~\mbox{times}} , 
\underbrace{\frac{\mathpzc{a}_2+\mathpzc{b}_1}{2}, \ldots, \frac{\mathpzc{a}_2+\mathpzc{b}_1}{2} }_{\mathpzc{a}_2 - \mathpzc{b}_1 +1 ~\mbox{times}})$

Let us show that $\delta \bwg \lambda'$.

Clearly on the vector $\delta -\lambda'$: $x_1 = \mathpzc{a}_1 -\frac{\mathpzc{a}_1+\mathpzc{b}_2}{2} > 0$, $x_k = \sum_{i=1}^k a_i$ and one observes that all subsequent $x_k$ are greater or equal to $x_n$, and $x_n$ is the sum of the elements (counted with multiplicities) in the vector $\delta$ minus $\frac{\mathpzc{a}_1+\mathpzc{b}_2}{2}(\mathpzc{a}_1 - \mathpzc{b}_2 +1)+\frac{\mathpzc{a}_2+\mathpzc{b}_1}{2}(\mathpzc{a}_2 - \mathpzc{b}_1 +1)$, therefore $x_n = 0$ as this sum ends up the same as in the proof of the previous proposition.

\vspace{0.2cm}
Let us show that $m$ is the unique, irreducible element obtained in $\mathscr S$  when taking repeatedly intersection and union of any two segments in any element $s \in \mathcal{S}$. Let us write an arbitrary $s \in \mathscr S$  as $(\mathcal{S}_1, \mathcal{S}_2, \ldots, \mathcal{S}_p)$, since we had a certain number of reals repeated twice in $c$, it is clear that some of the $\mathcal{S}_i$ are mutually linked. 

For our purpose, we write the vector of lengths of the segments in $s$: $(k_1, k_2, \ldots, k_p)$. Let us assume, without loss of generality, that $\mathcal{S}_1$ and $\mathcal{S}_2$ are linked. Taking intersection and union, we obtain two unlinked segments $\mathcal{S}'_1=\mathcal{S}_1\cup\mathcal{S}_2$ and $\mathcal{S}'_2=\mathcal{S}_1\cap\mathcal{S}_2$. If $k_1 \geq k_2$, then $k_1'= k_1 + a$, and $k_2'= k_2 - a$, i.e. the greatest length necessarily increases. Therefore, the potential $\sum_i k_i^2$ is increasing, while the number of segments is non-increasing. The process ends when we cannot take anymore intersection and union of linked segments, then the longest segment contains entirely the second longest, this is the element $m \in \mathscr S$ introduced above.

Since at each step (of taking intersection and union of two linked segments) the Langlands parameter $\lambda_{s'}$ of the element $s' \in \mathscr S$ is smaller than at the previous step (by Proposition \ref{lambdalambda'}), it is clear that $\lambda'$ is the minimal element for the order on Langlands parameter.
\end{proof}

\begin{rmk} \label{AAA}
Let us assume we fix the cuspidal representation $\sigma$ and two segments $(\mathcal{S}_1, \mathcal{S}_2)$. As a result of this proposition, the standard module $I_{P'}^G(\tau'_{\lambda'})$ induced from the unique irreducible generic essentially square integrable representation $\tau'_{\lambda'}$ obtained when taking intersection and union $(\mathcal{S}_1\cap \mathcal{S}_2)$ and $(\mathcal{S}_1\cup\mathcal{S}_2)$ (i.e. which embeds in $I_{P_1}^G(\sigma((\mathcal{S}_1\cap \mathcal{S}_2);(\mathcal{S}_1\cup\mathcal{S}_2))$) is irreducible by Theorem \ref{minLP=irr}.
\end{rmk}

\subsection{Intertwining operators} \label{IO}

In the following result, we play for the first time with cuspidal strings and intertwining operators. We fix a unitary irreducible cuspidal 
representation $\sigma$ of $M_1$ and let $(\mathpzc{a},\mathpzc{b},\underline{n})$ and $(\mathpzc{a'},\mathpzc{b'}, \underline{n'})$ be two elements in some $W_{\sigma}$-cuspidal string; i.e, there exists a Weyl group element $w \in W_{\sigma}$ 
such that $w(\mathpzc{a},\mathpzc{b}, 
\underline{n}) =(\mathpzc{a'},\mathpzc{b'},\underline{n'})$.

For the sake of readability we sometimes denote $I_{P_1}^G(\sigma(\lambda)):= I_{P_1}^G(\sigma_{\lambda})$ when the 
parameter $\lambda$ is expressed in terms of residual segments. We would like to study intertwining operators between $I_{P_1}^G(\sigma(\mathpzc{a},\mathpzc{b}, \underline{n}))$ and $I_{P_1}^G(\sigma(\mathpzc{a'},\mathpzc{b'}, \underline{n'}))$. As explained in Section \ref{io1} and Proposition \ref{nngenerick}, this operator can be decomposed in rank one operators. Let us recall how one can conclude on the non-genericity of their kernels in the two main cases.

\begin{example}[Rank one intertwining operators with non-generic kernel] \label{exampleio}

Let us assume $\Sigma_{\sigma}$ is irreducible of type $A,B,C$ or $D$. We fix a unitary irreducible cuspidal representation $\sigma$ and let $\alpha= e_i - e_{i+1}$ be a simple root in $\Sigma_{\sigma}$. The element $s_{\alpha}$ operates on $\lambda$ in $(a_{M_1}^G)^*$. In this first example, we illustrate the case where $s_{\alpha}$ acts as a coordinates' transposition on $\lambda$ written in the standard basis $\left\{e_i \right\}_i$ of $(a_{M_1}^G)^*$.

Let us focus on two adjacent elements in the residual segment corresponding to $\lambda$ (at the coordinates $\lambda_i$ and $\lambda_{i+1}$): $\left\{a,b\right\}$, let us consider the rank one operator which goes from $I_{P_1\cap (M_1)_{\alpha}}^{(M_1)_{\alpha}}(\sigma_{\ldots \left\{a,b\right\} \ldots})$ to $I_{\overline{P_1}\cap (M_1)_{\alpha}}^{(M_1)_{\alpha}}(\sigma_{{\ldots \left\{a,b\right\} \ldots}})$. By Proposition \ref{nngenerick} it is an operator with non-generic kernel if and only if $a<b$; Indeed if we denote $\lambda:= (\ldots ,a,b, \ldots)$, then $\prodscal{\check{\alpha}}{\lambda} = a-b < 0$ (The action of $s_{\alpha}$ on $\lambda$ leaves fixed the other coordinates of $\lambda$ that we simply denote by dots).

Since $\alpha \in \Sigma_{\sigma}$, by point (a) in Harish-Chandra's Theorem [Theorem \ref{hc}], there is a unique non-trivial element 
$s_{\alpha}$ in $W^{(M_1)_{\alpha}}(M_1)$ such that $s_{\alpha}(P_1\cap (M_1)_{\alpha})=\overline{P_1}\cap (M_1)_{\alpha}$ and which operates as the transposition from $(a,b)$ to $(b,a)$. The rank one operator from $I_{\overline{P_1}\cap (M_1)_{\alpha}}^{(M_1)_{\alpha}}(\sigma_{\ldots , a,b, \ldots})$ to $I_{s_{\alpha}(\overline{P_1}\cap (M_1)_{\alpha})}^{(M_1)_{\alpha}}(s_{\alpha}(\sigma_{\ldots ,a,b, \ldots})):= I_{P_1\cap (M_1)_{\alpha}}^{(M_1)_{\alpha}}(\sigma_{\ldots ,b,a, \ldots})$ is bijective. Eventually we have shown that the composition of those two which goes from $I_{P_1\cap (M_1)_{\alpha}}^{(M_1)_{\alpha}}(\sigma_{\ldots , a,b, \ldots})$ to $I_{P_1\cap (M_1)_{\alpha}}^{(M_1)_{\alpha}}(\sigma_{\ldots ,b,a, \ldots})$ has non-generic kernel.

\vspace{0.3cm}

If the Weyl group $W_{\sigma}$ is isomorphic to $S_n \rtimes \left\{\pm 1 \right\}$, the Weyl group element corresponding to $\left\{\pm 1 \right\}$ is the sign change and we operate this sign change on the latest coordinate of $\lambda$ (extreme right of the cuspidal string).

By the same argumentation as in the first example, for $a>0$, the operator $I_{P_1\cap (M_1)_{\alpha}}^{(M_1)_{\alpha}}(\sigma_{\ldots -a})$ to $I_{P_1\cap (M_1)_{\alpha}}^{(M_1)_{\alpha}}(\sigma_{\ldots a})$ has non-generic kernel.

\end{example}

\begin{example} \label{exampleio2}
Let $G$ be a classical group of rank $n$. Let us take $\sigma$ an irreducible unitary generic cuspidal representation of $M_1$, a standard 
Levi subgroup of 
$G$.
Let us assume $\Sigma_{\sigma}$ is irreducible of type $B$, and take $\lambda:= (s_1, s_2, \ldots, s_m)$ in $a_{M_1}^*$, $\rho$ an irreducible unitary 
cuspidal 
representation of $GL_k$, and $\sigma_c$ an irreducible unitary cuspidal representation of $G(k')$ $k' < n$.
Then $\sigma_{\lambda}$ is:
$$\sigma_{\lambda}:= \rho|.|^{s_1}\otimes \rho|.|^{s_2}\otimes \ldots \otimes\rho|.|^{s_m}\otimes \sigma_c$$

The element $s_{\alpha_i}$ operates as follows:

$$s_{\alpha_i}(\rho|.|^{s_1}\otimes \ldots \rho|.|^{s_i}\otimes \rho|.|^{s_{i+1}} \ldots \otimes \rho|.|^{s_m}\otimes \sigma_c)= \rho|.|^{s_1}\otimes 
\ldots \otimes 
\rho|.|^{s_{i+1}} \otimes \rho|.|^{s_i}\otimes \ldots \otimes \rho|.|^{s_m}\otimes \sigma_c$$

Indeed, for such $\alpha_i$ (which is in $\Sigma_{\sigma}$), one checks that property (a) in Theorem \ref{hc} holds: $s_{\alpha_i}(\sigma) 
\cong \sigma$. 
This is verified for any $i \in \left\{1, \ldots, n \right\}$.
The intertwining operator usually considered in this manuscript is induced by functoriality from the application $\sigma_{\lambda} \rightarrow 
s_{\alpha_i}
(\sigma_{\lambda})$.

\end{example}

\begin{lemma}\label{nngenerickk}
Let $b' \leq \ell + m, b \leq a$.
Fix a unitary irreducible cuspidal representation $\sigma$ of a maximal Levi subgroup in a quasi-split reductive group $G$, and two cuspidal strings $(a, b, \underline{n})$ and $(a, b',\underline{n'})$ in a $W_{\sigma}$-cuspidal string (notice that the right end of these are equals with value $a$).
If $b' \geq b$, the intertwining operator between $I_{P_1}^G(\sigma(a, b, \underline{n}))$ and $I_{P_1}^G(\sigma(a, b', \underline{n'}))$ has 
non-generic 
kernel.
\end{lemma}

\begin{proof}
In this proof, to detail the operations on cuspidal strings more explicitly we write the residual segments of type $B, C, D$ defined in Definition \ref{rs} as:
$$((\ell +m)(\ell +m-1) \ldots ((\ell+1)\ell^{n_{\ell}}(\ell-1)^{n_{\ell-1}}(\ell-2)^{n_{\ell-2}} \ldots 2^{n_2} 1^{n_1}0^{n_0})$$ 
where $n_i$ denote the number of times the (half)-integer $i$ is repeated.
We present the arguments for integers, the proof for half-integers follows the same argumentation.

First, assume $b \geq 0$, and consider changes on the cuspidal strings $$(a,\ldots, b', b'-1,\ldots b)((\ell+m)\ldots \ell^{n_{\ell}}
(\ell-1)^{n_{\ell-1}}\ldots 
b^{n_{b}} \ldots 2^{n_2} 1^{n_1}0^{n_0})$$ consisting in permuting successively all elements in $\left\{b, \ldots, b'-1\right\}$ with their right hand 
neighbor, as 
soon as this right hand neighbor is larger. 
We incorporate all elements starting with $b$ until $b'-1$ from the left into the right hand residual segment. The rank one intertwining 
operators associated to 
those permutations have non-generic kernel (see Example \ref{exampleio}); hence the intertwining operator from $I_{P_1}^G(\sigma(a, b, 
\underline{n}))$ to 
$I_{P_1}^G(\sigma(a, b',  \underline{n'}))$ as composition of those rank one operators has non-generic kernel. 

Assume now $b < 0$ and write $b = -\gamma$. Let us show that there exists an intertwining operator with non-generic kernel from the module 
induced from $I_{P_1}^G(\sigma(a, -\gamma, \underline{n}))$ to the one induced from $I_{P_1}^G(\sigma(a, b',\underline{n'}))$. The decomposition in rank one operators has the following two steps (the details on the first step are given in the next paragraph):
\begin{enumerate}
\item 
\begin{enumerate}[label=(\alph*)] 
	\item If $b'\geq 1 > b$
	From the cuspidal string 
	$$(a,\ldots, \gamma, \gamma-1,\ldots ,-\gamma)((\ell+m)\ldots \ell^{n_{\ell}}(\ell-1)^{n_{\ell-1}}\ldots b^{n_{b}} \ldots 2^{n_2} 
1^{n_1}0^{n_0})$$ to $$(a,\ldots, \gamma, \gamma-1,\ldots, 1)((\ell+m)\ldots \ell^{n_{\ell}}(\ell-1)^{n_{\ell-1}}\ldots b^{n_{b}} \ldots 2^{n_2} 1^{n_1}
0^{n_0+1},-1, \ldots ,-\gamma)$$
 and then to
 $$(a,\ldots, \gamma, \gamma-1,\ldots, 1)((\ell+m)\ldots \ell^{n_{\ell}}(\ell-1)^{n_{\ell-1}}\ldots b^{n_{b}+1} \ldots 2^{n_2+1} 1^{n_1+1}
0^{n_0+1})$$

\item If $0 \geq b'\geq b$
From the cuspidal string 
	$$(a,\ldots, \gamma, \gamma-1,\ldots ,-\gamma)((\ell+m)\ldots \ell^{n_{\ell}}(\ell-1)^{n_{\ell-1}}\ldots b^{n_{b}} \ldots 2^{n_2} 
1^{n_1}0^{n_0})$$ to $$(a,\ldots, \gamma, \gamma-1,\ldots, b')((\ell+m)\ldots \ell^{n_{\ell}}(\ell-1)^{n_{\ell-1}}\ldots b^{n_{b}} \ldots 2^{n_2} 1^{n_1}
0^{n_0+1},b'-1, \ldots ,-\gamma)$$
 and then to
$$(a,\ldots, \gamma, \gamma-1,\ldots, b')((\ell+m)\ldots \ell^{n_{\ell}}(\ell-1)^{n_{\ell-1}}\ldots b^{n_{b}+1} \ldots 2^{n_2+1} 1^{n_1+1}
0^{n_0+1})$$
\end{enumerate}

\item In case (a), from $(a, \ldots, 1)(\underline{n''})$ to $(a, \ldots, b')( \underline{n'})$ by the same arguments as in the case $b \geq 0$ treated in the first paragraph of this proof.

\end{enumerate}

We detail the operations in step 1:
\begin{enumerate}[label=(\roman*)]
\item Starting with $-\gamma$, all negative elements in $\left\{0, \ldots, -\gamma \right\}$ are successively sent to the extreme right of the 
second residual 
segment $(\underline{n})$. At each step, the rank one intertwining operator between $(a,p)$ and $(p,a)$ where $p$ is a negative integer (or 
half-integer) and 
$a>p$ has non-generic kernel.
\item We use rank one operators of the second type (sign chance of the extreme right element of the cuspidal string). Since they intertwine 
cuspidal strings 
where the last element changes from negative to positive, they have non-generic kernels. Then, the positive element is moved up left. The 
right-hand 
residual segment goes from 
$$((\ell+m)\ldots \ell^{n_{\ell}}(\ell-1)^{n_{\ell-1}}\ldots b^{n_{b}} \ldots 2^{n_2} 1^{n_1}0^{n_0+1},-1, \ldots ,-\gamma)$$ to 
$$((\ell+m)\ldots 
\ell^{n_{\ell}}(\ell-1)^{n_{\ell-1}}\ldots b^{n_{b}} \ldots 2^{n_2} 1^{n_1}0^{n_0+1},-1, \ldots ,\gamma)$$
 and then to $$((\ell+m)\ldots \ell^{n_{\ell}}
(\ell-1)^{n_{\ell-1}}\ldots 
b^{n_{b}} \ldots 2^{n_2} 1^{n_1}0^{n_0+1},\gamma,-1, \ldots ,-(\gamma-1))$$

Once changed to positive, permuting successively elements from right to left, one can reorganize the residual segment $(\ell+m)\ldots 
\ell^{n_{\ell}}
(\ell-1)^{n_{\ell-1}}\ldots b^{n_{b}} \ldots 2^{n_2} 1^{n_1}0^{n_0+1}, \gamma, \ldots 1)$ such as it is a decreasing sequence of (half)-integers. 
Again 
intertwining operators following these changes on the cuspidal string have non-generic kernels.
\end{enumerate}
\end{proof}

\begin{example}
Consider the cuspidal string (543210-1)(43 322 211 1 0) and the dominant residual point in its $W_{\sigma}$-cuspidal string: (54 433 3222 21111 10 0).
To the Weyl group element $w \in W_{\sigma}$ associate an intertwining operator from the module induced with string (534210-1)(43 322 211 1 0) to 
the one induced 
with cuspidal-string (54 433 3222 21111 10 0) which has non-generic kernel.

Indeed one will decompose it into transpositions $s_{\alpha}$ such as (-1,4) to (4,-1) and similarly for any  $4> i \geq 0$: $(-1,i)$ to $(i,-1)$.

This process will result in (543210)(43 322 211 1 0 -1).
Then one will change the -1 to 1, and by the above the associated rank-one operator also has non-generic kernel. Then use the rank one operators with non-generic kernel such as :$(0,1) \rightarrow (1,0)$.

Then notice that the '4', '3' and '2' in the middle of the sequence can be moved to the left with a sequence of rank one operators with non-generic kernel such as :$(0,4) \rightarrow (4,0); \ldots ;(3,4)\rightarrow (4,3)$. 
\end{example}

\begin{lemma} \label{Sio}
Let $(\mathcal{S}_1, \mathcal{S}_2, \ldots, \mathcal{S}_t)$ be an ordered sequence of $t$ linear segments and let us denote $\mathcal{S}_i= (\mathpzc{a}_i, \ldots, \mathpzc{b}_i)$, for any $i$ in $\left\{1, \ldots, t \right\}$. This sequence is ordered so that for any $i$ in $\left\{1, \ldots, t \right\}$, $s_i = \frac{\mathpzc{a}_i + \mathpzc{b}_i}{2} \geq s_{i+1}= \frac{\mathpzc{a}_{i+1} + \mathpzc{b}_{i+1}}{2}$. Let us assume that for some indices in $\left\{1, \ldots, t \right\}$ the linear residual segments are linked.

Let us denote $(\mathcal{S}'_1, \mathcal{S}'_2, \ldots, \mathcal{S}'_t)$ the ordered sequence corresponding to the end of the procedure of taking union and intersection of linked linear residual segments. 
This sequence is composed of at most $t$ unlinked residual segments $\mathcal{S}'_i= (\mathpzc{a}'_i,\ldots, \mathpzc{b}'_i), i \in \left\{1, \ldots, t \right\}$.

Taking repeatedly intersection and union yields smaller Langlands parameters for the order defined in Lemma \ref{bw}; and we denote the smallest element for this order, $\underline{s'}$. It corresponds to the sequence $(\mathcal{S}'_1, \mathcal{S}'_2, \ldots, \mathcal{S}'_t)$ as explained in Proposition \ref{lambdamin}.

Then, there exists an intertwining operator with non-generic kernel from the induced module $I_{P_1}^G(\sigma((\mathcal{S}'_1, \mathcal{S}'_2, \ldots, \mathcal{S}'_t; \underline{n}))$ to $I_{P_1}^G(\sigma((\mathcal{S}_1, \mathcal{S}_2, \ldots, \mathcal{S}_t; \underline{n}))$.
\end{lemma}

\begin{proof}
Let us first consider the case $t=2$.

Consider two linear (i.e of type $A$) residual segments, i.e strictly decreasing sequences of either consecutive integers or consecutive half-integers $\mathcal{S}_1:=(\mathpzc{a}_1, 
\ldots, \mathpzc{b}_1); \mathcal{S}_2:=(\mathpzc{a}_2, \ldots, \mathpzc{b}_2)$.

Assume $\mathpzc{a}_1 > \mathpzc{a}_2 > \mathpzc{b}_1 > \mathpzc{b}_2$ so that they are linked in the terminology of Bernstein-Zelevinsky. Taking intersection and union yield two unlinked linear residual segments $\mathcal{S}_1\cap\mathcal{S}_2 \subset \mathcal{S}_1\cup\mathcal{S}_2$:
$(\mathpzc{a}_1, \ldots, \mathpzc{b}_2)(\mathpzc{a}_2, \ldots, \mathpzc{b}_1)$
or $(\mathpzc{a}_2, \ldots, \mathpzc{b}_1)(\mathpzc{a}_1, \ldots, \mathpzc{b}_2)$ ordered so that $s_1'>s_2'$.

As in the proof of Lemma \ref{nngenerickk}, because $\mathpzc{a}_2 > \mathpzc{b}_2$ and also $\mathpzc{b}_1 > \mathpzc{b}_2$ there exists an intertwining operator with non-generic kernel from the module induced with cuspidal support $(\mathpzc{a}_1, \ldots, \mathpzc{b}_2)(\mathpzc{a}_2, \ldots, \mathpzc{b}_1)$ to the one induced with cuspidal support $(\mathpzc{a}_1, 
\ldots, \mathpzc{b}_1)(\mathpzc{a}_2, \ldots, \mathpzc{b}_2)$. This intertwining operator is a composition of rank one intertwining operators associated to permutations which have non-generic kernel (see Example \ref{exampleio}); as composition of those rank one operators, it has non-generic kernel.

Similarly, because $\mathpzc{a}_1 > \mathpzc{a}_2$, there exists an intertwining operator with non-generic kernel from the module induced with cuspidal support $(\mathpzc{a}_2, \ldots, \mathpzc{b}_1)(\mathpzc{a}_1, \ldots, \mathpzc{b}_2)$ to the one induced with cuspidal support
$(\mathpzc{a}_1, 
\ldots, \mathpzc{b}_1)(\mathpzc{a}_2, \ldots, \mathpzc{b}_2)$.

Let us now assume the result of this lemma true for $t$ linear residual segments. Consequently, there exists an intertwining operator with non-generic kernel from $I_{P_1}^G(\sigma(\mathcal{S}'_1, \mathcal{S}'_2, \ldots, \mathcal{S}'_t,\mathcal{S}_{t+1}; \underline{n}))$ to $I_{P_1}^G(\sigma(\mathcal{S}_1, \mathcal{S}_2, \ldots, \mathcal{S}_t, \mathcal{S}_{t+1}); \underline{n}))$.
In this case $\mathcal{S}_{t+1}$ and $\mathcal{S}'_t$ may be linked and taking union and intersection of them yields $\mathcal{S}'_{t+1}$ and $\mathcal{S}''_t$ and the existence of an intertwining operator with non-generic kernel from $I_{P_1}^G(\sigma(\mathcal{S}'_1, \mathcal{S}'_2, \ldots, \mathcal{S}''_t,\mathcal{S}'_{t+1};\underline{n}))$ to $I_{P_1}^G(\sigma(\mathcal{S}'_1, \mathcal{S}'_2, \ldots, \mathcal{S}'_t, \mathcal{S}_{t+1}; \underline{n}))$. 
The latter argument is repeated if $\mathcal{S}''_t$ and $\mathcal{S}'_{t-1}$ are linked, and so on.

Another case to consider would be $\mathcal{S}_{t+1} \subset \mathcal{S}'_t$  with $s_{t+1} \leq s'_t$, and $\mathcal{S}_{t+1}$ linked to $\mathcal{S}'_{t-1}$. Then, using the irreducibility of the induced from the two segments $\mathcal{S}'_t$ and $\mathcal{S}'_{t-1}$, one would interchange them, then deal with the intersection and union of $\mathcal{S}_{t+1}$ and $\mathcal{S}'_{t-1}$, obtain $\mathcal{S}'_{t+1}$ and $\mathcal{S}''_{t-1}$ and the existence of an intertwining operator from $I_{P_1}^G(\sigma(\mathcal{S}'_1, \mathcal{S}'_2, \ldots, \mathcal{S}'_t, \mathcal{S}''_{t-1}, \mathcal{S}'_{t+1}),\underline{n}))$ to $I_{P_1}^G(\sigma(\mathcal{S}'_1, \mathcal{S}'_2, \ldots, \mathcal{S}'_t, \mathcal{S}'_{t-1}, \mathcal{S}_{t+1}, \underline{n}))$.

Since the resulting segments $\mathcal{S}'_t$, $\mathcal{S}''_{t-1}$ and $\mathcal{S}'_{t+1}$ are unlinked, we can organize them so that their exponents are ordered. If $\mathcal{S}''_{t-1}$ is linked to any $\mathcal{S}'_i$, $i\neq t,t+1$, we repeat this argument. 

Eventually there exists an intertwining operator with non-generic kernel from
$I_{P_1}^G(\sigma(\mathcal{S}^*_1, \mathcal{S}^*_2, \ldots, \mathcal{S}^*_t,\mathcal{S}^*_{t+1}; \underline{n}))$ to $I_{P_1}^G(\sigma(\mathcal{S}_1, \mathcal{S}_2, \ldots, \mathcal{S}_t, \mathcal{S}_{t+1}; \underline{n}))$, where $(\mathcal{S}^*_1, \mathcal{S}^*_2, \ldots, \mathcal{S}^*_t,\mathcal{S}^*_{t+1})$ is the sequence of $t+1$ unlinked segments obtained at the end of the procedure of taking intersection and union.
\end{proof}

\subsection{A Lemma in the vein of Zelevinsky's Theorem}

Recall this fundamental result of Zelevinsky, for the general linear group, which was also presented as Theorem 5 in \cite{rodierb}. We use the notation introduced in Definition \ref{segment}.
\begin{prop}[Zelevinsky, \cite{Z}, Theorem 9.7]\label{rodierbZ}
 If any two segments, $\mathcal{S}_i, \mathcal{S}_j$, $j,i $ in $\left\{1, \ldots, n \right\}$ of the linear group are not linked, we have the irreducibility of $Z(\mathcal{S}_1)\times 
Z(\mathcal{S}_2) \times 
\ldots \times Z(\mathcal{S}_n)$ and conversely if $Z(\mathcal{S}_1)\times Z(\mathcal{S}_2) \times \ldots \times Z(\mathcal{S}_n)$ is irreducible, then all segments 
are mutually 
unlinked.
\end{prop}

Here, we prove a similar statement in the context of any quasi-split reductive group of type $A$.

\begin{lemma}\label{zelevinsky}
Let $\tau$ be an irreducible generic discrete series of a standard Levi subgroup $M$ in a quasi-split reductive group $G$.
Let $\sigma$ be an irreducible unitary generic cuspidal representation of a standard Levi subgroup $M_1$ in the cuspidal support of $\tau$. Let us assume $\Sigma_{\sigma}$ is irreducible of rank $d=rk_{ss}(G)-rk_{ss}(M_1)$ and type $A$.

Let $\underline{s}= (s_1, s_2, \ldots ,s_t) \in a_{M_1}^*$ be ordered such that $s_1 \geq s_2 \geq \ldots \geq s_t$ with $s_i = \frac{\mathpzc{a}_i + 
\mathpzc{b}_i}{2}$, for two real numbers $\mathpzc{a}_i \geq \mathpzc{b}_i$.

Then $I_P^G(\tau_{\underline{s}})$ is a generic standard module embedded in $I_{P_1}^G(\sigma_{\lambda})$ and $\lambda$ is composed of $t$ residual segments $\left\{(\mathpzc{a}_i,\ldots, \mathpzc{b}_i), i=1, \ldots, t \right\}$ of type $A_{n_i}$.

Let us assume that the $t$ segments are mutually unlinked. Then $\lambda$ is not a residual point and therefore the unique irreducible 
generic subquotient 
of the generic module $I_{P_1}^G(\sigma_{\lambda})$, is not a discrete series.
This irreducible generic subquotient is $I_P^G(\tau_{\underline{s}})$.
In other words, the generic standard module $I_P^G(\tau_{\underline{s}})$ is irreducible. Further, for any reordering $\underline{s'}$ of the 
tuple $
\underline{s}$, which corresponds to an element $w \in W$ such that $w\underline{s}=\underline{s'}$ and discrete series $\tau'$ of $M'$ such 
that $w\tau=
\tau'$ and $wM=M'$, we have  $I_{P'}^G(\tau'_{\underline{s'}})\cong I_P^G(\tau_{\underline{s}})$.
\end{lemma}

\begin{proof}
By the result of Heiermann-Opdam (Proposition \ref{opdamh}), there exists a standard parabolic subgroup $P_1$, a unitary cuspidal representation $\sigma$, a parameter $\nu \in 
\overline{(a_{M_1}^M*)^+}$ such that the generic discrete series $\tau$ embeds in $I_{M_1 \cap M}^M(\sigma_{\nu})$. 
By Heiermann's Theorem (see Theorem \ref{heir}), $\nu$ is a residual point so it is composed of residual segments of type $A_{n_i}$. Then twisting by $\underline{s}$ and 
inducing to $G$, we 
obtain:

$$I_P^G(\tau_{\underline{s}}) \hookrightarrow I_{P_1}^G(\sigma_{\lambda}) ~ \text{where} ~ \lambda= (\mathpzc{a}_i, \ldots, \mathpzc{b}_i)_{i=1}^t$$

Let $\pi$ be the unique irreducible generic subquotient of the generic standard module $I_P^G(\tau_{\underline{s}})$. Then using Langlands' classification and the standard module conjecture $\pi = J(P',\tau',\nu') \cong I_{P'}^G(\tau'_{\nu'})$. Assume $\tau'$ is discrete series. We apply again the result of Heiermann-Opdam to this generic discrete series to embed $I_{P'}^G(\tau'_{\nu'})$ in 
$I_{P_1'}
^G(\sigma'_{\lambda'})$.

As any representation in the cuspidal support of $\tau_{\underline{s}}$ must lie in the cuspidal support of $\pi$, any such representation much 
be conjugated 
to $\sigma'_{\lambda'}$, therefore $\lambda'$ is in the Weyl group orbit of $\lambda$. Let us consider this Weyl group orbit under the 
assumption that the $t$ 
segments $\left\{(\mathpzc{a}_i,\ldots, \mathpzc{b}_i), i=1, \ldots, t \right\}$ are unlinked.

Whether the union of any two segments in $\left\{(\mathpzc{a}_i, \ldots,\mathpzc{b}_i), i=1, \ldots, t \right\}$ is not a segment, or the segments are mutually 
included in one 
another, it is clear there are no option to take intersections and unions to obtain new linear residual segments.
Further, starting with $\lambda$, to generate new elements in its $W_{\sigma}$-orbit, one can split the segments $\left\{(\mathpzc{a}_i, \ldots, \mathpzc{b}_i), i=1, 
\ldots, t \right\}$. 
By Lemma \ref{split}, this procedure yields necessarily larger Langlands parameters.
Therefore, there is no option to reorganize them to obtain residual segments $(\mathpzc{a}'_j, \mathpzc{b}'_j)$ of type $A_{n'_j}$ such that $n'_j \neq n_i$ 
for some $i \in \left\{1, \ldots, t\right\}$ and $j \in \left\{1,\ldots, s \right\}$, for some $r$ such that $\sum_{j=1}^rn'_j = \sum_{i=1}^tn_i$.

The second option is to permute the order of the segments $\left\{(\mathpzc{a}_i, \ldots, \mathpzc{b}_i), i=1, \ldots, t \right\}$ to obtain any other parameter $
\lambda'$ in the 
Weyl group orbit of $\lambda$. From this $\lambda'$, one clearly obtains the parameter $\nu':= \underline{s'}$ as a simple permutation of the 
tuple $
\underline{s}$.

On the Langlands parameter $\underline{s}$, which is the unique among the $(\nu')$'s  described in the previous paragraph in the Langlands 
situation (we 
consider all standard modules $I_{P'}^G(\tau'_{\nu'})$), we can use Theorem \ref{minLP=irr} to conclude that the generic standard module 
$I_P^G(\tau_{\underline{s}})$ for $\nu=\underline{s}$ is irreducible.

Now, we want to show $I_{P'}^G(\tau'_{\underline{s'}})$ is isomorphic to $I_P^G(\tau_{\underline{s}})$. 

Looking at the cuspidal support, it is clear that there exists a Weyl group element in $W(M, M')$ sending $\sigma_{\lambda}$ to 
$\sigma'_{\lambda'}$, and therefore $\tau_{\underline{s}}$ to the Langlands data $(w\tau)_{w\underline{s}}:= \tau'_{\underline{s'}}$.

Consider first the case of a maximal parabolic subgroup $P$ in $G$.
Set $\underline{s} = (s_1, s_2)$, $\underline{s'} = (s_2, s_1)$ and $\tau'$ is a generic discrete series representation. 
We apply the map $t(w)$ between $I_{P}^G(\tau_{\underline{s}})$ and $I_{wP}^G((w\tau)_{w\underline{s}})$ 
which is an isomorphism. By definition, the parabolic $wP$ has Levi $M'$. Then, by Lemma 5.4 \cite{BDK} (see also the Remark 2.10 in \cite{BZI}) since the Levi subgroups and inducing representations are the same, the Jordan-H\"older composition series of $I_{wP}^G(\tau'_{\underline{s'}})$ and $I_{P'}^G(\tau'_{\underline{s'}})$ are the same, and since $I_{P}^G(\tau_{\underline{s}})$ is irreducible, they are isomorphic and irreducible.

Secondly, consider the case when the two parabolic subgroups $P$ and $P'$, with Levi subgroup $M$ and $M'$, are connected by a sequence of 
adjacent parabolic subgroups of $G$. Using Theorem \ref{minLP=irr} with any Levi subgroup in $G$, in particular a  Levi subgroup $M_{\alpha}$ (containing $M$ as a maximal Levi subgroup) shows that 
the 
representation $I_{P\cap M_{\alpha}}^{M_{\alpha}}(\tau_{\underline{s}})$ is irreducible.

Then, we are in the context of the above paragraph and $I_{s_{\alpha}(\overline{P}\cap M_{\alpha})}^{M_{\alpha}}((s_{\alpha}\tau)_{s_{\alpha}
\underline{s}})$ 
(the image of the composite of the map $J_{\overline{P}\cap M_{\alpha}|P\cap M_{\alpha}}$ with the map $t(s_{\alpha})$) is irreducible, and 
isomorphic to 
$I_{P\cap M_{\alpha}}^{M_{\alpha}}(\tau_{\underline{s}})$. 

Let us denote $Q$ the parabolic subgroup adjacent to $P$ along $\alpha$. Induction from $M_{\alpha}$ to $G$ yields that $I_{Q}^G(s_{\alpha}
\tau)_{s_{\alpha}
\underline{s}})$ is isomorphic to $I_P^G(\tau_{\underline{s}})$.
Writing the Weyl group element $w$ in $W(M,M')$ such that $wM= M'$ as a product of elementary symmetries $s_{\alpha_i}$, and applying a 
sequence of 
intertwining maps as above yields the isomorphism between $I_P^G(\tau_{\underline{s}})$ and $I_{P'}^G(\tau'_{\underline{s'}})$.
\end{proof}

\begin{rmk}
For an example, see \cite{BZI}, 2.6.
\end{rmk}

\section{Conditions on the parameter $\lambda$ so that the unique irreducible generic subquotient of $I_{P_1}^G(\sigma_{\lambda})$ is a subrepresentation} \label{cond}

The goal of this section is to present specific forms of the parameter $\lambda \in a_{M_1}^*$ such that the unique irreducible generic subquotient of $I_{P_1}^G(\sigma_{\lambda})$ with $\sigma$ irreducible unitary generic cuspidal representation of any standard Levi $M_1$ is a subrepresentation. There is an obvious choice of parameter satisfying this condition as it is proven in the following Lemma:

\begin{lemma}\label{firstresultslem} 
Let $\sigma$ be an irreducible generic cuspidal representation of $M_1$ and $\sigma_{\lambda}$ be a dominant residual point and consider the generic 
induced 
module $I_{P_1}^G(\sigma_{\lambda})$. Its unique irreducible generic square-integrable subquotient is a subrepresentation.
\end{lemma}

\begin{proof}
From Theorem \ref{heir}, since $\lambda$ is a residual point, $I_{P_1}^G(\sigma_{\lambda})$ has 
a discrete series subquotient. From Rodier's Theorem, it also has a unique irreducible generic subquotient, denote it $\gamma$.

From Theorem \ref{muic}, this unique irreducible generic subquotient is discrete series.
Consider this unique generic discrete series subquotient, by Proposition ~\ref{opdamh}, there exists a parabolic subgroup $P'$ such that $\gamma 
\hookrightarrow 
I_{P'}^G(\sigma'_{\lambda'})$, and $\lambda'$ dominant for $P'$.
Then the lemma follows from Proposition \ref{prop} in Section \ref{io1}.
\end{proof}

We need the following definition: 
\begin{dfn} \label{CS}
Let $(M_1, \sigma)$ be in the generic cuspidal support of an irreducible generic discrete series. 

Let us denote $M_1 = M_{\Theta}$. Let us assume that $\Theta = \bigcup_{i=1}^n\Theta_i$, where, for any $i \in \left\{1, \ldots, n-1 \right\}$, $\Theta_i$ is an irreducible component of type $A$.

We say this cuspidal support satisfies the conditions \emph{(CS)} (given in Proposition \ref{H} and Corollary \ref{mmm}) if:
\begin{itemize}
	\item $\Sigma_{\sigma}$ is irreducible of rank $d$.
	\item If $\Delta(P_1)= \left\{\alpha_1, \ldots, \alpha_{d-1}, \beta_d \right\}$ then $\Delta_{\sigma} = \left\{\alpha_1, \ldots, \alpha_{d-1}, \alpha_d \right\}$, where $\alpha_d$ can be different from $\beta_d$ if $\Sigma_{\sigma}$ is of type $B,C,D$.
	
	\item  For any $i \in \left\{1, \ldots, n-1 \right\}$, $\Theta_i$ has fixed cardinal. Furthermore, the interval between any two disjoint consecutive components $\Theta_i$, $\Theta_{i+1}$ is of length one. 
\end{itemize}
\end{dfn}

Our main result in this section is the following theorem.

\begin{thm} \label{conditionsonlambda}
Let us consider $I_{P_1}^G(\sigma_{\lambda})$ with $\sigma$ irreducible unitary generic cuspidal representation of a standard Levi $M_1$, and $\lambda \in a_{M_1}^*$ such that $(M_1,\sigma)$ satisfies the conditions $(CS)$ (see Definition \ref{CS}). Let $W_{\sigma}$ be the Weyl group of the root system $\Sigma_{\sigma}$.
The unique irreducible generic subquotient of $I_{P_1}^G(\sigma_{\lambda})$ is necessarily a subrepresentation if the parameter $\lambda$ is one of the following:
\begin{enumerate}
\item If $\lambda$ is a residual point:
\begin{enumerate}[label=(\alph*)] 
\item $\lambda$ is a dominant residual point.
\item $\lambda$ is a residual point of the form $(a, a_-)(\underline{n})$ with $(a,a_-)$ two consecutive jumps in the Jumps set associated to the dominant residual point in its $W_{\sigma}$-orbit.
\item $\lambda$ is a residual point of the form $(a, b)(\underline{n})$ such that the dominant residual point in its $W_{\sigma}$-orbit has associated Jumps set containing $(a,a_-)$ as two consecutive jumps and $b> a_-$.
\end{enumerate}
\item If $\lambda$ is not a residual point
\begin{enumerate}[label=(\alph*)] 
\item $\lambda$ is of the form $(a',b')(\underline{n'})$ such that the Langlands' parameter $\nu' = \frac{a'+b'}{2}$ is minimal for the order on Langlands parameter (see Subsection \ref{orderonL})
\item If $\lambda$ is of the form $(a,b)(\underline{n})$ with $a=a', b'<b$ in the $W_{\sigma}$-orbit of a parameter as in   (2).a).
\end{enumerate}
\end{enumerate}
\end{thm}

The proof of this theorem given in Subsection \ref{conditions}, relies on Moeglin's extended lemmas and an embedding result (\ref{embedding}).

\subsection{On some conditions on the standard Levi $M_1$ and some relationships between $W(M_1)$ and $W_{\sigma}$} \label{onsomeconditions}

Let $G$ be a quasi-split reductive group over $F$ (resp. a product of such groups) whose root system $\Sigma$ is of type $A,B,C$ or $D$, $\pi_0$ is an irreducible generic discrete series of $G$ whose cuspidal support contains the representation $\sigma_{\lambda}$ of a standard Levi subgroup $M_1$, where $\lambda \in a_{M_1}^*$ and $\sigma$ is an irreducible unitary cuspidal generic representation. 

Let $$d=rk_{ss}(G)-rk_{ss}(M_1)= \dim a_{M_1} -\dim a_G$$ 

Let us denote $M_1= M_{\Theta}$.
Then $\Delta - \Theta$ contains $d$ simple roots.

Let us denote $\Delta(P_1)$ the set of non-trivial restrictions (or projections) to $A_{M_1}$ (resp. to $a_{M_1}^G$) of simple roots in $\Delta$ such that elements in $\Sigma(P_1)$ (roots which are positive for $P_1$) are linear combinations of simple roots in $\Delta(P_1)$.

Let us denote $\Delta(P_1)= \left\{\alpha_1, \ldots, \alpha_{d-1}, \beta_d \right\}$ and $\underline{\alpha}_i$ the simple root in $\Delta$ which projects onto $\alpha_i$ in $\Delta(P_1)$.

As $(M_1,\sigma_{\lambda})$ is the cuspidal support of an irreducible discrete series, as explained in the Proposition \ref{1.13bis}, the set $\Sigma_{\sigma}$ is a root system of rank $d$ in $\Sigma(A_{M_1})$ and its basis, when we set $\Sigma(P_1)\cap\Sigma_{\sigma}$ as the set of positive roots for $\Sigma_{\sigma}$, is $\Delta_{\sigma}$.

\begin{prop} \label{H}
With the context of the previous paragraphs, let $\Sigma_{\sigma}$ be irreducible.
If $\Delta(P_1)= \left\{\alpha_1, \ldots, \alpha_{d-1}, \beta_d \right\}$ then $\Delta_{\sigma} = \left\{\alpha_1, \ldots, \alpha_{d-1}, \alpha_d \right\}$, where $\alpha_d$ can be different from $\beta_d$ if $\Sigma_{\sigma}$ is of type $B,C,D$.
\end{prop}

\begin{proof}
This is a result of the case-by-case analysis conducted in the independent paper \cite{PrSR}, where $\Delta_{\Theta}$ denotes there the $\Delta(P_1)$ considered in this Proposition. From its definition $\Sigma_{\sigma}$ is a subsystem in $\Sigma_{\Theta}$.
If $\Sigma_{\Theta}$ contains a root system of type $BC_d$, it is clear that the last root, denoted $\alpha_d$, of this system (which is either the short of long root depending on the chosen reduced system) can be different from $\beta_d$ if $\Sigma_{\sigma}$ is of type $D_d$.
\end{proof}

We have not included the root $\beta_d$ in $\Delta_{\sigma}$ because (as opposed to the context of classical groups) it is possible that there exists $\sigma$ an irreducible cuspidal representation such that $s_{\beta_d}\sigma \ncong \sigma$.

A typical example of the above Proposition (\ref{H}) is when $\Sigma$ if of type $B,C$ and $\Sigma_{\sigma}$ is of type $D$, then it occurs that $\Delta(P_1)$ contains $\beta_d = e_d$ or $\beta_d = 2e_d$ whereas $\Delta_{\sigma}$ contains $\alpha_d = e_{d-1} + e_d$.

\vspace{0.4cm}
This proposition allows us to use our results on intertwining operators with non-generic kernel (see Proposition \ref{nngenerick}, and Example 
\ref{exampleio}).

In the context of Harish-Chandra's Theorem \ref{hc}, the element denoted $s_{\alpha}$ corresponds to the element $\widetilde{w_{0}}^{(M_1)_{\alpha}}\widetilde{w_{0}^{M_1}}$ as defined in Chapter 1 in \cite{shabook}. 

Let us describe it: \\

Let $P= MN$ be a standard parabolic.
Let $\Theta \subset \Delta$, $M= M_{\Theta}$. In \cite{shabook}, Shahidi defines $\widetilde{w_0}$ as the element in $W(A_0, G)$ which sends $\Theta$ to a subset of $\Delta$ but every other root $\beta \in \Delta - \Theta$ to a negative root. 

If $\widetilde{w_{0}^{G}}$, $\widetilde{w_{0}^{M}}$ are the longest elements in the Weyl groups of $A_0$ in $G$ and $M$, respectively, then $\widetilde{w_0}= \widetilde{w_{0}^{G}}\widetilde{w_{0}^{M}}$. The length of this element in $W$ is the difference of the lengths of each element in this composition. Therefore, if a representative of this element in $G$ normalizes $M$, since it is of minimal length in its class in the quotient $\left\{w \in W |w^{-1}Mw= M\right\}\slash W^M$,  this representative belongs to $W(M)$.

When $P$ is maximal and self-associate (meaning $\widetilde{w_0}(\Theta) = \Theta$) then if $\alpha$ is the simple root of $A_{M}$ in Lie($N$), $\widetilde{w_0}(\alpha) = -\alpha$.
In this case $w_0 N w_0^{-1}= N^{-}$, the opposite of $N$ for $w_0$ a representative of $\widetilde{w_0}$ in $G$.

\begin{rmk} \label{remD}
Applying the previous paragraph to the context of $P_1\cap(M_1)_{\beta}$ and $(M_1)_{\beta}$, we first observe that $\widetilde{w_{0}}^{(M_1)_{\beta}}\widetilde{w_{0}^{M_1}}(\Theta) = \Theta$. Then, one notices that $\widetilde{w_{0}}^{(M_1)_{\beta}}\widetilde{w_{0}^{M_1}}$ sends $\beta$ to $-\beta$.

In analogy with the notations of Theorem \ref{hc}, let us denote $\widetilde{w_{0}}^{(M_1)_{\beta}}\widetilde{w_{0}^{M_1}} = s_{\beta}$, we have:
$s_{\beta}(P_1\cap(M_1)_{\beta})= \overline{P_1}\cap(M_1)_{\beta}$, then $s_{\beta}\lambda = \lambda$ if $\lambda$ is in $\overline{(a_{M_1}^G*)^+}$ and is a residual point of type $D$.
\end{rmk}


By definition, if $\alpha \in \Sigma_{\sigma}$, by Harish-Chandra's Theorem \ref{hc}, $s_{\alpha}(P_1 \cap (M_1)_{\alpha}) = \overline{P_1}\cap (M_1)_{\alpha}$ and $s_{\alpha}.M_1 = M_1$, and this means that $s_{\alpha}$ is a representative in $G$ of a Weyl group element sending $\Theta$ on $\Theta$.

\begin{cor} \label{mmm}
Let $\sigma$ be an irreducible cuspidal representation of a standard Levi subgroup $M_1$ and let us assume that $\Sigma_{\sigma}$ is irreducible of rank $d=rk_{ss}(G)-rk_{ss}(M_1)$ and type $A,B,C$ or $D$, then:

\begin{enumerate}
	\item For any $\alpha$ in $\Delta(P_1)$, $s_{\alpha} \in W(M_1)$.
	\item $W(M_1)= W_{\sigma}\cup\left\{s_{\beta_d}W_{\sigma} \right\}$. 
   \item Let $\sigma'$ (resp. $\sigma$) be an irreducible cuspidal representation of a standard Levi subgroup $M_1'$ (resp. standard Levi subgroup $M_1$). 
Let us assume they are the cuspidal support of the same irreducible discrete series. Then $M'_1 = M_1$.
\end{enumerate}

\end{cor}

\begin{proof} \textbf{Point (1):} 

Let us assume $\Theta$ has the form given in Appendix \ref{lab}, Theorem \ref{mainlab}, that is a disjoint union of irreducible components: $\bigcup_{i=1}^n\Theta_i$. Then, let us show that for any $\alpha$ in $\Delta(P_1)$, $s_{\alpha} \in W(M_1)$.

By definition, $s_{\alpha}$ is a representative in $G$ of the element  $\widetilde{w_0^{(M_1)_{\alpha}}}\widetilde{w_0^{(M_1)}}$.

Let us first assume that $\alpha_i$ is the restriction of the simple root connecting $\Theta_{i}$ and $\Theta_{i+1}$, both of type $A$, in the Dynkin diagram of $G$. Then $$\Delta^{(M_1)_{\alpha_i}}=\Theta_i\cup\left\{\alpha_i\right\}\cup\Theta_{i+1}\bigcup_{j\neq i, i+1}\Theta_j$$

The element $\widetilde{w_0^{M_1}}$ operates on each component as the longest Weyl group element for that component: it sends $\alpha_k \in \Theta_i$ to $-\alpha_{\ell_i+1-k}$ if $\ell_i$ is the length of the connected component $\Theta_i$. 

In a second time, $\widetilde{w_0^{(M_1)_{\alpha_i}}}$ operates on $\Theta_i\cup\left\{\alpha_i\right\}\cup\Theta_{i+1}$ in a similar fashion, and trivially on each component in $\bigcup_{j\neq i, i+1}\Theta_j$.

Secondly, let us assume that $\beta$ is the restriction of the simple root connecting $\Theta_{n-1}$ of type $A$ and $\Theta_{n}$ of type $B,C$ or $D$ in the Dynkin diagram of $G$. 

$\widetilde{w_0^{(M_1)}}(\Theta_{n-1}) = \Theta_{n-1}$ (since this element simply permutes and multiply by (-1) the simple roots in $\Theta_{n-1}$), while $\widetilde{w_0^{(M_1)}}(\Theta_{n}) = -\Theta_{n}$. Further, $\widetilde{w_0^{(M_1)_{\beta}}}$ acts as (-1) on all the simple roots in $\Theta_{n-1}\cup\Theta_n$. 

Eventually, $\widetilde{w_0^{(M_1)_{\beta}}}\widetilde{w_0^{(M_1)}}$ fixes $\Theta_{n}$ pointwise and sends each root in $\Theta_{n-1}$ to another root in $\Theta_{n-1}$. It also fixes pointwise $\bigcup_{j\neq n-1,n}^n\Theta_j$. Therefore, for any $\alpha$ in $\Delta(P_1)$, $\widetilde{w_0^{(M_1)_{\alpha}}}\widetilde{w_0^{(M_1)}}(\Theta) = \Theta$, hence $s_{\alpha} \in \left\{w \in W |w^{-1}M_1w= M_1\right\}$. Furthermore, since the length of this element is the difference of the lengths of each element in this composition, it is clear that $s_{\alpha}$ is of minimal length in its class in the quotient $\left\{w \in W |w^{-1}M_1w= M_1\right\}\slash W^{M_1}$, hence this element is in $W(M_1)$.


\textbf{Point (2)}

Any element in $W(M_1)$ is a representative of minimal length in its class in the quotient $\left\{w \in W |w^{-1}M_1w= M_1\right\}\slash W^{M_1}$. The $s_{\alpha} = \widetilde{w_0^{(M_1)_{\alpha}}}\widetilde{w_0^{(M_1)}}$ described above where the elements $\alpha \in \Delta(P_1)$ are a set of generators of $W(M_1)$.
Recall from Proposition \ref{H} that $\Delta(P_1)= \left\{\alpha_1, \ldots, \alpha_{d-1}, \beta_d \right\}$ and $\Delta_{\sigma} = \left\{\alpha_1, \ldots, \alpha_{d-1}, \alpha_d \right\}$, where $\alpha_d$ can be different from $\beta_d$ if $\Sigma_{\sigma}$ is of type $B,C,D$. 
Therefore, $W(M_1)= W_{\sigma}\cup\left\{s_{\beta_d} W_{\sigma}\right\}$.

We also recall that in the context of $\Sigma_{\sigma}$ of type $D_d$ and $\Sigma(P_1)$ of type $B_d$ or $C_d$ : $s_{\alpha_d} = s_{\alpha_{d-1}}s_{\beta_d}s_{\alpha_{d-1}}s_{\beta_d}$.

\textbf{Point (3)}

Let us denote $M_1' = M_{\Theta'}$, and $M_1 = M_{\Theta}$ and assume that $\Theta$ and $\Theta'$ are written as $\bigcup_{i=1}^n\Theta_i$, where, for any $i \in \left\{1, \ldots, n-1 \right\}$, $\Theta_i$ is an irreducible component of type $A$.

Since the cuspidal data are the support of the same irreducible discrete series, by Theorem 2.9 in \cite{BZI}, there exists $w \in W^G$ such that $M_1' = w.M_1$, $\sigma'= w.\sigma$. 
Since $M_1'$ is isomorphic to $M_1$, $\Theta'$ is isomorphic to $\Theta$.

Hence, applying the observations made in the first part of the proof of this Proposition to $M_1$ and $M_1'$, we observe $\Theta$ and $\Theta'$ share the same constraints: their components of type $A$ are all of the same cardinal and the interval between any two of these consecutive components is of length one.
Also, since $\Theta'$ is isomorphic to $\Theta$, its last component $\Theta'_m$ is of the same type as $\Theta_m$. Therefore, $\Theta' = \Theta$. Hence $M_1 = M_1'$.
\end{proof}

\begin{rmk} \label{rmkP1}
This implies that if $P_1=M_1U_1$ and $P_1'=M_1'U_1'$ are both standard parabolic subgroups such that their Levi subgroups satisfy the conditions of the previous Proposition, they are actually equal.
\end{rmk}

\subsection{A few preliminary results for the proof of Moeglin's extended lemmas}

Let us recall Casselman's square-integrability criterion as stated in \cite{wald} whose proof can be found in (\cite{casselman},(4.4.6)).
Let $\Delta(P)$ be a set of simple roots, then $~^+\!a^G_P*$, resp.$~^+\!\overline{a}^G_P*$, denote the set of $\chi$ in $a_M^*$ of the 
following form: $
\chi= \sum_{\alpha \in \Delta(P)} x_{\alpha}\alpha$ with $x_{\alpha} > 0$, resp. $x_{\alpha} \geq 0$. Further, denote $\pi_P$ the Jacquet 
module of $\pi$ with 
respect to $P$, and $\mathcal{E}xp$ the set of exponents of $\pi$ as defined in Section I.3 in \cite{wald}.

\begin{prop}[Propositions III.1.1 and III.2.2 in \cite{wald}]\label{casselmansi}
Let $\pi$ be an irreducible representation with unitary central character. The following conditions are equivalent:
\begin{enumerate}
\item $\pi$ is square-integrable (resp. tempered);

\item for any semi-standard parabolic subgroup $P=MU$ of $G$, and for any $\chi$ in $\mathcal{E}xp(\pi_P)$, $Re(\chi) \in ~^+\!a^G_P*$ (resp. $Re(\chi) \in ~^+\!\overline{a}^G_P*$).

\item for any standard parabolic subgroup $P=MU$ of $G$, proper and maximal, and for any $\chi$ in $\mathcal{E}xp(\pi_P)$, $Re(\chi) \in ~^
+\!a^G_P*$ (resp. $Re(\chi) \in ~^
+\!\overline{a}^G_P*$).
\end{enumerate}
\end{prop}

In the following two lemmas we will apply the previous Proposition as follows:

\begin{prop} \label{CC}
Let $\pi_0$ embed in $I_{P_1}^G(\sigma_{\lambda})$. Let us write the parameter $\lambda$ as a vector in the basis $\left\{e_i \right\}_{i\geq 0}$ (the basis of $a_{M_1}^*$ as chosen in the Definition \ref{rs}, for instance) as $
((x,y)+\underline{\lambda})$ for a linear residual segment $(x,y)$, and assume $\sum_{k \in [x,y]} k \leq 0$. Then $\pi_0$ is not square-integrable.

\end{prop}

\begin{proof}
Indeed, if $$\pi_0 \hookrightarrow I_{P_1}^G(\sigma((x, y)+ \underline{\lambda}))$$ by Frobenius reciprocity, the character $\chi_{\lambda}$ appears as exponent of the Jacquet module of $\pi_0$ with respect to $P_1$.
Let us write $\lambda$ as $$\sum_i x_i(e_i - e_{i+1}) + \underline{\lambda} =\sum_i y_ie_i + \underline{\lambda}$$  it is clear that, for any integer $j$, $x_j= \sum_{i=1}^j y_j$, and 
notice there is an 
index $j'$ such that $x_{j'} = \sum_{k \in [x,y]} k$.
Therefore, using the hypothesis of the Proposition, $x_{j'} = \sum_{k \in [x,y]} k \leq 0$. But then $\chi_{\lambda}$ does not satisfy the 
requirement of 
Proposition \ref{casselmansi} since $x_{j'}$ is negative.
\end{proof}

We will also use the following well-known result:
\begin{thm}[\cite{renard}, Theorem VII.2.6] \label{renard}
Let $(\pi, V)$ be a admissible irreducible representation of $G$. Then $(\pi, V)$ is tempered if and only if there exists a standard parabolic 
subgroup of $G$, 
$P=MN$, and a square integrable irreducible representation $(\sigma, E)$ of $M$ such that $(\pi, V) $ is a subrepresentation of 
$I_P^G(\sigma)$.
\end{thm}

%
%
%

\begin{lemma} \label{lemmabeta}
Let $\beta \in \Delta(P_1)$, and assume $\beta \notin \Delta_{\sigma}$, then the elementary intertwining operator associated to $s_{\beta} \in W$ is bijective at $\sigma_{\lambda} ~~ \forall \lambda \in a_{M_1}^*$.
\end{lemma}


\begin{proof}
Set $s=s_{\beta}$ for $\beta \in \Delta(P_1)$, and $\beta \notin \Delta_{\sigma}$.
Recall we have $J_{P_1|sP_1}J_{sP_1|P_1}$ equals $(\mu^{(M_1)_{\beta}})^{-1}$ up to a multiplicative constant.

Recall $\mathcal{O}$ denotes the set of equivalence classes of representations of the form $\sigma\otimes\chi$ where $\chi$ is an unramified character of $M_1$. The operator $\mu^{(M_1)_{\beta}}J_{P_1|s_{\beta}P_1}$ is regular at each unitary representation in $\mathcal{O}$ (see \cite{wald}, V.2.3), 
$J_{s_{\beta}
P_1|P_1}$ is itself regular on $\mathcal{O}$, since this operator is polynomial on $X^{nr}(G)$. By the general result mentioned after Lemma \ref{1.8}, the function $\mu^{(M_1)_{\beta}}$ has a pole at $\sigma_{\lambda}$ for $\lambda$ on 
the positive 
real axis, if $\mu^{(M_1)_{\beta}}(\sigma) = 0$. Therefore, by definition, since $\beta \notin \Delta_{\sigma}$, there is no pole at $
\sigma_{\lambda}$. Further, since the regular operators $J_{P_1|sP_1}$ and $J_{sP_1|P_1}$ are non-zero at any point, if $\mu^{(M_1)_{\beta}}$ does not have a 
pole at $
\sigma_{\lambda}$, these operators $J_{P_1|sP_1}$ and $J_{sP_1|P_1}$ are bijective.
\end{proof}

A consequence of this lemma is that for any root $\beta \in \Sigma(P_1)$ which admits a reduced decomposition without elements in $\Delta_{\sigma}$, the intertwining operators associated to $s_{\beta}$ are everywhere bijective.

\subsection{Extended Moeglin's Lemmas} \label{Moeglinandembedding}

In this section and the following the core of our argumentation relies on the form of the parameters $\lambda$; changes on the form of these parameters are induced by actions of Weyl group elements (see for instance Example \ref{exampleio2}). In fact, the Weyl group operates on $\sigma_{\lambda}$ and any Weyl group element decomposes in elementary symmetries $s_{\alpha_i}$ for $\alpha_i \in \Delta$. This kind of decomposition is explained in details in I.1.8 of the book \cite{MW}.
If $\alpha_i$ is in $\Delta_{\sigma}$, by Harish-Chandra's Theorem (Theorem \ref{hc}), $s_{\alpha_i}\sigma 
\cong \sigma$; however recall that for $\beta_d \in \Delta(P_1)$ (see Proposition \ref{H}), we may not have $s_{\beta_d}\sigma \cong \sigma$.

The three next lemmas, inspired by Remark 3.2 page 154 and Lemma 5.1 in Moeglin \cite{moeglinCSD} are used in our main embedding Proposition \ref{embedding} (of the irreducible generic discrete series) result.

Recall that in general $P_{\Theta'}$ is the parabolic subgroup associated to the subset $\Theta' \subset \Delta$, and $M_{\Theta'}$ contains all the roots in $\Theta'$. Recall that we denote $\underline{\alpha}_i$ the simple root in $\Delta$ which restricts to $\alpha_i$ in $\Delta(P_1)$.

\begin{lemma} \label{3.2}
Let $\pi_0$ be a generic discrete series of a quasi-split reductive group $G$ (of type $A,B,C$ or $D$) whose cuspidal support $(M_1,\sigma_{\lambda})$ satisfies the condition \emph{(CS)} (see the Definition \ref{CS}).

Let $$x,y \in \mathbb{R}, k-1= x -y \in \mathbb{N}$$ This defines the integer $k$. 

Let us denote $$M'= M_{\Delta - \left\{\underline{\alpha}_1, \ldots, \underline{\alpha}_{k-1}, \underline{\alpha}_{k}\right\}}$$ 
Let us assume there exists $w_{M'} \in W^{M'}(M_1)$, and an irreducible generic representation $\tau$ which is the unique generic subquotient of $I_{P_1\cap M'}^{M'}(\sigma_{\lambda_1^{M'}})$ such that 

\begin{equation}\label{enonce3.2}
\pi_0 \hookrightarrow I_{P'}^G(\tau_{(x,y)}) \hookrightarrow  I_{P_1}^G((w_{M'}\sigma)_{(x,y)+\lambda_1^{M'}}); ~ ~ \lambda_1^{M'} \in a_{M_1}^{M'}
\end{equation}
Let us assume $y$ is minimal for this property.

Then $\tau$ is square integrable.
\end{lemma}

\begin{proof}
Let us first remark that in Equation \ref{enonce3.2} the parameter in $a_{M_1}^*$ is decomposed as $$\underbrace{(x,y)}
_{\text{combination of} ~ \alpha_1, \ldots, \alpha_{k-1}} + \underbrace{\lambda_1^{M'}}_{\text{combination of} ~ \alpha_{k+1}, \ldots, \beta_d}$$
Let us denote $\tau$ the generic irreducible subquotient in $I_{P_1\cap M'}^{M'}(\sigma_{\lambda_1^{M'}})$, and let us show that $\tau$ is 
square integrable.

Assume on the contrary that $\tau$ is not square-integrable.

Then $\tau$ is tempered (but not square integrable) or non-tempered.
Langlands' classification [Theorem \ref{LC}] insures us that $\tau$ is a Langlands quotient $J(P'_L,\tau',\nu')$ for a parabolic subgroup $P'_L 
\supseteq P_1$ 
of $M'$ or equivalently a subrepresentation in $I_{P'_L}^{M'}(\tau'_{\nu'})$, $\nu' \in \overline{((a_{M'_L}^{M'})^*)^-}$ (equivalently $\nu' 
\bwL 0$, the 
inequality is strict in the non-tempered case). 

This is equivalent to claim there exists an irreducible generic cuspidal representation $\sigma'$, (half)-integers $\ell, m$ with $\ell - m +1 \in 
\mathbb{N}$ and 
$m \leq 0$ such that:

\begin{equation}\label{equation3.2}
\tau \hookrightarrow I_{P'_L}^{M'}(\tau'_{\nu'}) \hookrightarrow I_{P_1\cap M'}^{M'}(\sigma'((\ell, m)+{\lambda_2}^{M'}))
\end{equation}

\begin{equation}\label{star} 
	\sum_{k \in [\ell,m]} k \leq 0 
\end{equation}
We have extracted the linear segment $(\ell, m)$ out of the segment $\lambda_1^{M'}$ and named $\lambda_2^{M'}$ what is left. 

Let us justify Equation (\ref{star}): The parameter $\nu'$ reads 
$$ (\underbrace{\ldots,\frac{\ell+m}{2},\ldots}_{\ell-m+1~\mbox{times}},0, \ldots, 0)$$

$$\nu' \bwL 0 \Leftrightarrow \frac{\ell+m}{2} \leq 0 \Leftrightarrow m \leq -\ell \Leftrightarrow \sum_{k \in [\ell,m]} k \leq 0$$
From Equation (\ref{equation3.2})
\begin{equation}\label{eq32}
	\pi_0 \hookrightarrow I_{P'}^G(\tau_{(x,y)}) \hookrightarrow I_{P_1}^G(\sigma'((x,y)+(\ell,m)+\lambda_2^{M'}))
	\end{equation}
	
Since $\pi_0$ also embeds as a subrepresentation in $I_{P_1}^G(\sigma_{\lambda})$, by Theorem 2.9 in \cite{BZI} (see also 
\cite{renard} VI.5.4) there exists a Weyl group element $w$ in $W^G$ such that $w.M_1= M_1, w.\sigma'= \sigma$ and $w((x,y)+(\ell,m)+\lambda_2^{M'})= \lambda$. 
This means we can take $w$ in $W(M_1)$. But we can be more precise on this Weyl group element: from Equation (\ref{equation3.2}) and the hypothesis in the statement of the Lemma, we see we can take it in $W^{M'}(M_1)$ and it leaves the leftmost part of the cuspidal support, $\sigma_{(x,y)}$, invariant, this element therefore depends on $x$ and $y$. We denote this element $w_{M'}$.

Let 
\begin{flushleft}
$$M''= M_{\Delta- \left\{\underline{\alpha}_q, \ldots, \underline{\beta}_d\right\}} ~~ \hbox{where} ~~ q = x-y +1 + \ell - m +1.$$
\end{flushleft}

Now, let us consider two cases. First, let us assume $m \geq y$. 
If the two linear segments are unlinked and the generic subquotient in $I_{P_1\cap M''}^{M''}(\sigma'((x,y)+(\ell,m)))$ is irreducible, applying Lemma \ref{zelevinsky}, we 
can 
interchange them in the above Equation (\ref{eq32}) and we reach a contradiction to the Casselman Square Integrability criterion applied to the discrete series $\pi_0$ (considering its Jacquet module with respect to $P_1$, see Proposition \ref{CC} using $\sum_{k \in [\ell,m]}k \leq 0$).

By Proposition \ref{lambdamin} and Remark \ref{AAA}, if the two linear segments are linked the irreducible generic subquotient $\tau_{L, gen}$ of $I_{P_1\cap M''}^{M''}
((w_{M'}\sigma)((x,y)+(\ell,m)))$ embeds in $I_{P_1\cap M''}^{M''}((w.w_{M'}\sigma)((\ell,y)+(x,m)))$ (for some Weyl group element $w \in W^{M''}(M_1)$, such that 
$w.w_{M'}\sigma \cong w_{M'}\sigma$).

By Lemma \ref{Sio} there exists an intertwining operator with non generic kernel sending $\tau_{L, gen}$ to $I_{P_1\cap M''}^{M''}
((w_{M'}\sigma)((x,y)+(\ell,m)))$.
Then by [U] in $I_{P_1}^G((w_{M'}\sigma)((x,y)+(\ell,m)+\lambda_2^{M'}))$, $\pi_0$ embeds in $I_{P''}^G((\tau_{L, gen})_{\lambda_2^{M'}})$.

Therefore, inducing to $G$, we have 

$$\pi_0 \hookrightarrow I_{P''}^G((\tau_{L, gen})_{\lambda_2^{M'}}) \hookrightarrow I_{P'''}^G(I_{P_1\cap M''}^{M'''}((w_{M'}\sigma )((\ell,y)+(x,m)+
\lambda_2^{M'}))$$ but then 
since $\sum_{k \in [\ell,y]}k \leq 0$ (since $m \geq y$), we reach a contradiction to the Casselman Square Integrability criterion applied to the 
discrete series 
$\pi_0$ (considering its Jacquet module with respect to $P_1$).

Secondly, let us assume $m < y$. The induced representation $$I_{P_1\cap M''}^{M''}((w_{M'}\sigma)((x,y)+(\ell,m)))$$ is reducible only if $
\ell \in ]x, y-1]$. 
Then using Proposition \ref{lambdamin} and Remark \ref{AAA}, we know that the irreducible generic subquotient $\tau_{L, gen}$ of 
$$I_{P_1\cap M''}^{M''}((w_{M'}\sigma)((x,y)+(\ell,m)))$$ should embed in 
$$I_{P_1\cap M''}^{M''}((w_{M'}\sigma)((x,m)+(\ell,y)))$$ (or only in $I_{P_1\cap M''}^{M''}((w_{M'}\sigma)((x,m)))$ if $\ell = y-1$).

Applying Lemma \ref{zelevinsky}, we also know that it embeds in $I_{P_1\cap M''}^{M''}((w_{M'}\sigma)((\ell,y)+(x,m)))$ (we can interchange the order of the two unlinked segments $(\ell,y)$ and $(x,m)$).
Then, using Lemma \ref{Sio} and [U] as above, we embed $\pi_0$ in $I_{P''}^G((\tau_{L, gen})_{\lambda_2^{M'}}) \hookrightarrow I_{P_1}^G((w_{M'}\sigma)((x,y)+(\ell,m)+\lambda_2^{M'}))$.

But $\pi_0$ does not 
embed in $I_{P_1}^G((w_{M'}\sigma)((x,m)+(\ell,y) +\lambda_2^{M'})))$ since $y$ is minimal for such (embedding) property. 

Therefore, $\tau_{L, gen}$ rather embeds in the quotient $I_{P_1\cap M''}^{M''}((w_{M'}\sigma)((\ell,m)+(x,y)))$ of $I_{P_1\cap M''}^{M''}((w_{M'}\sigma)((x,y)+(\ell,m)))$.



Then $\pi_0$ embeds in $$I_{P''}^G((\tau_{L, gen})_{\lambda_2^{M'}}) \hookrightarrow I_{P''}^G(I_{P_1\cap M''}^{M''}((w_{M'}\sigma)((\ell,m)+(x,y))))_{\lambda_2^{M'}}= I_{P_1}^G((w_{M'}\sigma)((\ell,m)+ (x,y)+\lambda_2^{M'}))$$

Since 
$\sum_{k \in [\ell,m]}k 
\leq 0$, using Proposition \ref{CC}, we reach a contradiction.
\end{proof}

\begin{lemma} \label{moeglinlemmaextended}
Let $\pi_0$ be a generic discrete series of $G$ whose cuspidal support satisfies the conditions \emph{CS} (see the Definition \ref{CS}). Let $a, a_-$ be two consecutive jumps in the set of Jumps of $\pi_0$.

Let us assume there exists an irreducible representation $\pi'$ of a standard Levi $M'= M_{\Delta -\left\{\underline{\alpha}_1, \ldots, \underline{\alpha}_{a-a_-} \right\}}$ such that 

\begin{equation}\label{1}
\pi_0 \hookrightarrow I_{P'}^G(\pi'_{(a, a_-+1)}) \hookrightarrow I_{P_1}^G(\sigma_{(a, a_-+1)+\lambda}). 
\end{equation}

Then there exists a generic discrete series $\pi$ of $M'' = M_{\Delta - \left\{\underline{\alpha_{a+a_-+1}}\right\}}$ such that:

$\pi_0$ embeds in 
$I_{P''}^G((\pi)_{s\widetilde{\alpha}_{a+a_-}}) \hookrightarrow I_{P_1}^G(\sigma((a, -a_-)+(\underline{n})))$
 with $s= \frac{a-a_-}{2}$ 
and $(\underline{n})$ a residual segment.

\end{lemma}

We split the proof in two steps:

\subsubsection*{Step A}

We first need to show that $\pi'$ is necessarily tempered following the argumentation given in \cite{moeglinCSD}.
Assume on the contrary that $\pi'$ is not tempered. Langlands' classification [Theorem \ref{LC}] insures us that $\pi'$ is a subrepresentation in 
$I_{P_L}^{M'}
(\tau_{\nu})$, for a parabolic standard subgroup $P_L \supseteq P_1$ and  $$\nu \in ((a_{L}^{M'})^*)^-$$
This is equivalent to claim there exists $x,y$ with $x-y+1 \in \mathbb{N}$, and $y \leq 0$, a Levi subgroup 
$$L = M_{\Delta - \left\{\underline{\alpha}_1, \ldots, \underline{\alpha}_{a-a_-}\right\}\cup\left\{\underline{\alpha}_{x-y}\right\}}$$ 
a unitary cuspidal representation $w_{M'}\sigma$ in the $W(M_1)^{M'}$ group orbit of $\sigma$, and the element $\lambda \in (a_{M_1}^{M'})^*$ decomposes as $(x,y) + \lambda_1^{M'}$ such that:
\begin{equation*}
\pi' = I_{P_L}^{M'}(\tau_{\nu}) \hookrightarrow I_{P_1\cap M'}^{M'}((w_{M'}\sigma)((x, y)+\lambda_1^{M'}))
\end{equation*}

\begin{equation}
\sum_{k \in [x,y]} k < 0 
\end{equation}
The first equality in the first equation is due to the Standard module conjecture since $\pi'$ is generic. The second equation results from the 
following 
sequences of equivalences: $\nu <_{P_L} 0 \Leftrightarrow \frac{x+y}{2} < 0 \Leftrightarrow y < -x \Leftrightarrow \sum_{k \in [x,y]} k < 0$.

The element $w_{M'}$ in $W(M_1)^{M'}$ leaves the leftmost part, $\sigma_{(a, a_-+1)}$, invariant.

Then from Equation (\ref{1}) and inducing to $G$:
$$\pi_0 \hookrightarrow I_{P_1}^G((w_{M'}\sigma)((a, a_-+1)+(x, y)+ \lambda_1^{M'}))$$

We can change $(a, a_-+1)(x, y)$ to $(x, y)(a, a_-+1)$  if and only if the two segments $(a,\ldots, a_-+1)$ and $(x,\ldots, y)$ are unlinked (see the Lemma 
\ref{zelevinsky}). As $y \leq 0$, this condition is equivalent to $x \notin ]a, a_-]$.

If we can change, since $\sum_{k \in [x,y]} k < 0$, we get by Proposition \ref{CC} a contradiction to the square integrability of $\pi_0$.

Assume therefore we cannot change, then the two segments are linked by Proposition \ref{lambdalambda'}.
\begin{flushleft}
Let $M'''= M_{\Delta- \left\{\underline{\alpha}_q, \ldots, \underline{\beta}_d\right\}} ~\hbox{where}~ q = a- a_- + x-y +1 $.
\end{flushleft}
The induced representation 
$$I_{P_1\cap M'''}^{M'''}((w_{M'}\sigma)((a,\ldots, a_-+1)+(x,\ldots, y)))$$ 
has a generic submodule which is:
$$Z^{M'''}(P_1,w_L.w_{M'}\sigma, (a,\ldots, y)(x,\ldots,a_-+1))$$
 (for some Weyl group element $w_L$ such that $w_L.w_{M'}\sigma \cong w_{M'}\sigma$) 

We twist these by the character $\lambda_1^{M'}$ central for $M'''$, and therefore, by [U]: 

\begin{flushleft}
$\pi_0 \hookrightarrow I_{P'''}^G(Z^{M'''}(P_1, w_{M'}\sigma,(a,\ldots, y)(x,\ldots,a_-+1))_{\lambda_1^{M'}})$ \\
$\hookrightarrow I_{P'''}^G(I_{P_1\cap M'''}^{M'''}((w_{M'}\sigma)((a,\ldots, a_-+1)+(x,\ldots, y)))_{\lambda_1^{M'}})= I_{P_1}^G((w_{M'}
\sigma)((a,\ldots, y)+
(x,\ldots,a_-+1)+\lambda_1^{M'})$
\end{flushleft}

Let $Q'= L'U'$, we rewrite this as:
$$\pi_0 \hookrightarrow I_{Q'}^G(Z^{L'}(P_1,w'_L.w_{M'}\sigma, (a,\ldots, y)(\lambda_2^{M'}))) \hookrightarrow I_{P_1}^G((w'_L.w_{M'}
\sigma)((a,\ldots, y)+
\lambda_2^{M'}))\kern 50pt$$
$$\kern 250pt
:= I_{P_1}^G((w_{M'}\sigma)((a,\ldots, y)+\lambda_2^{M'}))$$ for some Weyl group element $w'_L$ such that $w'_L.w_{M'}
\sigma \cong 
w_{M'}\sigma$.

Further, we have $y < - a_-$ since $y$ is negative, $x \geq a_-$ and $\sum_{k \in [x,y]} k < 0$. In this context, the above Lemma \ref{3.2} 
claims there exists 
$y' \leq y$ :
$$\pi_0 \hookrightarrow I_{P_1}^G((w_{M'}\sigma)((a,\ldots, y')+\lambda_3^{M'}))$$

And then the unique irreducible generic subquotient $\pi'_0$ of $I_{P_1\cap N'}^{N'}(\sigma_{\lambda_3^{M'}})$ is square-integrable, or 
equivalently $
\sigma_{\lambda_3^{M'}}$ is a residual point for $\mu^{N'}$ (The type is given by $\Sigma_{\sigma}^{N'}$).
Further, $\sigma_{(a,\ldots, y')+\lambda_3^{M'}}$ is a residual point for $\mu^{G}$ (type given by $\Sigma_{\sigma}$), corresponding to the generic 
discrete series 
$\pi_0$.

Then the \emph{set of Jumps} of the residual segment associated to $\pi_0$ contains the \emph{set of Jumps} of the residual segment associated to $\pi'_0$ and two more elements $a$ and $-y'$. But then $a>-y'> a_-$, and this contradicts the fact that $a$ and $a_-$ are two consecutive jumps. 

We have shown that $\pi'$ is necessarily tempered.

\subsubsection*{Step B}

Let $(\underline{n}_{\pi_0})$ be the residual segment canonically associated to a generic discrete series $\pi_0$.
Let us now denote $a_{i+1}$ the greatest integer smaller than $a_i$ in the \emph{set of Jumps} of $(\underline{n}_{\pi_0})$. Therefore, the 
half-integers, 
$a_i$ and $a_{i+1}$ satisfy the conditions of this lemma.

As the representation $\pi'$ is tempered, by Theorem \ref{renard}, there exists a standard parabolic subgroup $P_{\#}$ of $M'$ and a discrete series $\tau'$ such that $\pi' \hookrightarrow I_{P_{\#}}^{M'}(\tau')$. 

Again, as an irreducible generic discrete series representation of a non necessarily maximal Levi subgroup, using the result of Heiermann-Opdam (Proposition \ref{opdamh}), there exists an irreducible cuspidal representation $\sigma'$ and a standard parabolic $P_{1,\#}$ of $M_{\#}$ such that 
$\tau'$ embeds in $I_{P_{1,\#}}^{M_{\#}}(\sigma'((\frac{a-a_--1}{2},-\frac{a-a_--1}{2})+\bigoplus_j(\mathpzc{a}_j,-\mathpzc{a}_j)+(\underline{n}_{\pi_0''})))$, where $(\underline{n}_{\pi_0''})$ is a residual segment corresponding to an irreducible generic discrete series $\pi_0''$ and $(\frac{a-a_--1}{2},-\frac{a-a_--1}{2})$ along with $(\mathpzc{a}_j,-\mathpzc{a}_j)$'s are linear residual segments for (half)-integers $\mathpzc{a}_j$.

Clearly, the point $(\frac{a-a_--1}{2},\ldots,-\frac{a-a_--1}{2})+\bigoplus_j(\mathpzc{a}_j,-\mathpzc{a}_j)+(\underline{n}_{\pi_0''})$ is in $\overline{a_{M_1}^{M_{\#}*}}^{+}$.

Then 

\begin{equation} \label{stepB}
\pi' \hookrightarrow I_{P_{1,\#}U_{\#}}^{M'}(\sigma'((\frac{a-a_--1}{2},\ldots,-\frac{a-a_--1}{2})+\bigoplus_j(\mathpzc{a}_j,\ldots,-\mathpzc{a}_j)+(\underline{n}_{\pi_0''})))
\end{equation}

Since $P_{1,\#}U_{\#}$ is standard in $P'$ which is standard in $G$, there exists a standard parabolic subgroup $P_1'$ in $G$, such that, when inducing Equation \ref{stepB}, we obtain:

\begin{equation} \label{stepB2}
\pi_0 \hookrightarrow I_{P'}^G(\pi'_{(a,\ldots, a_-+1)}) \hookrightarrow I_{P'_1}^G(\sigma'_{(a,\ldots, a_-+1)+\bigoplus_j(\mathpzc{a}_j,\ldots,-\mathpzc{a}_j)+(\underline{n}_{\pi_0''})})
\end{equation} 

Let us denote $(a, \ldots, a_-+1)+\bigoplus_j(\mathpzc{a}_j,\ldots,-\mathpzc{a}_j)+(\underline{n}_{\pi_0''}):= \lambda'$. 

\medskip

Since $\pi_0$ also embeds as a subrepresentation in $I_{P_1}^G(\sigma_{(a,\ldots, a_-+1)+\lambda})$, by Theorem 2.9 in \cite{BZI} (see also 
\cite{renard} VI.5.4) there exists a Weyl group element $w$ in $W^G$ such that $w.M_1= M_1', w.\sigma= \sigma'$ and $w((a, a_-+1)+\lambda)=\lambda'$. 

Since $\Sigma_{\sigma}$ is irreducible and $M_1'$ is standard, we have by Point (3) in Corollary \ref{mmm} that $M_1'= M_1$, and we can take $w$ in $W(M_1)$.
Further since $P_1$ and $P_1'$ are standard parabolic subgroups of $G$, and $\Sigma_{\sigma}$ is irreducible, they are actually equal (see Remark \ref{rmkP1}). Now, by Point (2) in Corollary \ref{mmm} any element in $W(M_1)$ is either in $W_{\sigma}$ or decomposes in elementary symmetries in $W_{\sigma}$ and $s_{\beta_d}W_{\sigma}$ and :
$$\sigma'= w\sigma= \left\{
    \begin{array}{ll}
         \sigma ~ \text{if} ~ w \in W_{\sigma} \\
          s_{\beta_d}\sigma ~ \hbox{otherwise}
    \end{array}
\right.$$
Let us assume we are in the context where $\sigma'= s_{\beta_d}\sigma \ncong \sigma$. As explained in the first part of Section \ref{cond} (see Proposition \ref{H}), this happens if $\Sigma_{\sigma}$ is of type $D$. Let us apply the bijective operator (see Lemma \ref{lemmabeta}) from $I_{P_1 \cap (M_1)_{\beta_d}}^{(M_1)_{\beta_d}}(s_{\beta_d}\sigma)_{\lambda'})$ to $I_{\overline{P_1} \cap (M_1)_{\beta_d}}^{(M_1)_{\beta_d}}((s_{\beta_d}\sigma)_{\lambda'})$ and then the bijective map $t(s_{\beta_d})$ (the definition of the map $t(g)$ has been given in the proof of Proposition \ref{prop}) to $I_{s_{\beta_d}(\overline{P_1} \cap (M_1)_{\beta_d})}^{(M_1)_{\beta_d}}(\sigma_{s_{\beta_d}\lambda'}) = I_{P_1 \cap (M_1)_{\beta_d}}^{(M_1)_{\beta_d}}(\sigma_{s_{\beta_d}\lambda'})$.

As explained in Remark \ref{remD}, $s_{\beta_d}\lambda'= \lambda'$ since $\lambda'$ is a residual point of type $D$.
Therefore, we have a bijective map from $I_{P_1 \cap (M_1)_{\beta_d}}^{(M_1)_{\beta_d}}(s_{\beta_d}\sigma)_{\lambda'})$ to $I_{P_1 \cap (M_1)_{\beta_d}}^{(M_1)_{\beta_d}}(\sigma_{\lambda'})$.
The induction of this bijective map gives a bijective map from $I_{P_1}^G(\sigma'_{(a,\ldots, a_-+1)+\bigoplus_j(\mathpzc{a}_j,\ldots,-\mathpzc{a}_j)+(\underline{n}_{\pi_0''})})$ to $I_{P_1}^G(\sigma_{(a,\ldots, a_-+1)+\bigoplus_j(\mathpzc{a}_j,\ldots,-\mathpzc{a}_j)+(\underline{n}_{\pi_0''})})$. Hence we may write Equation (\ref{stepB2}) as:
\begin{equation}\label{stepB3}
\pi_0 \hookrightarrow I_{P'}^G(\pi'_{(a, \ldots, a_-+1)}) \hookrightarrow I_{P_1}^G(\sigma_{(a,\ldots, a_-+1)+\bigoplus_j(\mathpzc{a}_j,\ldots,-\mathpzc{a}_j)+(\underline{n}_{\pi_0''})})
\end{equation} 

Let us set $a = a_i$, $a_- = a_{i+1}$ for $a_i$, $a_{i+1}$ two consecutive elements in the set of Jumps of $(\underline{n}_{\pi_0})$. Therefore, $(a_i,\ldots, a_{i+1}+1)\bigoplus_j(\mathpzc{a}_j,\ldots,-\mathpzc{a}_j)+(\underline{n}_{\pi''_0})$ is in the Weyl group orbit of the residual segment associated to $\pi_0$: $(\underline{n}_{\pi_0})$.


Let us show that $(a_i,\ldots, a_{i+1}+1)(a_{i+1},\ldots, -a_{i+1})(\underline{n}^i)$ is in the $W_{\sigma}$-orbit of $(\underline{n}_{\pi_0})$.
\\
\noindent
One notices that in the tuple $\underline{n}_{\pi_0}$ of the residual segment $(\underline{n}_{\pi_0})$ the following relations are satisfied:

\begin{equation} \label{conditions1}
n_{a_i} = n_{a_{i+1}} - 1
\end{equation}
\begin{equation} \label{conditions2}
n_i = n_{i-1}-1 ~ \text{or} ~ n_i = n_{i-1}, ~ \forall i>0
\end{equation}

Consequently, when we withdraw $(a_i,\ldots, a_{i+1}+1)$ from this residual segment, we obtain a segment $(\underline{n'})$ which cannot be a 
residual 
segment since $n'_{a_{i+1}} =  n'_{a_{i+1}+1} + 2$ for $i\neq 1$; or if $i=1$, $n'_{a_2} = 2$ but $a_2$ is now the greatest element in the set of 
Jumps 
associated to the segment $(\underline{n'})$, so we should have $n'_{a_2} = 1$.

Therefore, to obtain a residual point (residual segment $(\underline{n}_{\pi''_0})$), we need to remove twice $a_{i+1}$. 

Then, for any $0 <j< a_{i+1}$, if we remove twice $j$, $n'_j = n_j - 2$ and, for all $i$, the relations $ n'_j = n'_{j-1}-1 ~ \text{or} ~ n'_j = n'_{j-1}$ 
are still 
satisfied. As we also remove one zero, we have for $j=0$, $n'_0 = n_0 -1$ which is compatible with removing twice $j=1$.

The residual segment left, thus obtained, will be denoted $(\underline{n}^i)$. We have shown that $(a_i,\ldots, a_{i+1}+1)(a_{i+1},\ldots, -a_{i+1})(\underline{n}^i)$ is in the $W_{\sigma}$-orbit of $(\underline{n}_{\pi_0})$.

Since $(\underline{n}^i)$ is a residual segment, from the conditions detailed in Equations \ref{conditions1} and \ref{conditions2} (see also Remark \ref{rmkrs} in Section \ref{JB}) no symmetrical linear residual segment $(\mathpzc{a}_k,-\mathpzc{a}_k)$ can be extracted from $(\underline{n}^i)$ to obtain another residual segment $(\underline{n}_{\pi''_0})$ such that $(a_i,\ldots, a_{i+1}+1)(a_{i+1},\ldots, -a_{i+1})(\mathpzc{a}_k,-\mathpzc{a}_k)(\underline{n}_{\pi''_0})$ is in the $W_{\sigma}$-orbit of $(\underline{n}_{\pi_0})$.

So $(\underline{n}_{\pi''_0})=(\underline{n}^i)$ and
$$\pi'_{(a, a_-+1)} \hookrightarrow I_{P_1}^{M'}(\sigma((a_i, a_{i+1}+1)+(a_{i+1}, -a_{i+1})+(\underline{n}^i)))$$

Eventually, using induction in stages Equation (\ref{stepB}) rewrites: 
$$\pi_0 \hookrightarrow I_{P_1}^G(\sigma((a_i, a_{i+1}+1)+(a_{i+1}, -a_{i+1})+(\underline{n}^i))= \Theta$$
and since the two segments $(a_i,\ldots, a_{i+1}+1)$ and $(a_{i+1},\ldots, -a_{i+1})$ are linked, we can take their union and deduce there exists an irreducible generic essentially square integrable representation $\pi_{a_i}$ of a Levi subgroup $M^{a_i}$ in $P^{a_i}$ which once induced embeds as a subrepresentation in $\Theta$ and therefore 
by multiplicity one of the irreducible generic piece ([U], see \ref{convrodier}) , $\pi_0$,  we have:
$$\pi_0 \hookrightarrow I_{P^{a_i}}^G(\pi_{a_i}) \hookrightarrow I_{P_1}^G(\sigma((a_i, -a_{i+1})+(\underline{n}^i)))$$


%

\begin{prop} \label{embedding}
Let $(\underline{n}_{\pi_0})$ be a residual segment associated to an irreducible generic discrete series $\pi_0$ of $G$ whose cuspidal support satisfies the conditions \emph{CS} (see the Definition \ref{CS}).

Let $a_1 > a_2> \ldots > a_n$ be \emph{Jumps} of this residual segment. Let $P_1= M_1U_1$ be a standard parabolic subgroup, $\sigma$ be a 
unitary 
irreducible cuspidal representation of $M_1$ such that $\pi_0 \hookrightarrow I_{P_1}^G(\sigma(\underline{n}_{\pi_0}))$.

For any $i$, there exists a standard parabolic subgroup $P^{a_i} \supset 
P_1$ with Levi subgroup $M^{a_i}$, residual segment $(\underline{n}^i)$ and an irreducible generic essentially square-integrable representation $\pi_{a_i}= Z^{M^{a_i}}(P_1,\sigma,(a_i, -a_{i+1})
(\underline{n}^i))$ such that 
$\pi_0$ embeds as a subrepresentation in $$I_{P^{a_i}}^G(\pi_{a_i}) \hookrightarrow I_{P_1}^G(
\sigma((a_i, -
a_{i+1})+(\underline{n}^i)))$$
\end{prop}

\begin{proof}
By the result of Heiermann-Opdam [Proposition \ref{opdamh}] and Lemma \ref{firstresultslem}, to any residual segment $(\underline{n}_{\pi_0})$ we associate the unique irreducible generic discrete series subquotient in $I_{P_1}^G(\sigma(\underline{n}_{\pi_0}))$.

Then as explained in the Subsection \ref{JB} this residual segment defines uniquely \emph{Jumps} : $a_1 > a_2> \ldots > a_n$.

Start with the two elements $a_1= \ell+m$ and $a_2=\ell-1$ and consider the following induced representation:

\begin{multline}
I_{P_1}^G(\sigma((\ell+m , a_2+1=\ell)(\ell-1)^{n_{\ell-1}}(\ell-2)^{n_{\ell-2}} \ldots 0^{n_0})) \\
= I_{P}^G(I^M_{P_1\cap M}(\sigma((\ell+m , a_2=\ell)(\ell-1)^{n_{\ell-1}}(\ell-2)^{n_{\ell-2}} \ldots 0^{n_0}))
\end{multline}
Let us denote $\nu:= (\ell+m , a_2+1=\ell)(\ell-1)^{n_{\ell-1}}(\ell-2)^{n_{\ell-2}} \ldots 0^{n_0}$.

The induced representation $I^M_{P_1\cap M}(\sigma((\ell+m , a_2+1=\ell)(\ell-1)^{n_{\ell-1}}(\ell-2)^{n_{\ell-2}} \ldots 0^{n_0})):= I^M_{P_1\cap M}(\sigma_{\nu})
$ is a generic 
induced module.

The form of $\nu$ implies $\sigma_{\nu}$ is not necessarily a 
residual point for $\mu^M$. Indeed, the first linear residual segment $(\ell+m , a_2+1=\ell)$ is certainly a residual segment (of type $A$), but the second not 
necessarily.

Let $\pi$ be the unique irreducible generic subquotient of $I^M_{P_1\cap M}(\sigma_{\nu})$ (which exists by Rodier's Theorem). We have:
$\pi \leq I^M_{P_1\cap M}(\sigma_{\nu})$
and $I_P^G(\pi) \leq I_{P}^G(I^M_{P_1\cap M}(\sigma_{\nu})):= I_{P_1}^G(\sigma_{\lambda})$.

Assume $I_P^G(\pi)$ has an irreducible generic subquotient $\pi_0'$ different from $\pi_0$, then $\pi_0'$ and $\pi_0$ would be two generic 
irreducible 
subquotients in $I_{P_1}^G(\sigma_{\lambda})$ contradicting Rodier's theorem.
Hence $\pi_0 \leq I_P^G(\pi)$.

Further, since $\pi_0$ embeds as a subrepresentation in  
$$I_{P}^G(I^M_{P_1\cap M}(\sigma((\ell+m , a_2+1=\ell)+(\ell-1)^{n_{\ell-1}}
(\ell-2)^{n_{\ell-2}} \ldots 
0^{n_0})):=I_{P_1}^G(\sigma_{\lambda})$$ it also has to embed as a subrepresentation in $I_P^G(\pi)$.

Therefore, applying Lemma \ref{moeglinlemmaextended}, we conclude there exists a residual segment $(\underline{n^1})$ an essentially 
square integrable 
representation $\pi_{a_1}$ such that $\pi_0$ embeds as a subrepresentation in 
$$I_{P^{a_1}}^G(\pi_{a_1}) \hookrightarrow I_{P_1}^G((\sigma((a_1, -a_2)+(\underline{n}^1)))$$

Let us consider now the elements $a_2=\ell-1$ and $a_3$.
As in the proof of Lemma \ref{zelevinsky}, since the linear residual segments $(a_1, \ell-1)$ and $(\ell-1)$ are unlinked, we apply a composite map from the 
induced 
representation $I_{P_1\cap M'}^{M'}(\sigma((a_1, \ell-1)+(\ell-1)+\ldots 0^{n_0}))$ to $I_{P_1\cap M'}^{M'}(\sigma((\ell-1)+(a_1, \ell-1))+ \ldots +0^{n_0}))$. We can interchange 
the two 
segments and as in the proof of Lemma \ref{zelevinsky}, applying this intertwining map and inducing to $G$ preserves the unique irreducible 
generic 
subrepresentation of $I_{P_1}^G(\sigma_{\lambda})$.

We repeat this argument with 
$$I_{P_1\cap M''}^{M''}(\sigma((\ell-1)+(a_1, \ell-2)+(\ell-2)+ \ldots 0^{n_0})) ~~ \hbox{and} \qquad I_{P_1\cap M''}^{M''}(\sigma(\ell-1)+(\ell-2)+(a_1, \ell-2)+ \ldots +0^{n_0}))$$
and further repeat it with all exponents until $a_3 + 1$.

Eventually, the unique irreducible subrepresentation $\pi_0$ appears as a subrepresentation in $I_{P_1}^G(\sigma((a_2, a_3+1)+(a_1, 
a_3+1)+ 
(\ell-2)^{n_{\ell-2}-2}\ldots (a_3+1)^{n_{a_3+1} -2}\ldots 1^{n_{1}} 0^{n_{0}})$.

\begin{multline*}
\pi_0 \hookrightarrow I_{P'^{a_2}}^G(I_{P_1\cap M'^{a_2}}^{M'^{a_2}}(\sigma(a_2, a_3+1)+(a_1, a_3+1)+(\ell-2)^{n_{\ell-2}-2}\ldots (a_3+1)^{n_{a_3+1} 
-2}\ldots 
1^{n_{1}} 0^{n_{0}})) \\
:= I_{P'^{a_2}}^G(I_{P_1\cap M'^{a_2}}^{M'^{a_2}}((w\sigma)_{w\nu})) ~~ \hbox{where} ~~ w \in W_{\sigma}$$
\end{multline*}
Let $\pi$ be the unique irreducible generic subquotient of $I^{M'^{a_2}}_{P_1\cap M'^{a_2}}(\sigma_{w\nu})$ (which exists by Rodier's Theorem). 
We have:
$\pi \leq I^{M'^{a_2}}_{P_1\cap M'^{a_2}}(\sigma_{w\nu})$
and 
$$I_{P'^{a_2}}^G(\pi) \leq I_{P'^{a_2}}^G(I^{M'^{a_2}}_{P_1\cap M'^{a_2}}(\sigma_{w\nu})):= I_{P_1}^G(\sigma_{w\lambda})$$
Assume $I_{P'^{a_2}}^G(\pi)$ has an irreducible generic subquotient $\pi_0'$ different from $\pi_0$, then $\pi_0'$ and $\pi_0$ would be two 
generic irreducible 
subquotients in $I_{P_1}^G((w\sigma)_{w\lambda})$ contradicting Rodier's theorem.
Hence $\pi_0 \leq I_{P'^{a_2}}^G(\pi)$. Further, since $\pi_0$ embeds as a subrepresentation in  
$$I_{P'^{a_2}}^G(I^{M'^{a_2}}_{P_1\cap M'^{a_2}}(\sigma((a_2, a_3+1)+(a_1, a_3+1)+
(\ell-2)^{n_{\ell-2}-2}\ldots 
(a_3+1)^{n_{a_3+1} -2}\ldots 1^{n_{1}} 0^{n_{0}}):=I_{P_1}^G(\sigma_{w\lambda})$$
 it also embeds as a subrepresentation in $I_{P'^{a_2}}
^G(\pi)$.

Hence applying Lemma \ref{moeglinlemmaextended}, we conclude there exists a residual segment $(\underline{n^2})$ and an essentially 
square-
integrable representation $\pi_{a_2}= Z^{M^2}(P_1\cap M^{a_2}, \sigma,(a_2, -a_3)(\underline{n}^2))$ such that $\pi_0$ embeds as a subrepresentation 
in $I_{P^{a_2}}^G(\pi_{a_2}) \hookrightarrow I_{P_1}^G(\sigma((a_2, -a_3)+(\underline{n}^2))$.


Similarly, for any two consecutive elements in the \emph{set of Jumps}, $a_i$ and $a_{i+1}$, the same argumentation (i.e first 
embedding $\pi_0$ as 
a subrepresentation in $I_{P'^{a_i}}^G(\pi)$ using intertwining operators, and conclude with Lemma \ref{moeglinlemmaextended}) yields the embedding:

$$\pi_0 \hookrightarrow I_{P^{a_i}}^G(\pi_{a_i}) \hookrightarrow I_{P^{a_i}}^G(I_{P_1\cap M^i}
^{P^{a_i}}(\sigma((a_i, -a_{i+1})+(\underline{n}^i)))$$
for an irreducible generic essentially square-integrable representation $$\pi_{a_i}= Z^{M^{a_i}}(P_1\cap M^{a_i}, \sigma,(a_i, -a_{i+1})(\underline{n}^i))~~\hbox{of the Levi subgroup} ~ M^{a_i}.$$
\end{proof}

\subsection{Proof of the Theorem \ref{conditionsonlambda}} \label{conditions}

\begin{itemize}
\item (1)a) is the result of Lemma \ref{firstresultslem}.
\item (1)b) is the result of Proposition \ref{embedding}.
\item (1)c) Let us denote $\pi_0$ the unique irreducible generic subquotient in $I_{P_1}^G(\sigma_{(a,b)\underline{n}})$. By Proposition \ref{opdamh}, there exists a parabolic subgroup $P'$ such that $\pi_0$ embeds as a subrepresentation in the induced module $I_{P'}^G(\sigma'_{\lambda'})$, for $\sigma'_{\lambda'}$ a dominant residual point for $P'$.
Let $(w\sigma)_{w\lambda}$ be the dominant (for $P_1$) residual point in the $W_{\sigma}$-orbit of $\sigma_{\lambda}$, then (using Theorem 2.9 in \cite{BZI} or Theorem VI.5.4 in \cite{renard}) $\pi_0$ is the unique irreducible generic subquotient in $I_{P_1}^G((w\sigma)_{w\lambda})$, and  Proposition \ref{prop} gives us that these two ($I_{P'}^G(\sigma'_{\lambda'})$ and $I_{P_1}^G((w\sigma)_{w\lambda})$) are isomorphic.

The point $(w\sigma)_{w\lambda}$ is a dominant residual point with respect to $P_1$ : $w\lambda \in \overline{a_{M_1}^*}^+$ and there is a unique element in the orbit of the Weyl group $W_{\sigma}$ of a residual point which is dominant and is explicitly given by a residual segment using the correspondence of the Subsection 2.5.1. 
We denote  $w\lambda:= (\underline{n}_{\pi_0})$ this residual segment. Since $w \in W_{\sigma}$, $(w\sigma)_{w\lambda} \cong \sigma_{w\lambda}$. Hence, $\pi_0 \hookrightarrow I_{P_1}^G(\sigma(\underline{n}_{\pi_0}))$.

Since $\mathpzc{a} > \mathpzc{b}$, and $(\underline{n_{\pi_0}})$ is a residual segment, it is clear that $\mathpzc{a}$ is a jump. [Indeed, if you extract a linear residual 
segment $(\mathpzc{a},\ldots,\mathpzc{b})
$ such that $\mathpzc{a} > \mathpzc{b}$ from  $(\underline{n}_{\pi_0})$  such that what remains is a residual segment, then $\mathpzc{a}=a$ has to be in the \emph{set of 
Jumps} of the 
residual segment $(\underline{n}_{\pi_0})$ as defined in the Subsection \ref{SOJ}].
Let us denote $a_-$ the greatest integer smaller than $a$ in the \emph{set of Jumps}. Therefore, the (half)-integers, $a$ and $a_-$ satisfy the conditions of Proposition \ref{embedding}.
We will show below that $ \mathpzc{b} \geq -a_-$.
Let $P_{\flat}= P_{\Delta - \left\{\alpha_{a+a_-+1}\right\}}$ be a maximal parabolic subgroup, with Levi subgroup $M_{\flat}$, which contains $P_1$.

Let $\pi_{a} = Z^{M_{\flat}}(P_1,\sigma, w_{a_-}\lambda)$, for $w_{a_-} \in W_{\sigma}$ be the generic essentially square integrable representation with cuspidal support $(\sigma((a, -a_-)(\underline{n_{-a_-}}))$ associated to the residual segment $((a, -a_-)+(\underline{n_{-a_-}}))$ (in the $W_{\sigma}$-orbit of $(\underline{n}_{\pi_0})$). It is some discrete series twisted by the Langlands parameter $s_{-a_-}\widetilde{\alpha_{a+a_-+1}}$ with $s_{-a_-}=\frac{a - a_-}{2}$. By the Proposition \ref{embedding} we can write 
\begin{equation}\label{Cor1}
\pi_0 \hookrightarrow I_{P_{\flat}}^G(\pi_{a}) \hookrightarrow I_{P_1}^G(\sigma((a, -a_-)
(\underline{n_{-a_-}})))
\end{equation}

Here, we need to justify that given $a$, for any $\mathpzc{b}$ we have: $\mathpzc{b} \geq -a_-$. 

Consider again the residual segment $(\underline{n}_{\pi_0})$, and observe that by definition the sequence $(a, \ldots, -a_-)$ is the longest linear segment with greatest (half)-integer $a$ that one can withdraw from $(\underline{n}_{\pi_0})$ such that the remaining segment $(\underline{n_{-a_-}})$ is a residual 
segment of the same type and $(a, \ldots, -a_-)(\underline{n_{-a_-}})$ is in the Weyl group orbit of $(\underline{n}_{\pi_0})$.

Further, this is true for any couple $(a, a_-)$ of elements in the \emph{set of Jumps} associated to the residual segment $(\underline{n}
_{\pi_0})$. It is 
therefore clear that given $a$ and $a_-$ such that $s_{-a_-}=\frac{a - a_-}{2} > 0$ is the smallest positive (half)-integers as possible, we have 
$s_{\mathpzc{b}}= \frac{a
+b}{2} \geq s_{-a_-}=\frac{a - a_-}{2}$ and $\mathpzc{b} $ is necessarily greater or equal to $-a_-$.

Once this embedding is given, using Lemma \ref{nngenerickk}, there exists an intertwining operator with non-generic kernel from the induced module $I_{P_1}^G(\sigma((a, -a_-)(\underline{n_{-a_-}})))$ given in Equation (\ref{Cor1}) to any other induced module from the cuspidal support $\sigma(a, \mathpzc{b},\underline{n_{\mathpzc{b}}})$ with $\mathpzc{b} \geq -a_-$. 

Therefore, $\pi_0 \hookrightarrow I_{P_1}^G(\sigma(a, \mathpzc{b},\underline{n_{\mathpzc{b}}}))= I_{P_1}^G(\sigma_{\lambda}).$

\item (2)a) Since $\lambda$ is not a residual point, the generic subquotient is non-discrete series.
By Langlands' classification, Theorem \ref{LC}, and the Standard module conjecture, it has the form $J_{P'}^G(\tau'_{\nu'}) \cong I_{P'}^G(\tau'_{\nu'})$.
By Theorem \ref{minLP=irr}, $\nu'$ corresponds to the minimal Langlands parameter (this notion was introduced in the Theorem \ref{LC}) for a given cuspidal support.

For an explicit description of the parameter $\nu$, given the cuspidal string $(\mathpzc{a},\mathpzc{b}, \underline{n})$, the reader is encouraged to read the analysis conducted in the Appendix of the author's thesis manuscript \cite{these}.

The representation $\tau'$ (e.g. $St_q|.|^{\nu'}\otimes\pi'$ in the context of classical groups, for a given integer $q$) corresponds to a cuspidal string $(\mathpzc{a}', \mathpzc{b}', \underline{n'})$, and cuspidal representation $\sigma'$, that is:
$$I_{P'}^G(\tau'_{\nu'}) \hookrightarrow I_{P'_1}^G(\sigma'(\mathpzc{a}', \mathpzc{b}', \underline{n'}))$$

By the Theorem 2.9 in \cite{BZI}, we know the cuspidal data $(P_1,\sigma, (\mathpzc{a'},\mathpzc{b'}, \underline{n'}))$ and $(P'_1,\sigma', \lambda':=(\mathpzc{a}',\mathpzc{b}', \underline{n'}))$ are conjugated by an element $w \in W^G$. 

By Corollary \ref{mmm} and since $P_1$ and $P'_1$ are standard parabolic subgroups (see Remark \ref{rmkP1}), we have $P_1 = P'_1$, $w \in W(M_1)$. Any element in $W(M_1)$ decomposes in elementary symmetries with elements in $W_{\sigma}$ and $s_{\beta_d}W_{\sigma}$:

$$\sigma'= w\sigma= \left\{
    \begin{array}{ll}
         \sigma ~ \text{if} ~ w \in W_{\sigma} \\
         s_{\beta_d}\sigma ~ \hbox{otherwise}
    \end{array}
\right.$$

Let us assume we are in the context where $\sigma'= s_{\beta_d}\sigma \ncong \sigma$.
As explained in the first part of Section \ref{Moeglinandembedding}, this happens if $\Sigma_{\sigma}$ is of type $D$.

Let us apply the bijective operator (see Lemma \ref{lemmabeta}) from $I_{P_1 \cap (M_1)_{\beta_d}}^{(M_1)_{\beta_d}}(s_{\beta_d}\sigma)_{\lambda'})$ to $I_{\overline{P_1} \cap (M_1)_{\beta_d}}^{(M_1)_{\beta_d}}((s_{\beta_d}\sigma)_{\lambda'})$ and then the bijective map (the definition of the map $t(g)$ has been given in the proof of \ref{prop}) $t(s_{\beta_d})$ to $I_{s_{\beta_d}(\overline{P_1} \cap (M_1)_{\beta_d})}^{(M_1)_{\beta_d}}(\sigma_{s_{\beta_d}\lambda'}) = I_{P_1 \cap (M_1)_{\beta_d}}^{(M_1)_{\beta_d}}(\sigma_{s_{\beta_d}\lambda'})$.
As explained in Remark \ref{remD}, $s_{\beta_d}\lambda'= \lambda'$ since $\lambda'$ is a residual point of type $D$.
Therefore, we have a bijective map from $I_{P_1 \cap (M_1)_{\beta_d}}^{(M_1)_{\beta_d}}(s_{\beta_d}\sigma)_{\lambda'})$ to $I_{P_1 \cap (M_1)_{\beta_d}}^{(M_1)_{\beta_d}}(\sigma_{\lambda'})$.
The induction of this bijective map gives a bijective map from $I_{P'_1}^G(\sigma'(\mathpzc{a}', \mathpzc{b}', \underline{n'}))$ to $I_{P'_1}^G(\sigma(\mathpzc{a}', \mathpzc{b}', \underline{n'}))$.

\item (2)b)
Assume now that we consider a tempered or non-tempered subquotient in $I_{P_1}^G(\sigma(\mathpzc{a},\mathpzc{b}, \underline{n}))$. We first apply the argumentation developed in the previous point (2)a) to embed it in $I_{P'_1}^G(\sigma(\mathpzc{a}', \mathpzc{b}', \underline{n'}))$. Then it is enough to understand how one passes from the cuspidal string $(\mathpzc{a}', \mathpzc{b}', \underline{n'})$ to $(\mathpzc{a},\mathpzc{b}, \underline{n})$ to understand the strategy for embedding the unique irreducible generic subquotient as a subrepresentation $I_{P_1}^G(\sigma(\mathpzc{a},\mathpzc{b}, \underline{n}))$.

Starting from $(\mathpzc{a}, \mathpzc{b}, \underline{n})$, to minimize the  Langlands parameter $\nu'$, we usually remove elements at the end of the first segment 
(i.e. the 
segment $(\mathpzc{a},\ldots,\mathpzc{b}))$ to insert them on the second residual segment, or we enlarge the first segment on the right. This means either $\mathpzc{a}' < \mathpzc{a}$, or 
$\mathpzc{b}' < \mathpzc{b}$, or 
both.

If $\mathpzc{a}'=\mathpzc{a}$, and $\mathpzc{b}' < \mathpzc{b}$, in particular if $\mathpzc{b}' < 0$, we have a non-generic kernel operator between $I_{P_1}^G(\sigma(\mathpzc{a}', \mathpzc{b}', \underline{n'}))$ 
and $I_{P_1}
^G(\sigma(\mathpzc{a}, \mathpzc{b}, \underline{n}))$ as proved in Lemma \ref{nngenerickk}.
\end{itemize}

\subsection{An order on the cuspidal strings in a $W_{\sigma}$-orbit}

It is possible to describe the set of points in the $W_{\sigma}$-orbit of a dominant residual point $\lambda_D$ as follows. 

Let us define a set of points $\mathcal{L}$ in the $W_{\sigma}$-orbit of a dominant residual point $\lambda_D$ such that they are written as : $(\mathpzc{a},\mathpzc{b})(\underline{n})$ with at most one linear residual segment $(\mathpzc{a},\mathpzc{b})$ satisfying the condition $\mathpzc{a} > \mathpzc{b}$. Then $\mathpzc{a}$ is a Jump as explained in the proof of Theorem \ref{conditionsonlambda}, point 1)c).

Let us attach a positive integer $C(1,\lambda)= \#\left\{\beta \in \Sigma_{\sigma}^+ \vert \prodscal{\lambda}{\check{\beta}} < 0 ~\right\}$ to any of these points.  \\
By definition, $C(1,\lambda_D) = 0$. What are the points $\lambda$ in $\mathcal{L}$ such that the function $C(1,\lambda)$ is maximal? \\

\begin{lemma}
The function $C(1,.)$ on $\mathcal{L}$ is maximal for the points which are the form $(a,-a_-)(\underline{n})$ for $(a,a_-)$ any two consecutive elements in the Jumps sets associated to $\lambda_D$.
\end{lemma}

\begin{proof}
Let us choose a point in $\mathcal{L}$; since it is a point in $\mathcal{L}$, it uniquely determines a jump $a$ (as its left end).
For any fixed $a$, we show that the function $C(1,\lambda_a)$ is maximal for $\lambda_{a,-a_-} =(a,-a_-)(\underline{n})$. 
Let $\mathcal{L}_a$ denote the set of points in $\mathcal{L}$ such that the linear residual segment (if it exists) \emph{has left end} $a$. The union of the $\mathcal{L}_a$ where $a$ runs over the set of Jumps is $\mathcal{L}$.

Let us choose a point $\lambda_{a,\mathpzc{b}} = (a,\mathpzc{b})(\underline{n})_{\mathpzc{b}}$ in $\mathcal{L}_a$ and denote $L_{\mathpzc{b}}$ the length of the residual segment $(\underline{n})_{\mathpzc{b}}$. Recall also that $(\underline{n})_{\mathpzc{b}} = (\ell, \ldots \mathpzc{b}^{n_{\mathpzc{b}}} \ldots 0^{n_0})$

\begin{itemize}
\item Case $a > 0 > \mathpzc{b}$ 

Consider $\lambda_{\mathpzc{b}}$ and $\lambda_{\mathpzc{b}+1}$.

Let us consider first those roots which are of the forms $e_i - e_j, \, i>j$: On $\lambda_{\mathpzc{b}}$ the number of these roots which have non-positive scalar product is: $(-b)\times L_{\mathpzc{b}} + (L_{\mathpzc{b}} -n_0) + (L_{\mathpzc{b}} -(n_0+n_1)) + (L_{\mathpzc{b}} -(n_0+n_1+n_2+ \ldots + n_{\mathpzc{b}})) + C_{\mathpzc{b}+1}$ where $C_{\mathpzc{b}+1}$ is some constant depending on the multiplicities $n_i$ for $i\geq (\mathpzc{b}+1)$.

Secondly, let us consider the roots of the forms $e_i + e_j, \, i>j$; on  $\lambda_{\mathpzc{b}}$ the number of these roots which have non-positive scalar product is: $$L_{\mathpzc{b}} -(n_{\mathpzc{b}}+n_{b+1}+ n_{\ell}) + L_{\mathpzc{b}} -(n_{b-1} + n_{\mathpzc{b}} + n_{b+1}+ \ldots + n_{\ell}) + L_{\mathpzc{b}} -(n_{b-2}+ n_{b-1} + n_{\mathpzc{b}} + n_{b+1} + \ldots + n_{\ell}) + \ldots + \mathpzc{b} + \mathpzc{b}-1 + \mathpzc{b}-2 + \ldots  + 1$$

Finally, one should also take into account the roots of type $e_i, 2e_i$ or $e_i + e_d$ if $d$ is the dimension of $\Sigma_{\sigma}$ and of type $B$, $C$ or $D$. There are $b$ such roots in our context.

\begin{multline*}
C(1, \lambda_{\mathpzc{b}+1}) = (-\mathpzc{b}-1)\times (L_{\mathpzc{b}}+1) + (L_{\mathpzc{b}}+1-n_0) + (L_{\mathpzc{b}}+1 -(n_0+n_1)) + \ldots + (L_{\mathpzc{b}}+1-(n_0+n_1+\ldots +n_{\mathpzc{b}}))+ C_{b+1}\\
 + L_{\mathpzc{b}} +1 -(n_{\mathpzc{b}-1} + n_{\mathpzc{b}}+n_{\mathpzc{b}+1}+ \ldots + n_{\ell}) + L_{\mathpzc{b}} -(n_{b-2}+ n_{b-1} + n_{\mathpzc{b}} + n_{\mathpzc{b}+1} + \ldots + n_{\ell}) + \ldots + \mathpzc{b}-1 + \mathpzc{b}-2 + \ldots  + 1  + \mathpzc{b}-1
\end{multline*}
$$C(1, \lambda_{\mathpzc{b}}) - C(1, \lambda_{\mathpzc{b}+1}) = L_{\mathpzc{b}} -(n_{\mathpzc{b}}+n_{b+1}+ n_{\ell}) + \mathpzc{b} + \mathpzc{b} -(-L_{\mathpzc{b}} - \mathpzc{b} - 1+ \mathpzc{b}-1 + \mathpzc{b}+ \mathpzc{b}-1)$$
$$C(1, \lambda_{\mathpzc{b}}) - C(1, \lambda_{\mathpzc{b}+1}) = 2L_{\mathpzc{b}} -(n_{\mathpzc{b}}+n_{\mathpzc{b}+1}+ n_{\ell}) +3$$
Therefore $$C(1, \lambda_{b})> C(1, \lambda_{\mathpzc{b}+1})$$


\item  Case $a > \mathpzc{b} > 0$

Consider $\lambda_{\mathpzc{b}}$ and $\lambda_{\mathpzc{b}+1}$.
The number $C(1, \lambda_{\mathpzc{b}})$ and $C(1, \lambda_{\mathpzc{b}+1})$ differ by $L_{\mathpzc{b}} - (n_0 + n_1 + \ldots +n_{\mathpzc{b}})$. As this number is clearly positive, we have: $C(1, \lambda_{\mathpzc{b}})>C(1, \lambda_{\mathpzc{b}+1})$.
\end{itemize}

This shows that $C(1, .)$ decreases as the length of the linear residual segment $(a,b)$ decreases. Furthermore, from the definition of residual segment (Definition \ref{rs}) and the observations made on cuspidal lines, the sequence $(a, \ldots, -a_-)$ is the \emph{longest} linear segment with greatest (half)-integer $a$ that one can withdraw from $\lambda_D$ such that the remaining segment $(\underline{n_{-a_-}})$ 
is a residual segment of the same type and $(a, \ldots, -a_-)(\underline{n_{-a_-}})$ is in the $W_{\sigma}$-orbit of $\lambda_D$. Therefore, $C(1, \lambda_{a,-a_-})$ is maximal on the set $\mathcal{L}_a$. 
\end{proof}

As a consequence of this Lemma, we will denote the points of maximal $C(1, .)$, $\lambda_{a_i}$ for any $a_i$ in the jumps set of $\lambda_D$. \\

The elementary symmetries associated to roots in $\Sigma_{\sigma}$ permute the (half)-integers appearing in the cuspidal line $(\mathpzc{a},\mathpzc{b})(\underline{n})$. 

We illustrate the set $\mathcal{L}$ with a picture: Let us assume any two points in the $W_{\sigma}$-orbit are connected by an edge if they share the same parameter $a$ and/or the intertwining operator associated to the sequence of elementary symmetries connecting the two points has non-generic kernel. Any point in $\mathcal{L}$ is on an edge joining the points of maximal $C(1,.)$ to $\lambda_D$. We obtain the following picture.

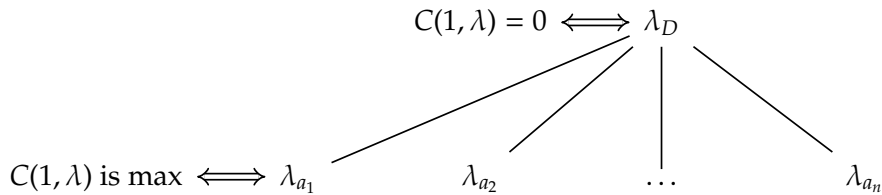
\begin{figure}[ht]
\centering
\begin{tikzcd}
 &  & {C(1,\lambda)=0} & \lambda_D \arrow[l, Leftrightarrow] \arrow[lldd, no head] \arrow[ldd, no head] \arrow[rrdd, no head] \arrow[dd, no head]  &  &  \\
 &  &  &  &  &  \\
{C(1,\lambda) ~\hbox{is max}} & \lambda_{a_1} \arrow[l, Leftrightarrow] & \lambda_{a_2} & \ldots &  & \lambda_{a_n}
\end{tikzcd}
\caption{The set $\mathcal{L}$} \label{fig:M1}
\end{figure}

Then the proof of the Theorem \ref{conditionsonlambda} could be thought about in this way: Relying on the extended Moeglin's Lemmas we obtain the embedding of the unique irreducible generic subquotient for a set of parameters $\{\lambda_{a_i}\}_i$. Those parameters are indexed by the jumps $a_i$ in a (finite) \emph{set of Jumps} associated to the dominant residual point $\lambda_D$ (they are in the $W_{\sigma}$-orbit of $\lambda_D$). Once this key embedding is given, for each jump $a$, we use intertwining operators with non-generic kernel to send the unique irreducible generic subrepresentation which lies in $I_{P_1}^G(\sigma_{\lambda_a}) = I_{P_1}^G(\sigma((a, -a_-)(\underline{n}))$ to $I_{P_1}^G(\sigma((a, \mathpzc{b})(\underline{n'}))$, for any $\mathpzc{b}> -a_-$ where $(\underline{n'})$ is a residual segment of the same type as $(\underline{n})$.

\section[Proof of the Conjecture for Discrete Series Subquotients when $\Sigma_{\sigma}$ is irreducible]{Proof of the Generalized Injectivity Conjecture for Discrete Series Subquotients} 
\label{gicds}

Before entering the proof of the Conjecture for Discrete Series Subquotients, let us mention two aside results. First, in order to use Theorem \ref{heir}, let us first prove the following lemma:

\begin{lemma}\label{lemmemu}
Under the assumption that $\mu^G$ has a pole at $s\tilde{\alpha}$ (assumption 1) for $\tau$ and $\mu^M$ has a pole at $\nu$ (for $\sigma$) of maximal order, for $\nu \in a_{M_1}^*$, $\sigma_{\nu + s\tilde{\alpha}}$ is a residual point.
\end{lemma}

\begin{proof}
We will use the multiplicativity formula for the $\mu$ function (see Section IV 3 in \cite{wald}, or the earlier result (Theorem 1) in \cite{silb}) :
$$\mu^G(\tau_{s\tilde{\alpha}}) = \frac{\mu^G}{\mu^{M}}(\sigma_{s\tilde{\alpha} + \nu})$$
We first notice that if $\mu^M$ has a pole in $\nu$ (for $\sigma$) of maximal order, for $\nu \in a_{M_1}^*$, $\mu^M$ also has a pole of maximal order at $\nu + s\tilde{\alpha}$ (since $s\tilde{\alpha}$ is in $a_M^*$, we twist by a character of $A_M$ which leaves the function $\mu^M$ unchanged).
Under the assumption 1, the order of the pole at $\nu + s\tilde{\alpha}$ of the right side of the equation is:
$$ \mbox{ord(pole for}  \mu^G \ \mbox{at} \ \nu + s\tilde{\alpha}) - (rk_{ss}(M) - rk_{ss}(M_1)) \geq 1$$ 
Since $M$ is maximal we have: $(rk_{ss}(G) - rk_{ss}(M)) = \dim(A_M) - \dim(A_G) = 1$, then 
\begin{flushleft}
$(rk_{ss}(M) - rk_{ss}(M_1)) + 1 = (rk_{ss}(M) - rk_{ss}(M_1)) + (rk_{ss}(G) - rk_{ss}(M))=(rk_{ss}(G) - rk_{ss}(M_1))$
\end{flushleft}
Hence, $\mbox{ord(pole of} \ \mu^G \ \mbox{in} \ \nu + s\tilde{\alpha}) \geq (rk_{ss}(G) - rk_{ss}(M_1))$, and the lemma follows.
\end{proof}

The element $\nu + s\tilde{\alpha}$ being a residual point (a pole of maximal order for $\mu^G$) for $\sigma$, by Theorem \ref{heir} we have 
a discrete 
series subquotient in $I_{P_1}^G(\sigma_{\nu+s\tilde{\alpha}})$. Further, consider the following classical lemma (see for instance \cite{zhangLreducibilities}):

\begin{lemma} \label{lem}
Take $\tau$ a tempered representation of $M$, and $\nu_0$ in the positive Weyl chamber. If $\nu_0$ is a pole for $\mu^G$ then 
$I_P^G(\tau_{\nu_0})$ is 
reducible.
\end{lemma}
This lemma results from the fact that when $\tau$ is tempered and $\nu_0$ in the positive Weyl chamber, $J_{\overline{P}|P}(\tau,.)$ is 
holomorphic at $
\nu_0$. If the $\mu$ function has a pole at $\nu_0$ then $J_{\overline{P}|P}J_{P|\overline{P}}(\tau,.)$ is the zero operator at $\nu_0$.
The image of $J_{P|\overline{P}}(\tau,.)$ would then be in the kernel of $J_{\overline{P}|P}(\tau,.)$, a subspace of $I_P^G(\tau_{\nu_0})$ which is null if $I_P^G(\tau_{\nu_0})$ is irreducible. This would imply $J_{P|\overline{P}}$ is a zero operator which is not possible. So $I_P^G(\tau_{\nu_0})$ must be reducible.

Under the hypothesis of Lemma \ref{lemmemu}, the module $I_P^G(\tau_{s\tilde{\alpha}})$ has a generic discrete series subquotient. We aim to prove in this section that this generic subquotient is a subrepresentation.

We present here the proof of the generalized injectivity conjecture in the case of a standard module induced from a maximal parabolic $P=MU$. Then, the roots in Lie($M$) are all the roots in $\Delta$ but $\alpha$. We first present the proof in case $\alpha$ is not an extremal root in the Dynkin diagram of $G$, and secondly when it is an extremal root.

\begin{prop} \label{dsr=1}
Let $\pi_0$ be an irreducible generic representation of a quasi-split reductive group $G$ of type $A,B,C$ or $D$ which embeds as a subquotient in the standard module 
$I_P^G(\tau_{s
\tilde{\alpha}})$, with $P=MU$ a maximal parabolic subgroup and $\tau$ discrete series of $M$. 

Let $\sigma_{\nu}$ be in the cuspidal support of the generic discrete series representation $\tau$ of the maximal Levi subgroup $M$ and we take $s
\tilde{\alpha}$ in 
$\overline{(a_M^*)^+}$, such that $I_P^G(\tau_{s\tilde{\alpha}}) \hookrightarrow I_{P_1}^G(\sigma_{\nu+s\tilde{\alpha}})$ and denote $\lambda= \nu+s\tilde{\alpha}$ in $\overline{a_{M_1}^M}^{+*}$. 

Let us assume that the cuspidal support of $\tau$ satisfies the conditions \emph{CS} (see the Definition \ref{CS}).

Let us assume that $\alpha$ is not an extremal simple root on the Dynkin diagram of $\Sigma$.

Let us assume $\sigma_{\lambda}$ is a residual point for $\mu^G$. This is equivalent to say that the induced representation $I_{P_1}^G(\sigma_{\lambda})$ 
has a discrete series subquotient.
Then, $\pi_0$, which is discrete series embeds as a submodule in $I_{P_1}^G(\sigma_{\lambda})$ and therefore in the standard module $I_P^G(\tau_{s\tilde{\alpha}}) \hookrightarrow I_{P_1}^G(\sigma_{\lambda})$.
\end{prop}

\begin{proof}
First, notice that if $s=0$, the induced module $I_P^G(\tau_{s\tilde{\alpha}})$ is unitary hence any irreducible subquotient is a subrepresentation; in the rest of the proof we can therefore assume $s\tilde{\alpha}$ in $(a_M^*)^+$.

We are in the context of the Subsection \ref{setting2}, and therefore we can write $\lambda:= (\mathpzc{a},\ldots, \mathpzc{b})(\underline{n})$, for some (half)-integers $\mathpzc{a} > \mathpzc{b}$, 
and residual segment $(\underline{n})$. In this context, as we denote ${s\tilde{\alpha}}$ the Langlands parameter twisting the discrete series $\tau$, then $s= s_{\mathpzc{b}}=\frac{\mathpzc{a} + \mathpzc{b}}{2}$.

Notice that since $\sigma_{\lambda}$ is in the $W_{\sigma}$-orbit of a dominant residual point whose parameter corresponds to a residual segment of type $B,C$ or $D$, $\mathpzc{a}$ and $\mathpzc{b}$ are not only reals but \emph{(half)-integers}. 
The conditions of application of Theorem \ref{conditionsonlambda} 1)b) or 1)c) are satisfied and therefore the unique irreducible generic subquotient in $I_{P_1}^G(\sigma_{\lambda})$ is a subrepresentation. By multiplicity one, it will also embed as a subrepresentation in the standard module $I_{P}^G(Z^M(P_1,\sigma, \lambda))$.
\end{proof}

\begin{rmk}
From the Theorem \ref{conditionsonlambda} and the argumentation given in the proof of the previous Proposition, it is easy to deduce that if $\pi_0$ appears as a submodule in the standard module $$I_{P_{\flat}}^G(Z^{M_{\flat}}(P_1,\sigma, w_{a_-}\lambda))$$ with Langlands parameter $s_{a_-}\widetilde{\alpha_{a+a_-+1}}$, it also appears as a submodule in any standard module $I_{P}^G(Z^M(P_1,\sigma,(a,\mathpzc{b},\underline{n_{\mathpzc{b}}}))$ with Langlands' parameter $s_{\mathpzc{b}}\tilde{\alpha} \geq s_{-a_-}\widetilde{\alpha_{a+a_-+1}}$ for the order defined in Lemma \ref{bw} as soon as $Z^M(P_1,\sigma, (a,\mathpzc{b},\underline{n_{\mathpzc{b}}}))$ has equivalent cuspidal 
support.
\end{rmk}

\subsubsection{The case of $\Sigma_{\sigma}^M$ irreducible} \label{SsM}

\begin{prop} \label{extremalroots}
Let $\pi_0$ be an irreducible generic discrete series of $G$ with cuspidal support $(M_1, \sigma)$ and let us assume $\Sigma_{\sigma}$ is irreducible.
Let $M$ be a standard maximal Levi subgroup such that $\Sigma_{\sigma}^M$ is irreducible.

Then, $\pi_0$ embeds as a subrepresentation in the standard module $I_{P}^G(\tau_{s\tilde{\alpha}})$, where $\tau$ is an irreducible generic discrete series of $M$.
\end{prop}

\begin{proof}
Assume $\Sigma_{\sigma}$ is irreducible of rank $d$, let $\Delta_{\sigma}:= \left\{ \alpha_1, \ldots, \alpha_d \right\}$ be the basis of $\Sigma_{\sigma}$ 
(following our choice of basis for the root system of $G$) and let us denote $\mathcal{T}$ its type.

We consider maximal standard Levi subgroups of $G$, $M \supset M_1$, such that the root system $\Sigma_{\sigma}^M$ is irreducible. Typically $M= M_{\Delta - \left\{\underline{\beta_d}\right\}}$.


Now, in our setting, $\sigma_{\nu}$ is a residual point for $\mu^M$. It is in the cuspidal support of the generic discrete series $\tau$ if and only 
if (applying Proposition \ref{1.13bis}):
$rk(\Sigma _{\sigma }^M) = d-1$. Let us denote $(\nu_2,\dots ,\nu_d)$ the residual segment corresponding to the irreducible generic discrete series $\tau$ of $M$. 

If $(\nu _2,\dots ,\nu _d)$ is a residual segment of type $A$ to obtain a residual segment $(\nu_1,\nu _2,\dots ,\nu _d)$ of rank $d$ and type:
\begin{itemize}
	\item $D$: we need $\nu_d= 0$ and  $\nu_1 = \nu _2+1$
	\item $B$: we need $\nu_d= 1$ and  $\nu_1 = \nu _2+1$
	\item $C$: we need $\nu_d =1/2$ and  $\nu_1 = \nu _2+1$
\end{itemize}

If $(\nu_2,\dots ,\nu_d)$ is a residual segment of type $\mathcal{T}$ ($B$, $C$, $D$) 
we need $\nu_1 = \nu _2+1$ to obtain a residual segment of type $\mathcal{T}$ and rank $d$. 

In all these cases, the twist $s\tilde{\alpha}$ corresponds on the cuspidal support to add one element on the left to the residual segment $
(\nu_2,\dots ,
\nu_d)$; then the segment $(\nu_1, \nu _2,\dots ,\nu _d):=(\lambda_1, \lambda_2,\dots ,\lambda_d)$ is a residual segment:

$$\pi_0 \leq  I_{P}^G(\tau_{s\tilde{\alpha}}) \hookrightarrow I_{P_1}^G(\sigma_{\lambda})$$
This is equivalent to say $\sigma_{\lambda}$ is a \emph{dominant} residual point and therefore, by Lemma \ref{firstresultslem}, $\pi_0$ embeds as a subrepresentation in $I_{P_1}^G(\sigma_{\lambda})$ and therefore in $I_{P}^G(\tau_{s\tilde{\alpha}})$ by  [U] in the standard module.
\end{proof} 

\subsection{Non necessarily maximal parabolic subgroups} \label{nonmaximal}

In the course of the main theorem in this section, we will need the following result:

\begin{lemma}\label{merge}
Let $\mathcal{S}_1, \mathcal{S}_2, \ldots, \mathcal{S}_t$ be $t$ unlinked linear segments with $\mathcal{S}_i= (\mathpzc{a}_i, \ldots, \mathpzc{b}_i)$ for any $i$. If
 $$(\mathpzc{a}_1,\ldots, \mathpzc{b}_1)(\mathpzc{a}_2,\ldots, \mathpzc{b}_2)\ldots (\mathpzc{a}_t,\ldots, \mathpzc{b}_t)( \underline{n})$$ is a residual segment $(\underline{n'})$; then at least one segment $
(\mathpzc{a}_i,\ldots,\mathpzc{b}_i)$ 
merges with $( \underline{n})$ to form a residual segment $(\underline{n''})$.
\end{lemma}

\begin{proof}
Consider the case of $t$ unlinked segments, with at least one disjoint from the others, we aim to prove that this segment can be inserted into 
$( \underline{n})$ independently of the others to obtain a residual segment.
For each such (disjoint from the others) segment $(\mathpzc{a}_i,\ldots,\mathpzc{b}_i)$, inserted, the following conditions are satisfied:
\begin{equation}\label{eq:merge}
	\left\{
    \begin{array}{ll}
        n'_{\mathpzc{a}_i+1} = n_{\mathpzc{a}_i+1} = n'_{\mathpzc{a}_i}-1 = n_{\mathpzc{a}_i}+1 -1 \\
        n'_{\mathpzc{b}_i} = n_{\mathpzc{b}_i}+1= n_{\mathpzc{b}_i -1} -1 +1 = n_{\mathpzc{b}_i -1} = n'_{\mathpzc{b}_i -1}
    \end{array}
\right.
\end{equation}

The relations $n'_{\mathpzc{a}_i+1}= n_{\mathpzc{a}_i+1}$ and  $n'_{\mathpzc{b}_i -1}=n_{\mathpzc{b}_i -1}$  come from the fact that the elements $(\mathpzc{a}_i+1)$ and $(\mathpzc{b}_i -1)$ cannot belong to 
any other segment unlinked to $(\mathpzc{a}_i, \ldots, \mathpzc{b}_i)$. If, for any $i$, those conditions are satisfied, $(\underline{n'})$ is a residual segment, by hypothesis.

Now, let us choose a segment which does not contain zero: $(\mathpzc{a}_j, \mathpzc{b}_j)$.
Since by the Equation (\ref{eq:merge}) $n_{\mathpzc{a}_j+1}=n_{\mathpzc{a}_j}$ and $n_{\mathpzc{b}_j}= n_{\mathpzc{b}_j -1} -1$, adding only $(\mathpzc{a}_j,\ldots, 
\mathpzc{b}_j)$ yields 
equations as (\ref{eq:merge}) and therefore a residual segment.

If this segment contains zero and is disjoint from the others, then adding all segments or just this one yields the same results on the numbers of zeroes and ones: $n'_0 = n''_0$, $n'_1 = n''_1$, therefore there is no additional constraint under these circumstances.

\medskip 

Secondly, let us consider the case of a chain of inclusions, that, without loss of generality, we denote $\mathcal{S}_1 \supset \mathcal{S}_2 \supset  \mathcal{S}_3  
\ldots \supset 
\mathcal{S}_t$. 
Starting from $( \underline{n'})$, observe that adding the $t$ linear residual segments yields the following conditions:
$$n'_{\mathpzc{a}_i +1} = n_{\mathpzc{a}_i +1} + i -1 = n'_{\mathpzc{a}_i} -1 = n_{\mathpzc{a}_i} + i -1$$
$$n'_{\mathpzc{b}_i} = n_{\mathpzc{b}_i} + i = n_{\mathpzc{b}_i-1} -1 + i= n'_{\mathpzc{b}_i-1}$$

Then, for any $i$, we clearly observe $n_{\mathpzc{a}_i +1} = n_{\mathpzc{a}_i}$; and $n_{\mathpzc{b}_i}=n_{\mathpzc{b}_i-1} -1$.
Assume we only add the segment $(\mathpzc{a}_1, \ldots, \mathpzc{b}_1)$, then we observe $n''_{\mathpzc{a}_1 +1} = n''_{\mathpzc{a}_1} -1$
and $n''_{\mathpzc{b}_1} = n''_{\mathpzc{b}_1-1}$, satisfying the conditions for $(\underline{n''})$ to be a residual segment.

Assume $\mathcal{S}_t$ contains zero, then any $\mathcal{S}_i$ also. Assume there is an obstruction at zero to form a residual segment when adding $t-1$ 
segments. If 
adding only $t-1$ zeroes does not form a residual segment, but $t$ zeroes do, we had $n'_0= \frac{n_1}{2}$.
Then $n_0 + t=  \frac{n_1}{2} + t= \frac{n_1+2t}{2}$ (the option $n'_1 = n_1+2t+1$ is immediately excluded since there is at most two '1' per 
segment $\mathcal{S}_i$). 

We need to add $2t$ times '1'. Then we need at least $2t-1$ times '2' and $2t-2$ times '3'..etc. Since, $n'_1 = n_1+2t$  all $\mathcal{S}_i$'s will 
contain (10 -1). 
There is no obstruction at zero while adding solely $\mathcal{S}_1$ (i.e $n_0 + 1= \frac{n_1+2}{2}$) and since $\mathcal{S}_1 \supset \mathcal{S}_2  \ldots 
\supset \mathcal{S}_t$ 
and $\mathcal{S}_1$ needs to contain $\mathpzc{a}_1 \geq \ell+ m$, $\mathcal{S}_1$ can merge with $(\underline{n})$ to form a residual segment.

Finally, it would be possible to observe the case of a residual segment $\mathcal{S}_1$ containing $\mathcal{S}_2$ and $\mathcal{S}_3$ with $\mathcal{S}_2$ and $\mathcal{S}_3$ disjoint (or two-or more- disjoint chains of inclusions).
Again, we have:
$$n'_{\mathpzc{a}_1 +1} = n_{\mathpzc{a}_1 +1}  = n'_{\mathpzc{a}_1} -1 = n_{\mathpzc{a}_1} + 1 -1$$
Assume we only add the segment $(\mathpzc{a}_1, \ldots, \mathpzc{b}_1)$, then we observe $n''_{\mathpzc{a}_1 +1} = n''_{\mathpzc{a}_1} -1$
and $n''_{\mathpzc{b}_1} = n''_{\mathpzc{b}_1-1}$, satisfying the conditions for $(\underline{n''})$ to be a residual segment.
\end{proof}

\begin{rmk} \label{rmksisj}
We show in this remark that if $s_i= \frac{\mathpzc{a}_i + \mathpzc{b}_i}{2} = s_j = \frac{\mathpzc{a}_j + \mathpzc{b}_j}{2}$, the linear segments $(\mathpzc{a}_i, \ldots, \mathpzc{b}_i)$ with $\mathpzc{a}_i > \mathpzc{b}_i$ and $(\mathpzc{a}_j, \mathpzc{b}_j)$ with $\mathpzc{a}_j > \mathpzc{b}_j
$ are such that one of them is included in the other (therefore unlinked).

If the length of the segments are the same, they are equal; without loss of generality let us consider the following case of different lengths:
\begin{equation}\label{length}
\mathpzc{a}_i - \mathpzc{b}_i +1 > \mathpzc{a}_j - \mathpzc{b}_j +1
\end{equation}

Since $\frac{\mathpzc{a}_i+\mathpzc{b}_i}{2}= \frac{\mathpzc{a}_j+\mathpzc{b}_j}{2}$, $\mathpzc{a}_i+\mathpzc{b}_i= \mathpzc{a}_j+\mathpzc{b}_j$ and from Equation (\ref{length}) $\mathpzc{a}_i -
\mathpzc{a}_j >  
\mathpzc{b}_i- \mathpzc{b}_j$ replacing $\mathpzc{b}_i$ by $\mathpzc{a}_j+\mathpzc{b}_j -\mathpzc{a}_i$, and further $\mathpzc{a}_i$ by $\mathpzc{a}_j+\mathpzc{b}_j -\mathpzc{b}_i$, we obtain:
$$ \mathpzc{a}_j+\mathpzc{b}_j -\mathpzc{b}_i - \mathpzc{a}_j > \mathpzc{b}_i- \mathpzc{b}_j \Leftrightarrow \mathpzc{b}_j > \mathpzc{b}_i $$
Therefore, $\mathpzc{a}_i > \mathpzc{a}_j > \mathpzc{b}_j > \mathpzc{b}_i$
\end{rmk}

Consequently, the content of the proofs of the next Theorem (\ref{dsgic2}), when considering the case of equal parameters $s_i=s_j$, remain the same.

\begin{thm}\label{dsgic2} 
Let us assume $\sigma_{\nu}$ is in the cuspidal support of a generic discrete series representation $\tau$ of a standard Levi subgroup $M$ of a quasi-split reductive group $G$. Let us assume that the cuspidal support of $\tau$ satisfies the conditions (CS) (see the Definition \ref{CS}).
Let us take $\underline{s}$ in $\overline{(a_M^*)^+}$, such that $I_P^G(\tau_{\underline{s}}) \hookrightarrow I_{P_1}^G(\sigma_{\nu+\underline{s}})$ and denote $\lambda= \nu+\underline{s}$ in $\overline{a_{M_1}^M}^{+*}$.
Let us assume $\sigma_{\lambda}$ is a residual point for $\mu^G$.
Let $\pi_0$ be an irreducible generic discrete series representation of $G$ which is a subquotient in $I_P^G(\tau_{\underline{s}})$.
Then, the unique irreducible generic square-integrable subquotient, $\pi_0$, in the standard module $I_{P}^G(\tau_{\underline{s}})  
\hookrightarrow I_{P_1}
^G(\sigma_{\lambda})$ is a subrepresentation.
\end{thm}

\begin{proof}
Let us assume that $\Sigma_{\sigma}^M$ is a disjoint union of $t$ subsystems of type $A$ and a subsystem of type $\mathcal{T}$.
Let $\underline{s}= (s_1, s_2, \ldots ,s_t)$ be ordered such that $s_1 \geq s_2 \geq \ldots \geq s_t \geq 0$ with $s_i = \frac{\mathpzc{a}_i + \mathpzc{b}_i}
{2}$, for two 
(half)-integers $\mathpzc{a}_i \geq \mathpzc{b}_i$.

Using the depiction of residual points in Subsection \ref{setting2}, we write the residual point 
\begin{flushleft}
$$\sigma(\bigoplus_{i=1}^t(\mathpzc{a}_i,\ldots,\mathpzc{b}_i)(\underline{n}))~ \hbox{where} ~ \lambda ~ \hbox{reads}~ \bigoplus_{i=1}^t(\mathpzc{a}_i,\ldots, \mathpzc{b}_i)(\underline{n})$$
\end{flushleft}
Let us denote the linear residual segments $(\mathpzc{a}_i,\ldots,\mathpzc{b}_i):=\mathcal{S}_i$ and assume that for some indices $i,j \in \left\{1,\ldots, t \right\}$, the 
segments $\mathcal{S}_i, \mathcal{S}_j$ are linked. By Lemma \ref{Sio}, there exists an intertwining operator with non-generic kernel from $I_{P_1}^G(\sigma((\mathcal{S}'_1, \mathcal{S}'_2, \ldots, \mathcal{S}'_t;\underline{n})))$ to $I_{P_1}^G(\sigma(\mathcal{S}_1, \mathcal{S}_2, \ldots, \mathcal{S}_t;\underline{n}))$. Therefore, if we prove the unique irreducible discrete series subquotient appears as subrepresentation in $I_{P_1}^G(\sigma(\mathcal{S}'_1, \mathcal{S}'_2, \ldots, \mathcal{S}'_t;\underline{n}))$, it will consequently appears as subrepresentation in $I_{P_1}^G(\sigma(\mathcal{S}_1, \mathcal{S}_2, \ldots, \mathcal{S}_t;\underline{n}))$.
This means we are reduced to the case of the cuspidal support $\sigma_{\lambda}$ being constituted of $t$ \emph{unlinked} segments.
 
Further, notice that by the above Remark \ref{rmksisj} when $s_i = s_j$, the segments $\mathcal{S}_i$, and $\mathcal{S}_j$ are \emph{unlinked}. This 
allows us to treat the case $s_1= s_2 = \ldots = s_t > 0$ and $s_1> s_2 = \ldots = s_t = 0$. 

So let us assume all linear segments $(\mathpzc{a}_i,\ldots, \mathpzc{b}_i)$ are unlinked.

We prove the theorem by induction on the number $t$ of linear residual segments.

First, $t=0$, let $P_0=G$, and $\pi$ be the generic irreducible square integrable representation corresponding to the dominant residual point $
\sigma_{\lambda}:= 
\sigma(\underline{n}_{\pi_0})$.

$$I_{P_0}^G(\pi) \hookrightarrow I_{P_1}^G(\sigma(\underline{n}_{\pi_0}))$$
By Lemma \ref{firstresultslem}, $\lambda$ being in the closure of the positive Weyl chamber, the unique 
irreducible generic discrete series subquotient is necessarily a subrepresentation.

The proof of the step from $t=0$ to $t=1$ is Proposition \ref{dsr=1}.

Assume the result true for any standard module
$I_{P'_{\Theta_{\leq t}}}^G(\tau_{\underline{s}}) \hookrightarrow  I_{P_1}^G(\sigma(\bigoplus_{i=1}^t(\mathpzc{a}_i, \ldots,\mathpzc{b}_i)(\underline{n})))$ with $t$ or less 
than $t$ linear residual segments, where $P'_{\Theta_{\leq t}}$ is any standard parabolic subgroup whose Levi subgroup is obtained by removing $t$ or less 
than $t$ simple (non-extremal) roots from $\Delta$.



We consider now $\pi_0$ the unique irreducible generic discrete series subquotient in 
$$I_{P_{\Theta_{t+1}}}^G(\tau'_{\underline{s'}}) \hookrightarrow 
I_{P_1}
^G(\sigma(\bigoplus_{i=1}^t(\mathpzc{a}_i,\ldots, \mathpzc{b}_i)(\mathpzc{a}_{t+1},\ldots,\mathpzc{b}_{t+1})(\underline{n'})))$$

To distinguish with the case of a discrete series $\tau$ of $P_{\Theta_t}$, we denote $\tau'$ the irreducible generic discrete series and $s'$ in $a_{M_{\Theta_{t+1}}}*^+$.

Using Lemma \ref{merge}, we know there is at least one linear segment with index $j \in [1,t+1]$ such that $(\mathpzc{a}_j, \ldots,\mathpzc{b}_j)$ can be 
inserted in ($\underline{n'})$ to form a residual segment. Without loss of generality, let us choose this index to be $t+1$ (else we use bijective intertwining operators on the unlinked segments to set $(\mathpzc{a}_j,\ldots, \mathpzc{b}_j)$ in the last position). Then, there exists a Weyl group element $w$ such that $w((\mathpzc{a}_{t+1},\ldots,\mathpzc{b}_{t+1})(\underline{n'}))= (\underline{n})$ for a residual segment $(\underline{n})$.

Let $M_1 = M_{\Theta}$ with $\Theta = \bigcup_{i=1}^{s} \Theta_i$ for some $s > t$ and $M'= M_{\Theta'}$ where $\Theta'= \bigcup_{i=1}^{s-2} \Theta_i\cup \Theta_t\cup \left\{\underline{\alpha}_t\right\} \cup \Theta_{t+1}$, if we assume (by convention) that the root $\underline{\alpha}_t$ connects the two connected components $\Theta_t$ and $\Theta_{t+1}$.

Since $M'\cap P$ is a maximal parabolic subgroup in $M'$,
we can apply the result of Proposition \ref{dsr=1} to $\pi'$ the unique irreducible discrete series subquotient in $I_{P_1\cap M'}^{M'}(\sigma(\mathpzc{a}_{t+1},\mathpzc{b}_{t+1})(\underline{n'}))$. 

Notice that $\Sigma^{M'}$ is a reducible root system, and therefore so is $\Sigma_{\sigma}^{M'}$; it is because we choose an irreducible component of $\Sigma^{M'}$ that we can apply the result of Proposition \ref{dsr=1}.

It appears as a subrepresentation in $I_{P_1\cap M'}^{M'}(\sigma(\underline{n}))$.

Then, since the parameter $\bigoplus_{i=1}^t(\mathpzc{a}_i,\ldots, \mathpzc{b}_i)$ corresponds to a central character $\chi$ for $M'$, we have:

$$I_{P'}^{G}(\pi'_{\chi}) \hookrightarrow I_{P'}^{G}(I_{P_1\cap M'}^{M'}(\sigma(\underline{n}))_{\bigoplus_{i=1}^t(\mathpzc{a}_i, \ldots, \mathpzc{b}_i)}) \cong I_{P_1}^{G}(\sigma(\bigoplus_{i=1}^t(\mathpzc{a}_i,\ldots, \mathpzc{b}_i)(\underline{n})))$$

By Proposition \ref{dsr=1}, the subquotient $\pi'$ appears as a subrepresentation in \\
$I_{P_1\cap M'}^{M'}(\sigma(\mathpzc{a}_{t+1},\ldots,\mathpzc{b}_{t+1})(\underline{n'}))$ and therefore in the standard module embedded in $I_{P_1\cap M'}^{M'}(\sigma(\mathpzc{a}_{t+1},\ldots,\mathpzc{b}_{t+1})(\underline{n'}))$ by [U].

Since the parameter $\bigoplus_{i=1}^t(\mathpzc{a}_i,\ldots, \mathpzc{b}_i)$ correspond to a central character for $M'$, we have:

$$I_{P'}^{G}(\pi'_{\chi}) \hookrightarrow I_{P'}^{G}(I_{P_1\cap M'}^{M'}(\sigma(\mathpzc{a}_{t+1},\mathpzc{b}_{t+1})(\underline{n'}))_{\bigoplus_{i=1}^t(\mathpzc{a}_i,\ldots,\mathpzc{b}_i)}) \cong I_{P_1}^{G}(\sigma(\bigoplus_{i=1}^t(\mathpzc{a}_i,\ldots,\mathpzc{b}_i)(\mathpzc{a}_{t+1},\ldots,\mathpzc{b}_{t+1})(\underline{n'})))$$
%

We have therefore two options: Either $I_{P'}^G(\pi'_{\chi})$ is irreducible and then it is the unique irreducible generic subrepresentation in $$I_{P'}^G(I_{P_1\cap M'}^{M'}(\sigma(\bigoplus_{i=1}^t(\mathpzc{a}_i,\ldots, \mathpzc{b}_i)(\mathpzc{a}_{t+1},\ldots,\mathpzc{b}_{t+1})(\underline{n'})))\kern110pt$$

$$\kern140pt= I_{P_1}^G(\sigma(\bigoplus_{i=1}^t(\mathpzc{a}_i,\ldots,\mathpzc{b}_i)(\mathpzc{a}_{t+1},\ldots,\mathpzc{b}_{t+1})(\underline{n'})))$$ and by multiplicity one in $I_{P_{\Theta_{t+1}}}^G(\tau'_{\underline{s'}})$. Otherwise, it is reducible, but then its unique irreducible generic subquotient is also the unique irreducible generic subquotient in $I_{P_1}
^G(\sigma(\bigoplus_{i=1}^t(\mathpzc{a}_i,\ldots, \mathpzc{b}_i)(\underline{n})))$.

Then, by induction hypothesis, it embeds as a subrepresentation in $I_{P_1}^G(\sigma(\bigoplus_{i=1}^t(\mathpzc{a}_i, \ldots,\mathpzc{b}_i)(\underline{n})))$; and by [U], also in $I_{P'}^G(\pi'_{\chi})$. Hence, it embeds in $I_{P_1}^G(\sigma(\bigoplus_{i=1}^t(\mathpzc{a}_i,\ldots, \mathpzc{b}_i)(\mathpzc{a}_{t+1},\ldots,\mathpzc{b}_{t+1})(\underline{n'})))$, and therefore in 
$I_{P_{\Theta_{t+1}}}^G(\tau'_{s'})$ 
concluding this induction argument, and the proof.
\end{proof}

\subsection{Proof of the Generalized Injectivity Conjecture for Non-Discrete Series Subquotients} \label{gicnnds}
We could have $I_P^G(\tau_{s\tilde{\alpha}})$ reducible without having hypothesis 1 in Lemma \ref{lem}
satisfied, that is without having $s\tilde{\alpha}$ a pole of the $\mu$ function for $\tau$; i.e the converse of the Lemma\ref{lem} doesn't necessarily hold.

It is only in this case that a non-tempered or tempered (but not square-integrable) generic subquotient may occur in $I_{P_1}^G(\sigma_{\nu+s
\tilde{\alpha}})
$.
\begin{prop} \label{nndsgic1}
Let $\sigma_{\nu}$ be in the cuspidal support of a generic discrete series representation $\tau$ of a maximal Levi subgroup $M$ of a quasi-split reductive group $G$. Let us 
take $s\tilde{\alpha}$ in $\overline{(a_M^*)^+}$, such that $I_P^G(\tau_{s\tilde{\alpha}}) \hookrightarrow I_{P_1}^G(\sigma_{\nu+s
\tilde{\alpha}})$ and denote 
$\lambda= \nu+s\tilde{\alpha}$ in $\overline{a_{M_1}^M}^{+*}$.

Let us assume that the cuspidal support of $\tau$ satisfies the conditions \emph{CS} (see the Definition \ref{CS}).

Let us assume $\sigma_{\lambda}$ is not a residual point for $\mu^G$, and therefore the unique irreducible generic subquotient in 
$I_P^G(\tau_{s
\tilde{\alpha}})$ is essentially tempered but not square integrable or not essentially tempered. 

Then, this unique irreducible generic subquotient embeds as a submodule in $I_{P_1}^G(\sigma_{\lambda})$  and 
therefore in the standard module $I_P^G(\tau_{s\tilde{\alpha}}) \hookrightarrow I_{P_1}^G(\sigma_{\lambda})$.
\end{prop}

\begin{proof}
First, notice that if $s=0$ the induced module $I_P^G(\tau_{s\tilde{\alpha}})$ is unitary hence any irreducible subquotient is a 
subrepresentation, in the rest of 
the proof we can therefore assume $s\tilde{\alpha}$ in $(a_M^*)^+$.

Let us denote $\pi_0$ the irreducible generic tempered or non-tempered representation which appears as subquotient in a standard module $I_P^G(\tau_{s
\tilde{\alpha}})$ induced from a maximal parabolic subgroup $P$ of $G$.

We are in the context of the Subsection \ref{setting2}, and therefore we can write $\lambda:= (\mathpzc{a},\ldots, \mathpzc{b})+(\underline{n})$, for some $\mathpzc{a} > \mathpzc{b}$, 
and residual 
segment $(\underline{n})$. Here, we assume $\sigma_{\lambda}$ is not a residual point.
Then $I_P^G(\tau_{s\tilde{\alpha}}) \hookrightarrow I_{P_1}^G(\sigma(\mathpzc{a},\mathpzc{b}, \underline{n}))$ has a unique irreducible generic subquotient which is tempered or non-tempered.

Following the proof of the Theorem \ref{conditionsonlambda} 2)a) and b), we can write this unique irreducible generic subquotient $I_{P'}^G(\tau'_{\nu'})$, either it is embeds in an induced module which satisfies the conditions 2)a) or 2)b) of the Theorem \ref{conditionsonlambda} and then we can conclude by [U]. This is the context of existence of an intertwining operator with non-generic kernel between the induced module with cuspidal strings  $(\mathpzc{a}', \mathpzc{b}', \underline{n'})$ and $(\mathpzc{a}, \mathpzc{b}, \underline{n})$.

Otherwise, one observes that passing from $(\mathpzc{a}', \mathpzc{b}', \underline{n'})$ to $(\mathpzc{a}, \mathpzc{b}, \underline{n})$ require certain elements $\gamma$, with $\mathpzc{a} \geq 
\gamma > \mathpzc{a}'$,  to move up, i.e. from right to left.
This means using rank one operators which change $(\gamma+n,\gamma)$ to $(\gamma,\gamma+n)$ for integers $n\geq 1$, those rank one operators may 
clearly have generic kernel.

In this context, we will rather use the results of Proposition \ref{dsr=1}.

Consider again $I_{P'}^G(\tau'_{\nu'})$ embedded in $I_{P_1}^G(\sigma(\mathpzc{a}', \mathpzc{b}', \underline{n'}))$. Let us denote $\pi'$ the unique irreducible generic discrete series subquotient corresponding to the dominant residual point $\sigma((\underline{n'}))$. Let $M''= M_{\Delta - \{\underline{\alpha_1}, \ldots, \underline{\alpha_{a-b+1}}\}}$ be a standard Levi subgroup, we have:
$$\pi' \hookrightarrow I_{P_1\cap M''}^{M''}((\sigma((\underline{n'})))$$

Since the character corresponding to the linear residual segment $(\mathpzc{a}',\ldots,\mathpzc{b}')$ is central for $M''$, we write:
$$\pi'_{(\mathpzc{a}',\ldots,\mathpzc{b}')} \hookrightarrow I_{P_1\cap M''}^{M''}(\sigma((\mathpzc{a}',\ldots,\mathpzc{b}')+(\underline{n'}))) \cong I_{P_1\cap M''}^{M''}(\sigma(\underline{n'}))_{(\mathpzc{a}',\ldots,\mathpzc{b}')}$$
Since $\tau'_{\nu'}$ is irreducible (and generic), we also have $\tau'_{\nu'} \hookrightarrow I_{P_1\cap M'}^{M'}(\sigma((\mathpzc{a}',\ldots,\mathpzc{b}')+(\underline{n'})))$ 
we know: 
\begin{equation}\label{111}
\tau'_{\nu'} \hookrightarrow I_{P''}^{M'}(\pi'_{(\mathpzc{a}',\ldots,\mathpzc{b}')}) \hookrightarrow I_{P_1\cap M'}^{M'}(\sigma((\mathpzc{a}',\ldots,\mathpzc{b}')+(\underline{n'})))
\end{equation}
By the generalized injectivity conjecture for square-integrable subquotient (Proposition \ref{dsr=1}), any standard module embedded in $I_{P_1\cap M''}^{M''}(\sigma((\underline{n'})))$ has $\pi'$ as subrepresentation. We may therefore embed $\pi'$ as subrepresentation in $$I_{P_1\cap M''}^{M''}((w_{\flat}\sigma)((a^{\flat}, \mathpzc{b}^{\flat},\underline{n}^{\flat})))$$

with $w_{\flat}\sigma \cong \sigma$, and therefore inducing Equation \ref{111} to $G$

$$I_{P'}^G(\tau'_{\nu'}) \hookrightarrow I_{P_1}^{G}((w_{\flat}\sigma)((\mathpzc{a}',\ldots,\mathpzc{b}')+(\mathpzc{a}^{\flat}, \mathpzc{b}^{\flat})(\underline{n}^{\flat}))$$
The sequence $(\mathpzc{a}^{\flat}, \mathpzc{b}^{\flat},\underline{n}^{\flat})$ is chosen appropriately to have an intertwinning operator with non-generic kernel from 
$I_{P_1}^G(\sigma((\mathpzc{a}',\ldots, \mathpzc{b}')+(\mathpzc{a}^{\flat}, \mathpzc{b}^{\flat},\underline{n}^{\flat}))$ to $I_{P_1}^G(\sigma(\mathpzc{a},\mathpzc{b},\underline{n}))$.

The unique irreducible generic subrepresentation $I_{P'}^G(\tau'_{\nu'})$ in
$I_{P_1}^G(\sigma(\mathpzc{a},\mathpzc{b},\underline{n}))$ cannot appear in the kernel and therefore appears in the image of this operator. It therefore appears as a subrepresentation in $I_{P_1}^G(\sigma(\mathpzc{a},\mathpzc{b},\underline{n}))$ and conclude by [U].
\end{proof}

\begin{thm}\label{nndsgic2}
Let $\sigma_{\nu}$ be in the cuspidal support of a generic discrete series representation $\tau$ of a standard Levi subgroup $M$ of a quasi-split reductive group.

Let us take $\underline{s}$ in $\overline{(a_M^*)^+}$, such that $I_P^G(\tau_{\underline{s}}) \hookrightarrow I_{P_1}^G(\sigma_{\nu+
\underline{s}})$ and 
denote $\lambda= \nu+\underline{s}$ in $\overline{a_{M_1}^M}^{+*}$.
Let us assume that $\sigma_{\lambda}$ is not a residual point for $\mu^G$ and that the unique irreducible generic subquotient satisfies the conditions \emph{CS} (see the Definition \ref{CS}).

Then, the unique irreducible generic in $I_{P}^G(\tau_{\underline{s}})$ (which is essentially tempered or non-tempered) embeds as a 
subrepresentation
in $I_{P}^G(\tau_{\underline{s}})\hookrightarrow I_{P_1}^G(\sigma_{\lambda})$.
\end{thm}

\begin{proof}
First, notice that, by the Remark \ref{rmksisj}, when $s_i = s_j$ the segments $\mathcal{S}_i$, and $\mathcal{S}_j$ are \emph{unlinked}. 

Using the argument given in Subsection \ref{setting2}, we write $\sigma_{\lambda}$ as $\sigma(\bigoplus_{i=1}^t(\mathpzc{a}_i,\ldots,\mathpzc{b}_i)
(\underline{n}))$, where $
\lambda$ reads $\bigoplus_{i=1}^t(\mathpzc{a}_i,\ldots,\mathpzc{b}_i)(\underline{n})$.

The proof goes along the same inductive line than in the proof of Proposition \ref{dsgic2}. 

The case of $t=1$ is Proposition \ref{nndsgic1}. That is, given a cuspidal support $(P_1,\sigma_{\lambda})$, for any standard module induced 
from a 
maximal parabolic subgroup $P$: $I_P^G(\tau_{\underline{s}}) \hookrightarrow I_{P_1}^G(\sigma_{\lambda})$, the unique irreducible generic 
subquotient is a 
subrepresentation.
We use an induction argument on the number $t$ of linear residual segments obtained when removing $t$ simple roots to define the Levi subgroup 
$M \subset P$. 
Considering that an essentially tempered or non-tempered irreducible generic subquotient in a standard module with $t$ linear residual 
segments 
$I_{P_{\Theta_t}}^G(\tau_{\underline{s}})$ is necessarily a subrepresentation; one uses the same arguments than in the proof of 
Theorem 
\ref{dsgic2} to conclude that a tempered or non-tempered irreducible generic subquotient in a standard module with $t+1$ linear residual 
segments 
$I_{P_{\Theta_{t+1}}}^G(\tau'_{\underline{s'}})$ is a subrepresentation, therefore proving the theorem.
\end{proof}

Eventually, we now consider the generic subquotients of $I_P^G(\gamma_{s\tilde{\alpha}})$ when $\gamma$ is a generic irreducible 
tempered 
representation.

\begin{cor}[Standard modules]\label{standard}
Let $G$ be a quasi-split reductive group of type $A,B,C$ or $D$ and let us assume $\Sigma_{\sigma}$ is irreducible.

The unique irreducible generic subquotient of $I_P^G(\gamma_{\underline{s}})$ when $\gamma$ is a generic irreducible 
tempered representation of a standard Levi $M$ is a subrepresentation.
\end{cor}

\begin{proof}
Let $P = MU$. By Theorem \ref{renard}, as a tempered representation of $M$, $\gamma$ appears as a subrepresentation of $I_{P_3 \cap M}^{M}(\tau)$ for some discrete series $\tau$ and standard parabolic $P_3= M_3U$ of $G$; $\tau$ is generic irreducible representation of the Levi subgroup $M_3$, therefore 
$$I_P^G(\gamma_{\underline{s}}) \hookrightarrow I_P^G(I_{M\cap P_3}^{M}(\tau))_{\underline{s}} \cong I_{P_3}^G(\tau_{\underline{s}})$$ 
where $P_3$ is not necessarily a maximal parabolic subgroup of $G$.
Since $\underline{s}$ is in $(a_M^*)^+$, $\underline{s}$ is in $\overline{(a_{M_3}^*)^+}$. Let us write this parameter $\overline{\underline{s}}$ when it is in $\overline{(a_{M_3}^*)^+}$.

The unique irreducible generic subquotients of $I_P^G(\gamma_{\underline{s}})$ are the unique irreducible generic subquotients of $I_{P_3}^G(\tau_{\overline{\underline{s}}})$, where $\overline{\underline{s}}$ is in $\overline{(a_{M_3}^*)^+}$. 
Since $P_3$ is not a maximal parabolic subgroup of $G$, we may now use Theorems \ref{dsgic2} and \ref{nndsgic2} with $\overline{\underline{s}}$ in $\overline{(a_{M_3}^*)^+}$ to conclude that these unique irreducible generic subquotients, whether square-integrable or not, are subrepresentations.
\end{proof}

\section{The case $\Sigma_{\sigma}$ reducible} \label{sigmared}

Let us recall that the set $\Sigma_{\sigma}$ is a root system in a subspace of $a_{M_1}^*$ (cf. \cite{silbergerSD} 3.5) and we assume that the 
irreducible 
components of $\Sigma _{\sigma }$ are all of type $A$, $B$, $C$ or $D$. In Proposition \ref{1.13bis}, we have denoted for each irreducible 
component $
\Sigma _{\sigma ,i}$ of $\Sigma _{\sigma}$, by $a_{M_1}^{M_i*}$ the subspace of $a_{M_1}^{G*}$ generated by $\Sigma _{\sigma ,i}$, by 
$d_i$ its 
dimension and by $e_{i,1},\dots ,e_{i, d_i}$ a basis of $a_{M_1}^{M_i*}$ (resp. of a vector space of dimension $d_i+1$ containing $a_{M_1}^{M_i*}$ if $\Sigma _{\sigma ,i}$ is of type $A$) so that the elements of the root system $\Sigma _{\sigma ,i}$ are written in this basis as in Bourbaki, \cite{bourbaki}.

The following result is analogous to Proposition 1.10 in \cite{heiermannope}.
Recall $\mathcal{O}$ denotes the set of equivalence classes of representations 
of the form $\sigma\otimes\chi$ where $\chi$ is an unramified character of $M_1$.

\begin{prop}
Let $P_1'=M_1U_1'$, and $P_1=M_1U_1$.
If the intersection of $\Sigma(P_1)\cap \Sigma(\overline{P'_1})$ with $\Sigma_{\sigma}$ is empty, the operator $J_{P_1'|P_1}$ is well 
defined and 
bijective on $\mathcal{O}$.
\end{prop}

\begin{proof}
The operator $J_{P_1'|P_1}$ is decomposed in elementary operators which come from intertwining operators relative to $(M_1)_{\alpha}$ with $\alpha \notin \Sigma_{\sigma}$, so it is 
enough to consider the case where $P_1$ is a maximal parabolic subgroup of $G$ and $P'_1= \overline{P_1}$. Then, if $\alpha \notin \Sigma_{\sigma}$ and by the same reasoning than in the Lemma \ref{lemmabeta}, the operator $J_{P_1'|P_1}$ is well defined and bijective at any point on $\mathcal{O}$.
\end{proof}

Let $G$, $\pi_0$ $\sigma_{\lambda}$, $\lambda \in a_{M_1}^*$ be defined as in the main Theorem \ref{mainresult2}.

In this section, we consider the case of a \emph{reducible} root system $\Sigma_{\sigma}$. As explained in Appendix \ref{lab}, this case occurs in particular when $\Sigma_{\Theta}$ (see the notations in Appendix \ref{lab}) is reducible, and then $\Theta$ has connected components of type $A$ of different lengths. An example is the following Dynkin diagram for $\Theta$:

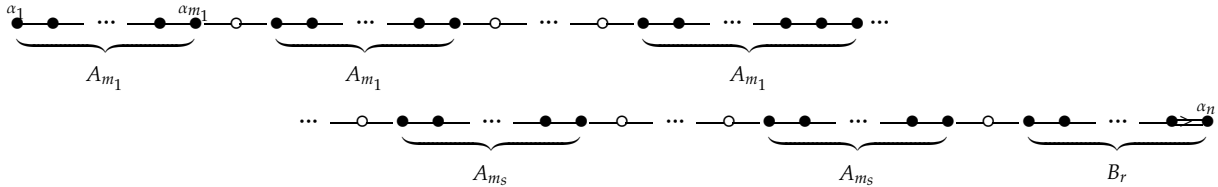
\begin{multline*}
\underbrace{\disque\noteNC{\alpha_1}\th\disque\th\points\th\disque \th\disque\noteNC{\alpha_{m_1}}}_{A_{m_1}} 
\th\cercle\th\underbrace{\disque\th\disque\th\points\th\disque  \th\disque}_{A_{m_1}} \th\cercle\th
\points\th\cercle\th \underbrace{\disque\th\disque\th\points\th\disque \th\disque\th\disque}_{A_{m_1}} \points \\
\points\th\cercle\th \underbrace{\disque\th\disque\th\points\th\disque \th\disque}_{A_{m_s}} \th\cercle\th
\points\th\cercle\th \underbrace{\disque\th\disque\th\points\th\disque 
\th\disque}_{A_{m_s}}\th\cercle\th
  \underbrace{\disque\th\disque\th\points\th
\disque \thh\fdroite\disque\noteNC{\alpha_n}}_{B_r}
\end{multline*}

Let us assume $\Theta$ is a disjoint union of components of type $A_{m_i} ~ i=1 \ldots s~; \hbox{and} \,\, m_i \neq m_{i+1}\,\hbox{for any i}$, where each component of type $A_{m_i}$ appears $d_i$ times. Set $m_i = k_i -1$. 

Let us denote $\Delta_{M_1}^i = \left\{\alpha_{i,1}, \ldots, \alpha_{i,d_i} \right\}$ the non-trivial restrictions of roots in $\Sigma$, generating the set $a_{M_1}^{M^i}*$. Similar to the case of $\Sigma_{\sigma}$ irreducible, we may have $\Delta_{\sigma_i} = \left\{\alpha_{i,1}, \ldots, \beta_{i,d_i}\right\}$ where $\beta_{i,d_i}$ can be different from $\alpha_{i,d_i}$ in the case of type $B,C$ or $D$. 
For any $i\neq s$, the pre-image of the root $\alpha_{i,d_i}$ is \emph{not} simple. 

Indeed, for instance, in the above Dynkin diagram, the first root 'removed' is $e_{k_1} - e_{k_1+1}$, the second is $e_{2k_1} - e_{2k_1+1}$, etc; they are simple roots and their restrictions to $A_{M_1}$ are roots of $\Delta_{M_1}^1$ (the generating set of $a_{M_1}^{M^1}*$) ; the last root to consider is $e_{k_1d_1} - e_{n-r+1}$ which restricts to $e_{k_1d_1}$; then the preimage of $e_{k_1d_1}$ is not simple. 

However, since $e_{n-r} - e_{n-r+1}$ restricts to $e_{n-r}$; the pre-image of $\alpha_{s,d_s}$ is simple.

The Levi subgroup $M^i$ is defined such that $\Delta^{M^i} = \Delta^{M_1}\cup \left\{\underline{\alpha_{i,1}}, \ldots, \underline{\alpha_{i,d_i}} \right\} ~~ \text{where} ~~\Delta_{M_1}^i = \left\{\alpha_{i,1}, \ldots, \alpha_{i,d_i} \right\}$.

It is a \emph{standard} Levi subgroup for $i=s$. This is quite an important remark since most of our results in the previous sections were conditional on having \emph{standard} parabolic subgroups.
%
 
Furthermore, since $\Sigma_{\sigma,i}$ generates $a_{M_1}^{M^i}*$ and is of rank $d_i$, the semi-simple rank of $M^i$ is $d_i + rk_{ss}(M_1)$.
Since $\Sigma_{\sigma,i}$ is irreducible, an equivalent of Proposition \ref{H} is satisfied for $M^i$.

\begin{prop} \label{sigmasigmareducible}
Let $\pi_0$ be an irreducible generic representation of a quasi-split reductive group $G$, and assume it is the unique irreducible generic subquotient in the standard module $I_P^G(\tau_{s\tilde{\alpha}})$, where $M$ is a maximal Levi subgroup (and $\alpha$ is not an extremal simple root on the Dynkin diagram of $\Sigma$) of $G$ and $\tau$ is an irreducible generic discrete series of $M$.
Let us assume $\Sigma_{\sigma}$ is reducible.

Then $\pi_0$ is a subrepresentation in the standard module $I_P^G(\tau_{s\tilde{\alpha}})$.
\end{prop}
\begin{proof}
The proof starts with the setting of Section \ref{io1}: $\tau$ is an irreducible generic discrete series of a maximal Levi subgroup, and by Heiermann-Opdam's result, $\tau \hookrightarrow I_{P_1\cap M}^{M}(\sigma_{\nu})$, for $\nu \in ({a_{M_1}^M}^*)^+$. Then, $\nu$ is a residual point for $\mu^M$. 

Let us write $\Sigma_{\sigma}^M= \bigcup_{i=1}^{r+1}{\Sigma_{\sigma, i}^M}$, then the residual point condition is $\dim((a_{M_1}^{M})^*) = 
rk(\Sigma_{\sigma}
^M) = \sum_{i=1}^{r+1}d_i^M$, where $d_i^M$ is the dimension of $(a_{M_1}^{M^i})^*$ generated by $\Sigma_{\sigma,i}^M$. The residual point $\nu$ decomposes in $r+1$ disjoint residual segments: $\nu= (\nu_1, \ldots, \nu_{r+1}):= (\underline{n_1}, \underline{n_2}, \ldots, \underline{n_{r+1}})$.

Since $\Sigma^M$ decomposes into two disjoint irreducible components, one of them being of type $A$, the restrictions of simple roots of this irreducible component of type $A$ in $\Delta^M$ generates an irreducible component of $\Sigma_{\sigma}$ of type $A$, let us denote this $A$ component $\Sigma_{\sigma, I}^M$ for $I \in \{1, \ldots, r+1\}$, $d_I= \mathpzc{b} - \gamma$, and denote $\nu_I+ s\tilde{\alpha}:= (\mathpzc{b},\ldots, \gamma)$ the twisted residual segment of type $A$. 

Let us further assume that there is one index $j$ such that there exists a residual 
segment $(\underline{n'_j})$ of length $\mathpzc{b} - \gamma + 1 + d_j$ and type $\mathcal{T}$ ($B,C$ or $D$) in the $W_{\sigma}$-orbit of $(\mathpzc{b}, \gamma)
(\underline{n_j})$ where the residual segment $(\underline{n_j})$ is of the same type as $\mathcal{T}$.

Since all intertwining operators corresponding to rank one operators associated to $s_{\beta}$ for $\beta \notin \Delta_{\sigma}$ are bijective (see Lemma \ref{lemmabeta}), all intertwining operators interchanging any two residual segments $(\underline{n_k})$ and $(\underline{n_{k'}})$ are bijective. Therefore, we can interchange the 
positions of all residual segments (or said differently interchange the order of the irreducible components for $i=1, \ldots, r+1$) and therefore set $(\mathpzc{b},\ldots, \gamma)(\underline{n_j})$  in the last position, i.e we set $I=r, j=r+1$. This flexibility is quite powerful since it allows us to circumvent the difficulty arising with $M^i$ not being standard for any $i \neq r$.

When adding the root $\alpha$ to $\Theta$ (when inducing from $M$ to $G$), we form from the disjoint union $\Sigma_{\sigma, r}^M \bigcup \Sigma_{\sigma, r+1}^M$ the irreducible root system that we denote $\Sigma_{\sigma,r}$. 

The Levi subgroup $M^r$ is the smallest standard Levi subgroup of $G$ containing $M_1$, the simple root $\alpha$ and the set of simple roots whose restrictions to $A_{M_1}$ lie in $\Delta_{M_1}^r$.
It is a group of semi-simple rank $d_r + rk_{ss}(M_1)$. We may, therefore, apply the results of the previous subsections with $\Sigma_{\sigma}$ irreducible to this context: Let us assume first the unique irreducible generic subquotient $\pi$ is discrete series. From the result of Heiermann-Opdam, we have:$$\pi \hookrightarrow I_{P_1\cap M^r}^{M^r}(\sigma(\underline{n'_r})) $$
where the residual segment $(\underline{n'_r})$ is the dominant residual segment in the $W_{\sigma}$-orbit of $(\mathpzc{b}, \gamma, \underline{n_r})$.
The unramified character $\chi$ corresponding to the remaining residual segments $(n_k)$'s, $k \neq r,r+1$ is a central character of $M^r$ (since it's expression in the $a_{M_1}^*$ is orthogonal to all the roots in $\Delta^{M^r}$). Then:
$$\pi_{\chi} \hookrightarrow I_{P_1\cap M^r}^{M^r}(\sigma(\underline{n'_r}))_{\bigoplus_{j\neq r,r+1}(\underline{n_j})} $$
As a result: 
\begin{equation} \label{sigmaredeq}
\pi_0 \hookrightarrow I_{P^r}^G(\pi_{\chi}) \hookrightarrow I_{P_1}^G(\sigma(\bigoplus_{j\neq r}(\underline{n_j})+(\underline{n'_r})))
\end{equation}
In Equation (\ref{sigmaredeq}), we claim $\pi_0$ embeds first in $I_{P_1}^G(\sigma(\bigoplus_{j\neq r}\underline{n_j})+(\underline{n'_r})))$ by the Heiermann-Opdam embedding result (since the residual segment $\bigoplus_{j\neq r}(\underline{n_j})+(\underline{n'_r})$ corresponds to a parameter in $\overline{(a_{M_1}^*)^+}$), therefore it should embed in $I_{P^r}^G(\pi_{\chi})$ by [U].

Applying our conclusion in the case of irreducible root system (in Proposition \ref{dsr=1}) to $\Sigma_{\sigma,r}$, we embed $\pi$ in the induced module $I_{P_1\cap M^r}^{M^r}(\sigma(\mathpzc{b}, \gamma, \underline{n_r}))$ as a subrepresentation (and therefore in a standard module $I_{P \cap M^r}^{M^r}(\tau_{\frac{\mathpzc{b}+ \gamma}{2}})$ embedded in $I_{P_1\cap M^r}^{M^r}(\sigma(\mathpzc{b}, \gamma, \underline{n_r}))$).

$$\pi_{\chi} \hookrightarrow I_{P_1\cap M^r}^{M^r}(\sigma(\mathpzc{b}, \gamma, \underline{n_r}))_{\bigoplus_{j\neq r,r+1}(\underline{n_j})} \cong  I_{P_1\cap M^r}^{M^r}(\sigma(\mathpzc{b}, \gamma, \underline{n_r})+\bigoplus_{j\neq r,r+1}(\underline{n_j})) $$
Therefore: 
$$\pi_0 \hookrightarrow I_{P^r}^G(\pi_{\chi}) \hookrightarrow I_{P_1}^G(\sigma(\bigoplus_{j\neq r}(\underline{n_j})+(\mathpzc{b}, \gamma, \underline{n_r}))$$
In case $\pi$ is non-(essentially) square integrable, i.e tempered or non-tempered, and embeds in $I_{P_1\cap M^r}^{M^r}((\sigma(\mathpzc{b}',\gamma', \underline{n'_r}))$ (see the construction in the Section \ref{conditions}, 2)a)), we had shown in Proposition \ref{nndsgic1} there existed an intertwining operator with non-generic kernel sending $\pi$ in $I_{P_1\cap M^r}^{M^r}(\sigma( \mathpzc{b},\gamma, \underline{n_r}))$.
Since the other remaining residual segments $(n'_k)$'s, $k \neq r,r+1$ do not contribute when minimizing the Langlands parameter $\nu'$, the unique irreducible generic subquotient in $I_{P_1}^G(\sigma(\bigoplus_{k\neq r}(\underline{n_k})+(\mathpzc{b},\gamma,\underline{n_r})))$ embeds in $I_{P_1}^G(\sigma(\bigoplus_{k\neq r}(\underline{n_k})+(\mathpzc{b}',\gamma',\underline{n'_r})))$ and we can use the inducting of the previously defined (at the level of $M^r$) intertwining operator to send this generic subquotient as a 
subrepresentation in $I_{P_1}^G(\sigma(\bigoplus_{k \neq r}(\underline{n_k})+(\mathpzc{b},\gamma,\underline{n_r})))$.
We conclude the argument as usual: by [U].
\end{proof}

\begin{prop} \label{sigmasigmareducible2}
Let $\pi_0$ be an irreducible generic representation and assume it is the unique irreducible generic subquotient in the standard module $I_P^G(\tau_{s})$, where the set of simple roots in $M$ ($\Delta^M$) is the set of simple roots $\Delta$ minus $t$ simple roots, $s=(s_1, \ldots, s_t)$ such that $s_1 \geq s_2 \geq \ldots \geq s_t$ and $\tau$ is an irreducible generic discrete series.

Then it is a subrepresentation.
\end{prop}

\begin{proof}
The representation $\tau$ is an irreducible generic discrete series of a non-maximal Levi subgroup $M$ such that $I_{P}^G(\tau_{s})$ is a standard module. By Heiermann-Opdam's result, $\tau \hookrightarrow I_{P_1\cap M}^{M}(\sigma_{\nu})$, for $\nu \in \overline{({a_{M_1}^M}^*)^+}$. Then, $\nu$ is a residual point for $\mu^M$.

Let us denote $M = M_{\Theta}$. 
Then $\Theta=\bigcup_{i=1}^{t+1}\Theta_i$ where $\Theta_i$, for $i \in \left\{1, \ldots, t \right\}$ is of type $A$.

Since $M_1$ is a standard Levi subgroup of $G$ contained in $M$, we can write $\Sigma_{\sigma}^M= \bigcup_{i=1}^{t+r}\Sigma_{\sigma, i}^M$, then the residual point condition is $\dim((a_{M_1}^{M})^*) = 
rk(\Sigma_{\sigma}
^M) = \sum_{i=1}^{r+t}d_i^M$, where $d_i^M$ is the dimension of $(a_{M_1}^{M^i})^*$ generated by $\Sigma_{\sigma,i}^M$. The residual point $\nu$ 
decomposes in $t$ linear residual segments along with $r$ residual segments: $\nu= (\nu_1, \ldots, \nu_{r+t}):= (\underline{n_1}, \underline{n_2}, \ldots, \underline{n_{r+t}})$.

Adding the twist $s=(s_1, \ldots, s_t)$, we obtain a parameter $\lambda$ in $(a_{M_1}^G)^*$ composed of $t$ twisted linear residual segments $\left\{ (\mathpzc{a}_i,\ldots,\mathpzc{b}_i)\right\}_{i=1}^{t}$ and $r$ residual segments $(\underline{n_{t+1}}, \underline{n_{t+2}}, \ldots, \underline{n_{t+r}})$.

Let us first assume that $\lambda$ is a residual point.
 
This means all linear residual segments can be incorporated in the $r$ residual segments of type $\mathcal{T}$ to form residual segments $\left\{(\underline{n'_j})\right\}_{j=1}^r$ of type $\mathcal{T}$ and length $d_i$ such that $\sum_i d_i = d$ where $d$ is $rk_{ss}(G)-rk_{ss}(M_1)= \dim a_{M_1} -\dim a_G$. It is also possible that, as twisted linear residual segments they are already in a form as in Proposition \ref{extremalroots}. In that case, the linear residual segment need not be incorporated in any residual segment of type $\mathcal{T}$.

Furthermore, as in the proof of Theorem \ref{dsgic2}, we can reduce our study to the case of unlinked residual linear segments.

By Heiermann-Opdam's Proposition (\ref{opdamh}):
$$\pi_0 \hookrightarrow I_{P_1}^G(\sigma(\bigoplus_j\underline{n'_j}))$$

Let us consider the last irreducible component $\Sigma_{\sigma,r}$ of $\Sigma_{\sigma}$ and the residual segment $(\underline{n'_r})$ associated to it.

Let us assume this irreducible subsystem is obtained from some subsystems $\Sigma_{\sigma, i}^M$ of type $A$ denoted $A_q , \ldots, A_s$ and one of type $\mathcal{T}$ when inducing from $M$ to $G$

$\left\{ A_q , \ldots, A_s \right\} 
\leftrightarrow \left\{\mathcal{T}\right\} ~~~~~
\left\{(\mathpzc{b}_{r,q},\ldots, \gamma_{r,q}), \ldots, (\mathpzc{b}_{r,s},\ldots, \gamma_{r,s})\right\} \leftrightarrow \left\{(\underline{n_r})\right\}$

The Levi subgroup $M^r$ is the smallest standard Levi subgroup of $G$ containing $M_1$, $s$ simple roots (among the $t$ simple roots in $\Delta - \Theta$) and the set of roots whose restrictions to $A_{M_1}$ lie in $\Delta_{M_1}^r$.
It is a group of semi-simple rank $d_r + rk_{ss}(M_1)$.

We may therefore apply the results of the previous subsections with $\Sigma_{\sigma}$ irreducible to this context: the unique irreducible generic discrete series, $\pi$, in the induced module $I_{P_1\cap M^r}^{M^r}(\sigma(\bigoplus_{j=q}^s(\mathpzc{b}_{r,j}, \gamma_{r,j})+ (\underline{n_r}))$ is a subrepresentation.

As in the proof of the Proposition \ref{sigmasigmareducible}, since $\pi$ also embeds in $I_{P_1\cap M^r}^{M^r}(\sigma(\underline{n'_r}))$, when we add the twist by the central character corresponding to $\bigoplus_{k\neq r}(n'_k)$, we obtain: 
$$\pi_0 \hookrightarrow I_P^G(\pi_{\chi}) \hookrightarrow I_{P^r}^G(I_{P_1\cap M^r}^{M^r}(\sigma(\bigoplus_{j=k}^s(\mathpzc{b}_{r,j}, \ldots,\gamma_{r,j})+ (\underline{n_r}))_{\bigoplus_{k\neq r,r+1}(\underline{n'_k})}))$$

In case $\pi$ is non-tempered, and embeds (as a subrepresentation) in $I_{P_1\cap M^r}^{M^r}((\sigma(\mathpzc{b}',\gamma',\underline{n''_r}))$, we had shown in Proposition \ref{nndsgic1} there 
existed an intertwining operator with non-generic kernel sending $\pi$ in $I_{P_1\cap M^r}^{}(\sigma( \mathpzc{b}, \gamma, \underline{n_r}))$.

Since the other remaining residual segments $(n'_k)$'s, $k \neq r$ do not contribute when minimizing the Langlands parameter $\nu'$, the 
unique irreducible generic subquotient in 
$I_{P_1}^G(\sigma(\bigoplus_{k\neq r}(\underline{n'_k})+(\mathpzc{b},\gamma,\underline{n_r})))$ embeds in 
$I_{P^r}^G(\sigma(\bigoplus_{k\neq r}(\underline{n'_k})+(\mathpzc{b}',\gamma',\underline{n'_r})))$ 
and we can use the inducting of the previously defined intertwining operator to send this generic subquotient as a 
subrepresentation in $I_{P_1}^G(\sigma(\bigoplus_{k\neq r}(\underline{n'_k})+(\mathpzc{b},\gamma,\underline{n_r})))$.

Then 
$$\pi_0 \hookrightarrow I_{P^r}^G(\pi_{\chi}) \hookrightarrow I_{P_1}^G(\sigma(\bigoplus_{k\neq r}(\underline{n'_k})+\bigoplus_{j=q}^s(\mathpzc{b}_{r,j}, \gamma_{r,j})+ (\underline{n_r})))$$

We conclude the argument with [U] as usual.

Using bijective intertwining operators, we now reorganize this cuspidal support so as to put the linear residual segments $\bigoplus_{j=q}^s(\mathpzc{b}_{r,j}, \gamma_{r,j})$ on the left-most part and $\Sigma_{\sigma, r-1}$ in the right-most part. The residual segment $(\underline{n'_{r-1}})$ is (possibly) again formed of some linear residual segments $(\mathpzc{b}_i, \gamma_i)$ and the residual segment $(\underline{n_{r-1}})$.
We argue just as above. 
Since the linear residual segments are unlinked, we can reorganize them so as to insure $s_1 \geq s_2 \geq \ldots s_t$.

Eventually repeating this procedure, 
$$\pi_0 \hookrightarrow I_{P_1}^G(\sigma(\bigoplus_{i=1}^t(\mathpzc{b}_i, \gamma_i)+\bigoplus_{j=1}^r(\underline{n_j})))$$

Further, by [U] the generic piece also embeds as a subrepresentation in the standard module.
\end{proof}

\begin{cor}\label{sigmasigmareducible22}
Let $\pi_0$ be an irreducible generic representation of $G$ and assume it is the unique irreducible generic subquotient in the standard module $I_P^G(\gamma_{\underline{s}})$, where $M$ is a standard Levi subgroup of $G$. Let us assume $\Sigma_{\sigma}$ is reducible.

Then it is a subrepresentation.
\end{cor}

\begin{proof}
Let $P= MU$. We argue as in the Corollary \ref{standard}: using the Theorem \ref{renard}, the tempered representation of $M$, $\gamma$, appears as a subrepresentation of $I_{P_3 \cap M}^{M}(\tau)$ for 
some discrete series $\tau$ and standard parabolic subgroup $P_3= M_3U$ of $G$; $\tau$ is a generic irreducible representation of the standard Levi subgroup $M_3$, therefore 
$$I_P^G(\gamma_{\underline{s}}) \hookrightarrow I_P^G(I_{M\cap P_3}^{M}(\tau))_{\underline{s}} \cong I_{P_3}^G(\tau_{\underline{s}})$$ 
where $P_3$ is not necessarily a maximal parabolic subgroup of $G$.

Since $\underline{s}$ is in $(a_M^*)^+$, $\underline{s}$ is in $\overline{(a_{M_3}^*)^+}$. Let us write this parameter $\overline{\underline{s}}$ when it is in $\overline{(a_{M_3}^*)^+}$.

The unique irreducible generic subquotients of $I_P^G(\gamma_{\underline{s}})$ are the unique irreducible generic subquotients of $I_{P_3}^G(\tau_{\overline{\underline{s}}})$, where $\overline{\underline{s}}$ is in $\overline{(a_{M_3}^*)^+}$. Since $P_3$ is not a maximal parabolic subgroup of $G$, we use the result of the Proposition \ref{sigmasigmareducible2}.
\end{proof}

\section{Exceptional groups} \label{excp}
The arguments developed in the context of reductive groups whose roots systems are of classical type may apply in the context of exceptional groups provided the set $W(M_1)$ is equal to the Weyl group $W_{\sigma}$ or differ by one element as in the Corollary \ref{mmm}. However, this hypothesis shall not be necessarily satisfied, as the $E_8$ Example 5.3.3 in \cite{silbergerSD} illustrates: in this example, where $W_{\sigma}$, the Weyl group of $\Sigma_{\sigma}$, is of type $D_8$, it shall be rather different from $W(A_0)$. 

In an auxiliary work \cite{PrSR}, we have observed that in most cases where a root system of rank $d = \hbox{dim}(a_{M_1}^*\slash a_{G}^*)$ occurs in $\Sigma_{\Theta}$, it is of type $A$ or $D$; or of very small rank (such as in $F_4$).

Further, the main result of \cite{PrSR} (Theorem 2) is that only classical root systems occur in $\Sigma_{\Theta}$; except when $G$ is of type $E_8$ and $\Theta$ contains one (any) root of $E_8$.

This latter case along with the case of $\Theta = {\emptyset}$ (in the context of exceptional groups), $\Sigma_{\Theta} = \Sigma$, $M_1 = P_0= B$ and $\sigma$ a generic irreducible representation of $P_0$ (in particular the case of trivial  representation $\sigma$) shall be treated in an independent work since the combinatorial arguments given in this work shall not apply as easily.
 
Furthermore, it might be necessary for the case $E_8$ and $\Theta$ containing only one root to obtain a result analogous to the Proposition \ref{1.13bis} \emph{which includes the exceptional root systems}; it would allow to use the weighted Dynkin diagrams (of exceptional type) to express the coordinates of residual points.

\begin{enumerate}
\item 
Let us assume $\Sigma_{\Theta}$ contains $\Sigma_{\sigma}$ of type $A$ and the basis of $\Delta_{\Theta}$ contains at least two projections of simple roots in $\Delta$: $\alpha$ and $\beta$. Let us assume that the standard module is $I_P^G(\tau_{s{\tilde{\alpha}}})$ such that $\tau$ is a discrete series of $M$ and $\Delta_M = \Delta - \{\alpha\}$.
The proof of the Generalized Injectivity conjecture for $\Sigma_{\sigma}$ of type $A$ (see \cite{these}) carries over this context if the Levi $M'$ given there is such that $\Delta_{M'} = \Delta- \{\beta\}$ and one should pay attention to the choice of (order of simple roots in the) basis $\Delta_{\Theta}$ to insure that the parameter $\nu'$ for the root system $\Sigma_{\sigma}^{M'}$ splits into two residual segments appropriately (hence also an appropriate choice of $M$ determining $\Sigma_{\sigma}^M$).
Let us simply recall that from the Lemmas \ref{zelevinsky} and  \ref{lambdalambda'}, we know that if there is an embedding of the irreducible generic subquotient $I_{P'}^G(\tau'_{\underline{s'}})$ into $I_{P_1'}^G(\sigma'_{\lambda'})$, the parameter $\lambda'$ is in the $W_{\sigma}$-orbit of $\lambda$, hence $M_1'= w.M_1= M_1$ and $\sigma'= w.\sigma= \sigma$ since $w\in W_{\sigma}$.
\item Under the assumption that $W(M_1)$ equals $W_{\sigma}$ or $W(M_1)= W_{\sigma}\cup\left\{s_{\beta_d}W_{\sigma} \right\}$ (see Corollary \ref{mmm}) and $s\beta_d \lambda = \lambda$, the cases where $\Sigma_{\sigma}$ is irreducible of type $D_d$ in $\Sigma_{\Theta}$ can be dealt with the methods proposed in this work.
\end{enumerate}

It follows:
\begin{prop}
Let $G$ be a quasi-split reductive group of exceptional type, $\Sigma$ its root system, and $\Delta$ a basis of $\Sigma$. Let $P$ be a standard parabolic subgroup $P=MU$ of $G$.

Let us consider $I_{P}^G(\tau_{\underline{s}})$ with $\tau$ an irreducible discrete series of $M$, $\underline{s} \in (a_M^*)^+$. Let $\sigma$ be a unitary cuspidal representation of $M_1$ in the cuspidal support of $\tau$ and assume $\Sigma_{\sigma}$ (defined with respect to $G$) is of type $A$ and irreducible of rank $d=rk_{ss}(G)-rk_{ss}(M_1)$. Further assume that $\Delta_{\sigma}$ contains at least two restrictions of simple roots in $\Delta$.
 
Then, the unique irreducible generic subquotient of $I_{P}^G(\tau_{\underline{s}})$ is a subrepresentation.
\end{prop}

\subsection{Generalized Injectivity in $G_2$}
\begin{thm}
Let $G$ be of type $G_2$. Let $\pi_0$ be the unique irreducible generic subquotient of a standard module $I_P^G(\tau_s)$, then it is a subrepresentation.
\end{thm}

We follow the parametrization of the root system of $G_2$ as in Mui\'c \cite{muicreg2}: $\alpha$ is the short root and $\beta$ the long root. We have $M_{\alpha} \cong GL_2$, $M_{\beta} \cong GL_2$. Without loss of generality, let us assume $\tau$ is a discrete series representation of $M=M_{\alpha}$, the reasoning is the same for $M_{\beta}$.
As $\tau$ is a discrete series for $GL_2$, $\tau = St_2$.
$$\tau \hookrightarrow I_B^{M_{\alpha}}(|.|^{1/2}|.|^{-1/2})$$
We twist $\tau$ with $s\tilde{\alpha}$; $\tau_{s\tilde{\alpha}} \hookrightarrow I_B^{M_{\alpha}}(|.|^{1/2}|.|^{-1/2})\otimes|.|^s$
$$I_P^G(\tau_{s\tilde{\alpha}}) \hookrightarrow I_B^G(|.|^{s+1/2}|.|^{s-1/2})$$

Conjecturally for two values of $s$ (since there are only two weighted Dynkin diagrams conjecturally in bijection with dominant residual points) we obtain a dominant residual point of type $G_2$.
Since they are dominant residual points, the unique generic subquotient in $I_B^G(|.|^{s+1/2}|.|^{s-1/2})$ is a subrepresentation, and therefore appears as subrepresentation in $I_P^G(\tau_{s\tilde{\alpha}})$.

Suppose the value of $s$ is such that $(s+1/2,s-1/2)$ is not a dominant residual point.
The set up considered is that of $St_2 \hookrightarrow I_B^{M}(|.|^{1/2}|.|^{-1/2})$ twisted by $|.|^{s}$ so that it embeds in $I_B^M(|.|^{s+1/2}|.|^{s-1/2})$. Then, $I_P^G(St_2|.|^s) \hookrightarrow I_B^{G}(|.|^{s+1/2}|.|^{s-1/2})$. 
Using the result of Casselman-Shahidi (generalized injectivity conjecture for cuspidal inducing data) it is clear that the generic irreducible subquotient in $I_B^{G}(|.|^{s+1/2}|.|^{s-1/2})$ embeds as a subrepresentation.


\subsubsection{The case of a non-discrete series induced representation}

We now consider the general case of a standard module, with $\tau$ a tempered representation of $M \cong GL_2$.
As an irreducible tempered representation of $GL_2$, $\tau \cong I_B^{GL_2}(\textbf{1}\otimes\textbf{1})$.
Then the standard module is $I_P^G(\tau_s) \cong I_P^G(I_B^{GL_2}(\textbf{1}\otimes\textbf{1})\otimes |.|^s)\cong I_B^{G}((\textbf{1}\otimes\textbf{1})\otimes |.|^s) = I_B^{G}(|.|^s|.|^s)$ . Since $I_B^{G}(\textbf{1}\otimes\textbf{1})$ is unitary, its unique generic subquotient is itself; and there is nothing to prove. 
\subsubsection{Residual segments}
As an aside, we compute the residual segments of type $G_2$ here.
The weighted Dynkin diagrams for $G_2$ are: \\
\hbox{\hspace{7cm} \cercle\noteN{\alpha}\noteS{2}\thhh\fgauche\cercle\noteS{2}; ~~~~~ \cercle\noteN{\alpha}\noteS{0}\thhh\fgauche\cercle\noteS{2}} \\

\vskip0.2cm
\small Let $\lambda=(\lambda_1,\lambda_2)$ means that $\lambda=\lambda_1 (2\alpha+\beta)+\lambda_2(\alpha+\beta)$. On the other hand, it is known that

\begin{equation} \label{rel}
<2\alpha+\beta,\alpha^\vee>=1 \, \, , <\alpha+\beta, \alpha^\vee>=-1 \,\, , 
<2\alpha+\beta,\beta^\vee>=0 \, \, , <\alpha+\beta,\beta^\vee>=1
\end{equation}

From the first weighted Dynkin diagram above, the parameter $\lambda$ satisfies:
$$<\lambda,\alpha^\vee>=1, <\lambda,\beta^\vee>=1 $$
From the above relations \ref{rel}, one should be able to compute that the residual segment is
$\lambda=(2,1)$.

In the second weighted Dynkin diagram, the parameter $\lambda$ satisfies:
$$<\lambda, \alpha^\vee>=0, \,\, <\lambda,\beta^\vee>=1$$
And using the above relations \ref{rel}, we conclude that the residual segment is (1,1).
\vspace{0.3cm}

\appendix
\begin{center}
\textbf{APPENDIX}
\end{center}

\section*{Bala-Carter theory} \label{BCT}

This short appendix is written with considerably more details in the author's PhD thesis \cite{these}.

Let $\mathcal{N}= \mathcal{N}_\mathfrak{g}$ be the cone of nilpotent elements in $\mathfrak{g}$. This cone is the disjoint union of a finite number of $G$-orbits. Let $\mathcal{O}$ be a nilpotent orbit in $G\backslash \mathcal{N}$ and let $x \in  \mathcal{O}$ be a representative element. A theorem of 
Jacobson-Morozov extends $x$ to a standard ($\mathfrak{sl}_2$) triple $\left\{e,h,f\right\} \in \mathfrak{g}$, where $h$ can be chosen to lie in the 
fundamental 
dominant Weyl chamber :
$$\left\{h' \in \mathfrak{g} | \text{Re}(\alpha(h')) \geq 0, \forall \alpha \in \Delta ~ ~ \text{and whenever} ~ \text{Re}(\alpha(h')) =0, \text{Im}
(\alpha(h')) \geq 0 
\right\}$$

\begin{thm}[Kostant,\cite{kostant}]
Let $\Delta= \left\{\alpha_1,\ldots, \alpha_n\right\}$. A nilpotent orbit $\mathcal{O}$ is completely determined by the values $[\alpha_1(h), 
\alpha_2(h), 
\ldots, \alpha_n(h)]$.
\end{thm}
For every simple root $\alpha$ in $\Delta$, we have $\prodscal{\alpha}{h} \in \left\{0,1,2 \right\}$ (see section 3.5 in \cite{collingwood}).

If we label every node of the Dynkin diagram of $\mathfrak{g}$ with the eigenvalues $\alpha(h)=\prodscal{\alpha}{h}$ of $h$ on the 
corresponding simple 
root space $\mathfrak{g}_{\alpha}$, then all labels are 0,1 or 2. We call such a labeled Dynkin diagram, \textbf{a weighted Dynkin diagram}.

\subsection*{Weighted Dynkin diagrams}\label{WDD1}
The diagrams presented here are also presented in Carter's book \cite{carter}, page 175.

\subsubsection*{$A_d$}
$\cercle\noteN{\alpha_1}\noteS{2} \th\cercle\noteN{\alpha_2}\noteS{2}\th\points\points\points \th\cercle\noteN{\alpha_d}\noteS{2}$

\smallskip 
\subsubsection*{$C_d$}

$\underbrace{\cercle\noteN{\alpha_1}\noteS{2} \th\cercle\noteN{\alpha_2}\noteS{2}\points\th\cercle\noteS{2}}_{m}
\th\underbrace{\cercle\noteS{2}\th\cercle\noteS{0}}_{p_1}\th 
\cercle\noteS{2}\th\cercle\noteS{0}\points\cercle\noteS{0}\th 
\underbrace{\cercle\noteS{2}\th\cercle\noteS{0}\th\points\cercle\noteS{0}}_{p_k}
\thh\fgauche\cercle\noteN{\alpha_d}\noteS{2}$

\vskip0.3cm
\par\noindent
with $m+p_1+ \ldots p_k+1=d$, $p_1=2, p_{i+1} = p_i$ or $p_i + 1$ for each $i$. 
($k=0$, $m=l-1$ is a special case)

\subsubsection*{$B_d$}

\smallskip

$\underbrace{\cercle\noteN{\alpha_1}\noteS{2} \th\cercle\noteN{\alpha_2}\noteS{2}\points \cercle\noteS{2}}_{m}\th\underbrace{\cercle
\noteS{2}\th
\cercle\noteS{0}}_{p_1}\th \cercle\noteS{2}\th\cercle\noteS{0}\points\th\cercle\noteS{0} \th\underbrace{\cercle\noteS{2}\th\cercle\noteS{0}\th
\points\cercle
\noteS{0}}_{p_k}\thh\fdroite\cercle\noteN{\alpha_d}\noteS{0}$

\smallskip
\par\noindent
with $m+p_1+ \ldots p_k=d$, $p_1=2, p_{i+1} = p_i$ or $p_i + 1$ for $i=1,2, \ldots, k-2$ and 
$$p_k= \left\{
    \begin{array}{ll}
        \frac{p_{k-1}}{2} ~\mbox{if} ~ p_{k-1} ~\mbox{is even} \\
        \frac{p_{k-1} -1}{2} ~\mbox{if} ~ p_{k-1}  ~\mbox{is odd}
    \end{array}
\right.$$

\vskip0.3cm

\par\noindent
In addition the diagram:
\smallskip

\hbox{\cercle\noteN{\alpha_1}\noteS{2} \th\cercle\noteN{\alpha_2}\noteS{2}\points \th \cercle\noteS{2}\points \th \cercle\noteS{2}\th \cercle
\noteS{2}\th\cercle
\noteS{2} \points \cercle\noteS{2}\th\cercle\noteS{2}\th\points\cercle\noteS{2}\thh\fdroite\cercle\noteS{2}}

\smallskip
\par\noindent
is distinguished.

\subsubsection*{$D_d$}

$\underbrace{\cercle\noteN{\alpha_1}\noteS{2} \th\cercle\noteN{\alpha_2}\noteS{2}\points\cercle\noteS{2}}_{m}\th\underbrace{\cercle
\noteS{2}\th\cercle
\noteS{0}\th\cercle\noteS{2}\th\cercle\noteS{0}\points\cercle\noteS{2}\th\cercle\noteS{0}}_{2k}\fourche\montepeu{\noteE 
2}\descendpeu{\noteE 2}$

\smallskip

with $m+2k+2=d$, and those of the form
\par\noindent
$\underbrace{\cercle\noteN{\alpha_1}\noteS{2} \th\cercle\noteN{\alpha_2}\noteS{2}\points\cercle\noteS{2}}_{m}\th \underbrace{\cercle
\noteS{2}\th 
\cercle\noteS{0}}_{p_1}\th \cercle\noteS{2}\th\cercle\noteS{0}\points \cercle\noteS{0}\th\underbrace{\cercle\noteS{2}\cercle\noteS{0}\points 
\cercle
\noteS{0}\fourche\montepeu{\noteE 2}\descendpeu{\noteE 2}}_{p_k}$

\vskip0.3cm
\par\noindent
with $m+p_1+ \ldots p_k=l$, $p_1=2, p_{i+1} = p_i$ or $p_i + 1$ for $i=1,2, \ldots, k-2$ and 
$$p_k= \left\{
    \begin{array}{ll}
        \frac{p_{k-1}}{2} ~\mbox{if} ~ p_{k-1} ~\mbox{is even} \\
        \frac{p_{k-1} +1}{2} ~\mbox{if} ~ p_{k-1}  ~\mbox{is odd}
    \end{array}
\right.$$
%

%
%
%
%

The key notion used by Bala and Carter was the notion of distinguished nilpotent element. It is an element that is not contained in any proper 
Levi 
subalgebra.
Alternatively, a nilpotent element $n \in \mathfrak{g}$ is called distinguished if it does not commute with any non-zero semi-simple element of $\mathfrak{g}$.
Or also, a nilpotent element $X$ (resp. orbit $\mathcal{O}_X$) is distinguished if the only Levi subalgebra containing $X$ (resp. meeting $\mathcal{O}_X$) is  $\mathfrak{g}$ itself.
%
%
%
%
%
%
\begin{dfn}[distinguished parabolic subalgebra]
A parabolic subalgebra $\mathfrak{p} = \mathfrak{l} + \mathfrak{u}$ of $\mathfrak{g}$ is called distinguished if $\dim\mathfrak{l} = \mathfrak{u}
\slash 
[\mathfrak{u},\mathfrak{u}]$, in which $\mathfrak{p} = \mathfrak{l}\oplus\mathfrak{u}$ is a Levi decomposition of $\mathfrak{p}$, with Levi part $\mathfrak{l}$.
\end{dfn}

%
%
%

The Theorem 5.9.5 in \cite{carter} implies the following correspondence:
\begin{equation}
\label{eq:distpara2}
\left\{\parbox{120pt}{Nilpotent Ad($G$)-orbits of $\mathfrak{g}$}\right\} \leftrightarrow 
\left\{\parbox{120pt}{$G$ conjugacy classes of pairs $(\mathfrak{p},\mathfrak{m})$ of $\mathfrak{g}$}\right\} 
\end{equation}
in which $\mathfrak{m}$ is a Levi factor, $\mathfrak{p}\subseteq \mathfrak{m'}$ 
is a distinguished parabolic subalgebra of the semi-simple part of $\mathfrak{m}$.

Let us give a few more results on distinguished orbits, in particular the Theorem \ref{8.2.14} explains the partitions used in \ref{SOJ}:

We need to introduce a grading: given a non-zero nilpotent element in $\mathfrak{g}$, using the standard triple above, the Jacobson-Morozov Lie algebra homomorphism $\phi: \mathfrak{sl}_2 \rightarrow \mathfrak{g}$ satisfies $\phi(e) = n \in \mathfrak{n}$ and $\phi(h) = \gamma$ is in the dominant chamber of $\mathfrak{t}$.

The adjoint action of $\mathfrak{t}$ on $\mathfrak{g}$ yields a grading $\mathfrak{g} = \oplus_{i\in \mathbb{Z}}\mathfrak{g}(i)$
in which $$\mathfrak{g}(i) = \left\{x \in \mathfrak{g}|ad(\gamma)(x) = ix \right\}; ~~ [\mathfrak{g}(i) ,\mathfrak{g}(j)] \subseteq \mathfrak{g}(i+j)$$ and 
$n \in 
\mathfrak{g}(2)$.
Further, set 

\begin{equation}\label{2.54}
\left\{\begin{aligned}
&\mathfrak{p}=  \mathfrak{p}(\gamma) = \oplus_{i\geq 0}\mathfrak{g}(i) \\
 &\mathfrak{u} = \oplus_{i> 0}\mathfrak{g}(i) \\
  &\mathfrak{l} = \mathfrak{g}(0)
  \end{aligned}\right.
\end{equation}

The Lie subalgebra $\mathfrak{p}$ contains $\mathfrak{b}$, and is thus a parabolic subalgebra whose Levi decomposition is  $\mathfrak{p} = 
\mathfrak{u}
\oplus \mathfrak{l}$.

On the other hand, starting with a subset $J \subseteq \Delta$, and denoting $\mathfrak{p}_J$ the standard parabolic subalgebra, one defines a function $\eta_J : \Phi_0 \rightarrow \mathbb{Z}$, defined on roots of $\Delta$ as twice the indicator function of $J$ and extended linearly to all roots. 

We obtain a grading: $\mathfrak{g} = \oplus_{i \geq 0}\mathfrak{g}_J(i)$ by declaring
$\mathfrak{g}_J(0) = \mathfrak{t}\oplus\sum_{\eta_J(\alpha)=0}\mathfrak{g}_{\alpha}$ and otherwise $\mathfrak{g}_J(i) = 
\sum_{\eta_J(\alpha)=i}\mathfrak{g}
_{\alpha}$.
Then, $\mathfrak{p}_J =\oplus_{i\geq 0}\mathfrak{g}_J(i)$ and its nilpotent radical is $\mathfrak{n}_J =\oplus_{i> 0}\mathfrak{g}_J(i)$.

To summarize, to the standard triple containing $n$ one attaches a parabolic subalgebra $\mathfrak{q}$ of $\mathfrak{g}$ with Levi decomposition $\mathfrak{q} = \mathfrak{l}\oplus \mathfrak{u}$.

If $\dim\mathfrak{g}(1) = 0$, then we call $n$ (resp. $\mathcal{O}_n$) an even nilpotent element (even nilpotent orbit, respectively).
\begin{prop}[Corollary 3.8.8 in \cite{collingwood}]\label{g(1)=0}
A weighted Dynkin diagram has labels only 0 or 2 if and only if it corresponds to an even nilpotent orbit (i.e, if $\dim\mathfrak{g}(1) = 0$) 
\end{prop}

The two following propositions are taken from Chapter 2 of Di Martino's thesis \cite{marcelo}:
\begin{prop}\label{pJ}
The standard parabolic subalgebra $\mathfrak{p}_J$ is distinguished if and only if $\dim\mathfrak{g}_J(0) = \dim\mathfrak{g}_J(2)$. In this 
case, if $n$ is any 
element in the unique open orbit of the parabolic subgroup $P_J$ on its nilpotent radical $\mathfrak{n}_J$, then the parabolic subalgebra 
associated to $n$ 
as in (\ref{2.54}) equals $\mathfrak{p}_J$.
\end{prop}

A distinguished nilpotent element also satisfies the following:
\begin{prop} \label{n}
A nilpotent element $n \in \mathfrak{g}$ is distinguished if and only if $\dim\mathfrak{g}(0) = \dim\mathfrak{g}(2)$. Moreover, if $n \in 
\mathfrak{g}$ is 
distinguished, then $\dim\mathfrak{g}(1)=0$.
\end{prop}

\begin{thm}[Theorem 8.2.3 in \cite{collingwood}]
Any distinguished orbit in $\mathfrak{g}$ is even.
\end{thm}

\begin{thm}[Theorem 8.2.14 in \cite{collingwood}]\label{8.2.14}~\par
\begin{enumerate}
	\item If $\mathfrak{g}$ is of type A, then the only distinguished orbit is principal.
	\item If $\mathfrak{g}$ is of type B, C or D, then an orbit is distinguished if and only if its partition has no repeated parts. Thus the 
partition of a 
distinguished orbit in types B, D has only odd parts, each occurring once, while the partition of a distinguished orbit in type C has only even 
parts, each 
occurring once.
\end{enumerate}
\end{thm}
\subsection*{Distinguished Nilpotent orbits and residual points}
The connection with the notion of residual point is now made accessible. 

Let $G$ be a Chevalley (semi-simple) group and $T \subseteq B$ a maximal split torus and a Borel subgroup. We have a root datum $
\mathcal{R}(G,B,T)$.
By reversing the role of $X^*(T)$ and $X_*(T)$, we obtain a new root datum $\mathcal{R}^{\vee}=(X_*(T), \Delta, X^*(T), \Delta^{\vee})$. 
Let $(~^L\!G, ~^L\!B, ~^L\!T)$ be the triple with root datum $\mathcal{R}^{\vee}$. The L-group $~^L\!G$ is the dual group, with maximal torus $~^L\!T$, and 
Borel subgroup $~^L\!B
$. Denote the respective Lie algebra $~^L\!\mathfrak{g}, ~^L\!\mathfrak{t}$ and $~^L\!\mathfrak{b}$.
Let $(V^*, \prodscal{}{})$ be a finite dimensional Euclidean space containing and spanned by the root system: $\Delta \subseteq V^*$, the 
canonical 
pairing between $V$ and $V^*$ is denoted by $\prodscal{}{}$. We fix an inner product on $V$ by transport of structure from  $(V^*, \prodscal{}
{})$ via the 
canonical isomorphism $V^* \rightarrow V$ associated with $\prodscal{}{}$. Thus this map becomes an isometry, and for each $\alpha \in 
\Delta$, the 
coroot $\check{\alpha} \in V$ is given as the image of $2\prodscal{\alpha}{\alpha}^{-1}\alpha \in V^*$. 

To this data we associate the Weyl group $W_0$ generated by the reflexions $s_{\alpha} ~~ (s_{\alpha}(x) = x -\prodscal{x}{\check{\alpha}}
\alpha ~ 
\mbox{and} ~ s_{\alpha}(y) = y -\prodscal{\alpha}{y}\check{\alpha})$ over the hyperplanes $H_{\alpha} \subseteq  V^*$ consisting of elements 
$x \in V^*$ 
which are orthogonal to $\check{\alpha}$ with respect to $\prodscal{}{}$.

Let us make a remark before stating the correspondence result related to our use in this manuscript:

\begin{rmk}
The bijective correspondence (below) is originally formulated for residual subspaces.
Let $k$ be the \enquote{coupling parameter} as defined in \cite{opdamheckman}.
An affine subspace $L \subseteq V$ is called residual if, for a root system $\Phi$ (in a root datum)
$$\#\left\{\alpha \in \Phi \vert \prodscal{\alpha}{L} = k \right\} = \#\left\{\alpha \in \Phi \vert \prodscal{\alpha}{L}= 0 \right\} + \hbox{codim} L$$
(If $\mathcal{R}$ is semi-simple, there exist residual subspaces which are singletons $\left\{\lambda\right\} \subseteq V$, the residual points).

For example, when the parameter $k$ (called \enquote{coupling parameter} in \cite{opdamheckman}) equals 1, the Weyl vector $\rho = 
\frac{1}
{2}\sum_{\alpha \in \Phi}\alpha$ is a residual point, since the above equation is verified. More generally, for any $k=(k_{\alpha})_{\alpha \in 
\Phi}$, the vector 
$\rho(k)= \frac{1}{2}\sum_{\alpha \in \Phi}k_{\alpha}\alpha$ is a residual point.

Then the bijective correspondence is given between the set of nilpotent orbits in the Langlands dual Lie algebra $~^L\!\mathfrak{g}$ and the 
set of $W_0$-
orbits of residual subspaces.
\end{rmk}

We mention the following result partially related to Proposition \ref{1.13bis}.
The bijective correspondence concerns only unramified characters and we fix the parameter $k_{\alpha} =1$ for all $\alpha \in \Phi_0$.

\begin{prop} \label{resbij}
There is a bijective correspondence $\mathcal{O}_{W_0\lambda(\mathcal{O})} \leftrightarrow W_0\lambda(\mathcal{O})$ between the set of 
distinguished 
nilpotent orbits in the Langlands dual Lie algebra $~^L\!\mathfrak{g}$ and the set of $W_0$-orbits of residual points.
\end{prop}

\begin{proof}
This particular bijection is a specific case of the larger bijective correspondence given between the set of nilpotent orbits in the Langlands dual 
Lie algebra 
$~^L\!\mathfrak{g}$ and the set of $W_0$-orbits of residual subspaces. It is discussed in details in [\cite{opdamspec}, Appendices A and B], 
but also in 
[\cite{heiermannorbit}, Proposition 6.2].
\end{proof}

Let $(~^L\!\mathfrak{m}, ~^L\!\mathfrak{p})$ be a representative of a class, for which 
$~^L\!\mathfrak{m} =~^L\!\mathfrak{g}$ and $~^L\!\mathfrak{p} \subseteq ~^L\!\mathfrak{g}$ is a standard distinguished parabolic subalgebra. We have a 
corresponding 
distinguished nilpotent orbit $\mathcal{O}$. With Proposition \ref{pJ}, the data $~^L\!\mathfrak{p}$ is equivalent to the assignment of an even 
weighted 
Dynkin diagram: $2\lambda(\mathcal{O})$.

Since we have $\dim\mathfrak{g}(0) = \dim\mathfrak{h} + \#\left\{\alpha \in \Phi|\prodscal{\check{\alpha}}{2\lambda(\mathcal{O})} = 0 \right\}$ and $\dim \mathfrak{g}(2) =  \#\left\{\alpha \in \Phi| \prodscal{\check{\alpha}}{2\lambda(\mathcal{O})}=2 \right\}$, the assignment of an even weighted Dynkin diagram implies $\dim\mathfrak{g}(0)= \dim\mathfrak{g}(2)$ and this equality sets $
\lambda(\mathcal{O})$ as a residual point.

The definition of $\lambda(\mathcal{O})$ depends on the choice of positive roots and Borel subgroup 
$~^L\!B$. A different choice yields a different element on the same $W_0$-orbit. See \cite{opdamspec}, Appendices A and B, and particularly Proposition 8.1 in \cite{opdamspec}.

%
%
%
%
%

\normalsize
\section*{Projections of roots systems} \label{lab}
Let us first follow the notations of the book of \cite{renard}, Chapter V. We will also use the notations of Section \ref{preliminaries}. Let $X^*(G)$ denote the group of rational characters of $G$; its dual is $X_*(G)$. We denote $a_0 = X_*(A_0)\otimes_{\Z}\R$ and $a_0^* = X^*(A_0)\otimes_{\Z}\R$.

The duality between $X^*(A_0)$ and $X_*(A_0)$ extends to a duality (canonical pairing) between the vector spaces $a_0$ and $a_0^*$ (see the Chapter V of \cite{renard}, or the author's PhD thesis).

Because of the existence of the scalar product (sustaining the duality), the restriction map from $(a_0^G)^*$ to $(a_{\Theta}^G)^*$ is a \emph{projection} map from $(a_0^G)$ to $(a_{\Theta}^G)$. 
With the notations of the Section \ref{cond}, the roots in $\Delta(P_1)$ generating $(a_{M_1})^*$ are non -trivial restrictions of roots in $\Delta \setminus \Delta^{M_1}$ (recall that in the notations of \cite{MW}, I.1.6, $\Delta^{M_1}$ are the roots of $\Delta$ which are in $M_1$), and $(a_{M_1})$ is generated by the projection of roots in $\Delta^{\vee} \setminus \Delta^{M_1~\vee}$.
In the article \cite{PrSR}, we have rather considered projections of roots. We have studied the set $\Sigma_{\Theta}$, projections of roots onto the orthogonal to $\Theta$. Let us denote $d$ the dimension of $a_{\Theta}$, i.e the cardinal of $\Delta - \Theta$. 
\begin{thm}[see \cite{PrSR}] \label{mainlab}
Let $\Sigma$ be an irreducible root system of classical type (i.e of type $A,B,C$ or $D$). The subsystems in $\Sigma_{\Theta}$ are necessarily of classical type. In addition, if the irreducible (connected) components of $\Theta$ of type $A$ are all of the same length and the interval between each of them of length one, then $\Sigma_{\Theta}$ contains an irreducible root system of rank $d$ (non necessarily reduced).
\end{thm}
We have used the following observation, from \cite[Equation (10) in VI.3, Proposition 12 in VI.4, Chapter VI]{bourbaki}: Let $\alpha$ and $\beta$ be two non-orthogonal elements of a root system.
Set $$C=\left(\frac{1}{\cos({\alpha},{\beta})}\right)^2\qquad\hbox{and}\qquad R=\frac{||{\alpha}||^2}{||{\beta}||^2}\,\,.$$ Thus, if $||\alpha||\ge||\beta||$
$$ \frac{C}{R}\in\{2^2,1,(2/3)^2\}\qquad\hbox{and}\qquad CR=4\,\,.
$$

\subsection*{The case of reducible $\Sigma_{\Theta}$}

In \cite{PrSR} have seen that in order to obtain a projected root system irreducible and of rank $d$, we had to impose several constraints. Let us explain once more some of them. Let us first consider two components $A_m$ and $A_q$ of (the Dynkin diagram of) $\Theta$, let $e_r$ and $e_s$ be the vectors in the basis vectors of smallest index such $\Xi_r= \left\{e_r, \ldots, e_{r+m} \right\}$ corresponds to $A_m$ and $\Xi_s$ to $A_q$. Let us assume two simple consecutive roots $\alpha_{k-1}$ and $\alpha_k$ are outside of $\Theta$ and $k= r+m+1 = s-1$. Then $\Xi_k = \left\{ e_k \right\}$.
Let us consider the projections of $\alpha_{k-1}$ and $\alpha_{k}$:
Since $e_k$ is orthogonal to all roots in $\Theta$, $\overline{e_k} = e_k$. 
Therefore:
$||\overline{\alpha_{k-1}}||^2= ||\overline{e_{k-1}}- \overline{e_k}||^2 = \frac{1}{m+1}+1\,\, , ||\overline{\alpha_k}||^2= ||\overline{e_{k}}- \overline{e_{k+1}}||^2 = 1+ \frac{1}{q+1} \,\,.$

Then $C=\left(\frac{1}{\cos(\overline{\alpha_{k-1}},\overline{\alpha_k})}\right)^2= (\frac{1}{m+1}+1)(1+ \frac{1}{q+1})$
and if we assume $||\overline{\alpha_{k-1}}||\ge||\overline{\alpha_k}||$ i.e  $m\ge q$, we have:
$$R=\frac{||\overline{\alpha_{k-1}}||^2}{||\overline{\alpha_k}||^2}=\frac{\frac{1}{m+1}+1}{\left(1+ \frac{1}{q+1}\right)}$$
 
If $\alpha_k$ and $\alpha_{k-1}$ were to be part of a root system, we would need $$ \frac{C}{R}=\left(1+\frac{1}{m+1}\right)^2\in\{2^2,1,(2/3)^2\}\qquad\hbox{and}\qquad
CR=\frac{1}{\left(1+\frac{1}{q+1}\right)}^2=4
\,\,.$$
This implies $m=0$ and $\left(1+\frac{1}{q+1}\right) = 1/4$ a contradiction.
This illustrates the fact that in the main theorem (Theorem \ref{mainlab}) the intervals between the irreducible connected components of $\Theta$ need to be of length one, and \emph{at most one}.

In general, the complement of Theorem \ref{mainlab} above is the following:

\begin{thm}
Let $\Sigma$ be an irreducible root system of type $B,C$ or $D$. 
If the irreducible (connected) components of $\Theta$ of type $A$ are all \emph{not} of the same length and the interval between each of them of length one, then $\Sigma_{\Theta}$ contains a reducible root system of rank $d$ (non necessarily reduced); $\Sigma_{\Theta} = \bigoplus_i\Sigma_{\Theta,i}$ and if $d_i$ is the rank of the irreducible i-th component, then $\sum_i d_i = d$.

The number of irreducible components ($\Sigma_{\Theta,i}$) is as many as there are \emph{changes of length} plus one. That is, if there are $d_1$ components of type $A_{m_1}$, followed by $d_2$ components of type $A_{m_2}$, et cetera until $d_s$ components of type $A_{m_s}$, such that $m_i \neq m_{i+1}$ for any $i$, and one last component of type $B$ or $C$ or $D$, there are $s-1$ changes in the length ($m_i$) and therefore $s$ irreducible connected components in $\Sigma_{\Theta}$.
The set $\Sigma_{\Theta}$ is composed of irreducible components of type $A$ and possibly one component of type $B, C$ or $D$.
\end{thm}

\begin{proof}
We have explained the condition on the interval being of at most length one in the paragraph preceding the statement of the theorem. We do not repeat here the methods of proof for the case of $\Sigma_{\Theta}$ irreducible which apply here: in particular the treatment of the case $e_n \notin \Theta$, the reduction to this case's argumentation when $e_n \in \Theta$, and the argumentation showing that the components of type $A$ of $\Theta$ should be of the same length to obtain a root system in the projection. 
\\

We consider the case of root system of type $B, C, D$.
Let us then assume that we have $d_1$ components of type $A_{m_1}$ in $\Theta$, by the argumentation given in \cite{PrSR}, we obtain a root system of type $BC_{d_1}$. Let us assume that these  $d_1$ components of type $A_{m_1}$ are followed by  $d_2$ components of type $A_{m_2}$, $m_2 \neq m_1$. Let us denote $e_{1,d_1}$  the vector associated to the last component of type $A_{m_1}$ and $e_{2,1}$  the vector associated to the first component of type $A_{m_2}$. 

The projection $\overline{e_{1,d_1}} -\overline{e_{2,1}}= \frac{e_{(d_1-1)m_1+1}+e_{(d_1-1)m_1+2}+ \ldots +e_{(d_1-1)m_1+m_1}}{m_1+1} - \frac{e_{d_1m_1+1}+e_{d_1m_1+2}+ \ldots +e_{d_1m_1+m_2}}{m_2+1}$ of $e_{1,d_1} - e_{2,1}$ 
cannot be a root in $\Sigma_{\Theta}$ (it would contradict the conditions of validity of the value $C$ and $R$ when calculated with respect to the last root of the previously considered $BC_{d_1}$).

However, the projections of the roots corresponding to the intervals between any two of the $d_2$ components of type $A_{m_2}$ (say of $\overline{e_s} - \overline{e_t}$) along with all roots of the form $\pm e_s$ or $\pm e_t$ (resp. $\pm 2e_s$ or $\pm 2e_t$) form a root system of type $BC_{d_2}$. Some specificities, such as root system of type $C$ appearing in the projection for certain cases under $\Sigma$ of type $C$ or $D$ carry over here (see \cite{PrSR}).

The key mechanism assuring that the sum of the $d_i$ equals $d$ is the observation that one need \emph{three} consecutive components of type $A_q$ of a given length $q$ (followed by components of length $m \neq q$) to obtain in the projection a $BC_3$ \emph{(hence of rank three!)} whereas one would obtain only a $A_2$ type of root system. This means that even if the root connecting the $A_q$ to $A_m$ is not a root in the projection, i.e \enquote{we are missing a simple root}; we get a simple root of type $\overline{e_i}$ or $\overline{2e_i}$.

One may notice that another possibility would be to obtain a reducible root system such as $A_1 \times A_1 \times \ldots \times A_1$. This case is not excluded but it would not be possible to find such a system of \emph{maximal rank}.

Indeed, let us briefly recall the formulas written for the case of $\Sigma$ of type $A$ in \cite{PrSR}, where we consider three vectors $e_r, e_s$ and $e_t$ whose projections are associated to three components of $\Theta$ of type $A_m, A_p$ and $A_q$: Let $\alpha =e_i - e_j$ be a root whose projection is $\overline{\alpha}=\pm(\overline{e_r}-\overline{e_s})\,\,$ and $\beta=e_k-e_l$ a root whose projection is $\overline{\beta}=\pm(\overline{e_s}-\overline{e_t})$, then the square of the scalar product of $\overline{\alpha}$ and  $\overline{\beta}$ is $$\left(<\overline{\alpha},\overline{\beta}>\right)^2=\frac{1}{(p+1)^2}\,\,.$$

This excludes the possibility of $\alpha$ and $\beta$ being orthogonal. Therefore for two consecutive roots in the projection (projections of simple roots), it is not possible to obtain a system of type $A_1 \times A_1$.

If there is a sequence of connected consecutive components of $\Theta$ of type $A$ that we index by an integer $i$ (in increasing order) and length $q_i$ with $q_i \neq q_{i+1}$ for any $i$, let us denote $\overline{\alpha_i} = \overline{e_r} - \overline{e_s}$ where $e_r \in A_{q_i}$ and $e_s \in A_{q_{i+1}}$. 

Further, let us denote $\overline{\alpha_{i+2}} = \overline{e_t} - \overline{e_z}$ where $e_t \in A_{q_{i+2}}$ and $e_z \in A_{q_{i+3}}$. The orthogonal roots $\overline{\alpha_i}$ and $\overline{\alpha_{i+2}}$ form a root system of type $A_1 \times A_1$.

The root $\overline{\alpha_{i+1}} = \overline{e_s} - \overline{e_t}$ does not contribute to this subsystem.

Therefore, the maximal number of $A_1$ factor such that the reducible root system $A_1 \times A_1$ appear in $\Sigma_{\Theta}$ is $d \slash 2$.

By a similar reasoning, it would be possible to obtain a reducible system of type $A_2 \times A_2 \times \ldots \times A_2$ if $\Theta$ is composed of a succession of connected components of type $A$ such that the three first ones are of length $m$, the three next ones of length $q \neq m$, etc. Then the projection of the root connecting $A_m$ and $A_q$ would not contribute to this subsystem.
Again, this would never give any reducible system of \emph{maximal rank d}.

Because to any change of length of the $A$ components, the corresponding root (connecting the two components of different length) cannot appear as a (simple) root in the projection, we are missing a root (of the set $\Delta - \Theta$ of size $d$) at any change of length. In the case $\Sigma$ is of type $A$, this 'missing' root is not replaced by any short or long root ($e_i$ or $2e_i$), therefore it is impossible to obtain a basis of root system in the projection. In other words, there does not exist any reducible root system of maximal rank in the projection $\Sigma_{\Theta}$ of $\Sigma$ of type $A$.
\end{proof}

\vspace{0.2cm}

Let us illustrate one case of the previous theorem with a Dynkin diagram of $\Sigma$ of type $B$:
\begin{multline*}
\underbrace{\disque\noteNC{\alpha_1}\th\disque\th\points\th\disque \th\disque\noteNC{\alpha_{m_1}}}_{A_{m_1}} 
\th\cercle\th\underbrace{\disque\th\disque\th\points\th\disque  \th\disque}_{A_{m_1}} \th\cercle\th
\points\th\cercle\th \underbrace{\disque\th\disque\th\points\th\disque \th\disque\th\disque}_{A_{m_1}} \points \\
\th\cercle\th \underbrace{\disque\th\disque\th\points\th\disque \th\disque}_{A_{m_s}} \th\cercle\th
\points\th\cercle\th \underbrace{\disque\th\disque\th\points\th\disque 
\th\disque}_{A_{m_s}}\th\cercle\th
  \underbrace{\disque\th\disque\th\points\th
\disque \thh\fdroite\disque\noteNC{\alpha_n}}_{B_r}
\end{multline*}

%
%

\bibliographystyle{plain}
\bibliography{newgic}

\end{document}